\documentclass[1 leqno,11pt,psfig]{amsart}
\usepackage{amssymb, amsmath,amsmath,latexsym,amssymb,amsfonts,amsbsy, amsthm,mathtools,graphicx,CJKutf8,CJKnumb,CJKulem}
\usepackage{graphicx,epsfig}
\usepackage{graphicx}
\usepackage{amssymb}
\usepackage{graphicx}
\usepackage{graphicx,epsfig}
\usepackage{amsmath}
\usepackage{mathrsfs}
\usepackage{amssymb}
\usepackage[colorlinks=true,citecolor=blue,bookmarks=true]{hyperref}

\usepackage{amssymb,mathrsfs,amsfonts,amsmath,amsbsy,tikz}
\usepackage{fontenc}
\usepackage{textcomp}

\numberwithin{equation}{section}

\allowdisplaybreaks

\newcommand{\beq}{\begin{equation}}
\newcommand{\eeq}{\end{equation}}
\newcommand{\ben}{\begin{eqnarray}}
\newcommand{\een}{\end{eqnarray}}
\newcommand{\beno}{\begin{eqnarray*}}
\newcommand{\eeno}{\end{eqnarray*}}

\newtheorem{theorem}{Theorem}[section]
\newtheorem{definition}[theorem]{Definition}
\newtheorem{lemma}[theorem]{Lemma}
\newtheorem{proposition}[theorem]{Proposition}

\newtheorem{remark}[theorem]{Remark}
\newtheorem{Step}{Step}
\newtheorem{Theorem}{Theorem}[section]

\newtheorem{Corollary}[Theorem]{Corollary}

\begin{document}
\begin{CJK*}{UTF8}{gkai}
\title[Dynamical magneto-rotational instability]{Dynamical magneto-rotational instability}

\author{Zhiwu Lin}
\address{School of Mathematical Sciences, Fudan University,
 200433, Shanghai, P. R. China}
\email{zwlin@fudan.edu.cn}

\author{Yucong Wang}
\address{
School of Mathematics and Computational Science, Xiangtan University, 411105, Xiangtan, Hunan, P. R. China}
\email{yucongwang666@163.com}

\author{Wenpei Wu}
\address{
School of Mathematics and Computational Science, Xiangtan University, 411105, Xiangtan, Hunan, P. R. China}
\email{wenpeiwu16@163.com}

\date{\today}

\maketitle

\begin{abstract}
Magneto-rotational instability (MRI) is an important instability mechanism for
rotating flows with magnetic fields. In particular, when the strength of the
magnetic field tends to zero, the stability criterion for rotating flows is
generally different from the classical Rayleigh criterion for rotating flows
without a magnetic field. MRI has wide applications in astrophysics,
particularly to the turbulence and enhanced angular momentum transport in
accretion disks. For the case of vertical magnetic fields, we give rigorous
proof of linear MRI and a complete description of the spectra and semigroup
growth of the linearized operator. Moreover, we prove nonlinear stability and
instability from the sharp linear stability/instability criteria.

\end{abstract}


\section{Introduction}

Consider the incompressible ideal magneto-hydrodynamic (MHD) system
\begin{align}\label{EPM}
\begin{cases}
\partial_tv+v\cdot\nabla v+\nabla p=H\cdot\nabla H,\\
\partial_t H+v\cdot\nabla H=H\cdot\nabla v,\\
\text{div}v=\text{div} H=0.
\end{cases}
\end{align}
Here, the velocity $v(t,x)\in\mathbb{R}^{3}$, magnetic field $H(t,x)\in
\mathbb{R}^{3}$, scalar pressure $p(t,x)$, where $t\geq0,x\in\Omega
=\{x=(x_{1},x_{2},x_{3})\in\mathbb{R}^{3}|\ R_{1}\leq\sqrt{x_{1}^{2}+x_{2}%
^{2}}<R_{2},x_{3}\in\mathbb{T}_{2\pi}\}$, $0\leq R_{1}<R_{2}\leq\infty$. The
\textquotedblleft ideal\textquotedblright\ means the effects of viscosity and
electrical resistivity are neglected. We impose the initial condition
\begin{align}\label{1.2'}
(v, H)(0,x)=(v^0,H^0)(x), \quad x\in\Omega
\end{align}
and boundary conditions
\begin{align}\label{bdc}
\begin{cases}
 v\cdot\mathbf{n}|_{\partial\Omega}= H\cdot\mathbf{n}|_{\partial\Omega}=0,\\
 (v,H)(t,x)\to (v_0,H_0)(x) \quad as \quad |x|\to \infty, \quad \mbox{if}\quad R_2=\infty;\\
 v(t,x_1,x_2,x_3)=v(t,x_1,x_2,x_3+2\pi),\\ H(t,x_1,x_2,x_3)=H(t,x_1,x_2,x_3+2\pi),
\end{cases}
\end{align}
with the unit outward normal vector $\mathbf{n}=\left(  \frac{x_{1}}%
{\sqrt{x_{1}^{2}+x_{2}^{2}}},\frac{x_{2}}{\sqrt{x_{1}^{2}+x_{2}^{2}}%
},0\right)  $. In the following, we shall consider the axisymmetric solutions
of the system \eqref{EPM}. The incompressible ideal MHD equations \eqref{EPM},
which describe the macroscopic motion of an electrically conducting fluid,
have many steady solutions \cite{MR1216012,TT1988}. One simple axisymmetric
steady solution $(v_{0},H_{0})(x)$ is a rotating flow with a vertical magnetic
field, that is,
\begin{align}\label{steady1}
\begin{cases}
v_0(x)=\mathbf{v}_{0}(r)\mathbf{e}_\theta=r\omega(r)\mathbf{e}_\theta,\\
H_0(x)=\epsilon b(r)\mathbf{e}_z,\\
\end{cases}
\end{align}
where $\epsilon\neq0$ is a constant, $\left(  r,\theta,z\right)  $ are the
cylindrical coordinates with $r=\sqrt{x_{1}^{2}+x_{2}^{2}},\ z=x_{3}$,
$\left(  \mathbf{e}_{r},\mathbf{e}_{\theta},\mathbf{e}_{z}\right)  $ are unit
vectors along $r,\theta,z$ directions, $\omega(r)\in C^{3}(R_{1},R_{2})$ is
the angular velocity of the rotating fluid, the magnetic profile $b(r)\in
C^{3}(R_{1},R_{2})$ has a positive lower bound.

The study of stability of rotating flows (i.e., without a magnetic field) has a
long history, going back to Rayleigh \cite{Rayleigh1880} in 1880s, who showed
that a sufficient condition for stability is that the square of angular
momentum of the rotating fluid increases outwards, i.e.,
\begin{equation}
\partial_{r}(\omega^{2}r^{4})>0,\ \text{for all }r\in(R_{1},R_{2}).
\label{Rayleigh}%
\end{equation}
Rayleigh's stability criterion was shown to be sharp for axisymmetric
perturbations of 3D Euler equations \cite{DR2004,Synge1933}.

For 3D axisymmetric Euler equations, the angular momentum of each fluid
element is invariant along the fluid trajectory. This is important for proving
Rayleigh's stability criterion. The addition of magnetic fields implies that
the angular momentum of each fluid element is no longer conserved along the
trajectory. This difference suggests that the stability criterion for a
rotating flow with a magnetic field might be very different from the Euler
case. The influence of a vertical and uniform magnetic field (i.e.,
$b(r)=$constant) on the stability of the rotating flows was first studied by
Velikhov \cite{VEP1959} and Chandrasekhar \cite{CS1960Couette}, who derived a
sufficient condition for linear stability of a rotating flow in the limit of
vanishing magnetic fields that the square of the angular velocity increases
outwards, i.e.,
\begin{equation}
\partial_{r}(\omega^{2})>0,\ \ \text{for all }r\in(R_{1},R_{2}).\label{MRI}%
\end{equation}
It is remarkable that this stability criterion in the limit of vanishing
magnetic fields is different from the Rayleigh stability criterion
(\ref{Rayleigh}) for rotating flows without magnetic fields. Indeed, it is
easy to see that the condition (\ref{MRI}) implies that Rayleigh stability
condition (\ref{Rayleigh}) but not vice versa. If the stability condition
(\ref{MRI}) fails, it was suggested in \cite{VEP1959} and \cite{CS1960Couette}
that there is linear instability with small magnetic fields and they also
showed that the unstable eigenvalues are necessarily real. Acheson and Hide
\cite{AH1973} suggested that this instability mechanism plays a role in the
Earth's geodynamo problem. Such instability of rotating flows induced by small
magnetic fields is called magneto-rotational instability (MRI) in the
literature. Although there were some subsequent works in later decades
\cite{AG1978,FK1969}, the importance of MRI was not fully realized until 1991,
when Balbus and Hawley \cite{BH1991} provided a relatively simple elucidation
and a physical explanation of the important role that MRI plays in the
turbulence and enhanced angular momentum transport in astrophysical accretion
disks. In particular, for Keplerian rotation $\omega^{2}(r)=\frac{GM_{\ast}%
}{r^{3}}$ which is widely used in modeling accretion disks around black holes,
it is Rayleigh stable but becomes unstable by MRI with the addition of small
magnetic fields. We refer to the reviews \cite{SB2003,BH1991,BH1998,JMStone2011} for
the history and results of this important topic.

In the physical literature \cite{BH1991,BH1998,GKS2010,PJKA2014}, MRI is
usually obtained by a local dispersion analysis which we sketch below. First,
the spectral problem for $b=1$ can be reduced to the following nonlinear
eigenvalue problem
\begin{align}\label{noneigen}
&  \left(  \Lambda+\frac{k^{2}\epsilon^{2}}{\Lambda}\right)  r\partial
_{r}\left(  \frac{1}{r}\partial_{r}{\tilde{\varphi}}(r)\right)
  =k^{2}\left[  \Lambda+\frac{k^{2}\epsilon^{2}}{\Lambda}+\frac{4\omega^{2}%
}{\Lambda}\left(  1-\frac{k^{2}\epsilon^{2}}{\Lambda^{2}+k^{2}\epsilon^{2}%
}\right)  +\frac{r\partial_{r}(\omega^{2})}{\Lambda}\right]  {\tilde{\varphi}%
}(r),
\end{align}
with boundary conditions
\begin{align}\label{1.7boundarycd}
\tilde{\varphi}(R_{1})=\tilde{\varphi}(R_{2})=0,
\end{align}
where $\Lambda>0$ is the unstable eigenvalue, and $k$ is the $z-$frequency.
The local dispersion analysis is to consider localized perturbations with a
short radial wavelength, and obtain an algebraic dispersion relation as an
approximation of (\ref{noneigen}). Specifically, we take $\tilde{\varphi
}(r)=p(r)e^{ik_{r}r}$ for $r\in(R_{1},R_{2})$, where $p(r)\in C_{0}^{\infty
}(r_{0}-\delta,r_{0}+\delta)$ for $r_{0}\in(R_{1},R_{2})$ and $\delta>0$
small. Then the second-order ODE \eqref{noneigen} becomes
\begin{align*}
&  \left(  \Lambda+\frac{k^{2}\epsilon^{2}}{\Lambda}\right)  r\left[
-p+\frac{1}{k_{r}^{2}}\left(  \frac{p^{\prime}}{r}\right)  ^{\prime}+\frac
{i}{k_{r}}\frac{p^{\prime}}{r}+\frac{i}{k_{r}}\left(  \frac{p}{r}\right)
^{\prime}\right]  e^{ik_{r}r}\\
&  =\frac{k^{2}}{k_{r}^{2}}\left[  \Lambda+\frac{k^{2}\epsilon^{2}}{\Lambda
}+\frac{4\omega^{2}}{\Lambda}\left(  1-\frac{k^{2}\epsilon^{2}}{\Lambda
^{2}+k^{2}\epsilon^{2}}\right)  +\frac{r\partial_{r}(\omega^{2})}{\Lambda
}\right]  pe^{ik_{r}r}.
\end{align*}
By assuming that the radial wave number is sufficiently large, i.e., $k_{r}%
\gg1$, and the perturbation is highly localized at $r_{0}\in(R_{1},R_{2})$,
i.e., $\delta\ll1$, we obtain the following \textquotedblleft dispersion
equation\textquotedblright\
\begin{equation}
\left(  1+\frac{k^{2}}{k_{r}^{2}}\right)  (\Lambda^{2}+\epsilon^{2}k^{2}%
)^{2}+\frac{k^{2}}{k_{r}^{2}}\frac{\partial_{r}(\omega^{2}r^{4})}{r^{3}%
}|_{r=r_{0}}(\Lambda^{2}+\epsilon^{2}k^{2})-\frac{k^{4}}{k_{r}^{2}}%
\epsilon^{2}4\omega(r_{0})^{2}=0.\label{dispersion equation}%
\end{equation}
Then we have the following approximate formula for the roots of
(\ref{dispersion equation})
\[
\Lambda^{2}\approx%
\begin{cases}
-\frac{k^{2}}{k^{2}+k_{r}^{2}}\frac{\partial_{r}(\omega^{2}r^{4})}{r^{3}%
}|_{r=r_{0}}-(\frac{4\omega^{2}r^{3}}{\partial_{r}(\omega^{2}r^{4}%
)}+1)|_{r=r_{0}}\epsilon^{2}k^{2}\quad(\text{epicyclic modes})\\
-\frac{\partial_{r}(\omega^{2})r^{4}}{\partial_{r}(\omega^{2}r^{4})}%
|_{r=r_{0}}\epsilon^{2}k^{2}\quad(\text{MRIs})
\end{cases}
\]
with small magnetic field strength $\epsilon^{2}k^{2}\ll\frac{\partial
_{r}(\omega^{2}r^{4})}{r^{3}}|_{r=r_{0}}$. MRI for small magnetic fields can
be derived from above formula. However, the instability found by such local
analysis is often associated with the continuous spectra in many fluid
problems. We note that steady flows with unstable continuous spectra might be
nonlinearly stable. For example, for 2D Euler equations, steady flows with
hyperbolic stagnation points necessarily have unstable continuous spectra
(\cite{FV1991,LS2003,LH1991}) but can still be nonlinearly stable
(\cite{Lin2004CMP,LLZ2023}). Therefore, to verify MRI it is important to prove
the existence of global unstable modes satisfying the nonlinear eigenvalue
problem (\ref{noneigen}). In the astrophysical literature
(\cite{CP1995,CPS1994,DBG2018,LFF2015,MK2008,SPKL2022}), much efforts have
been made on the global analysis to solve the eigenvalue problem
(\ref{noneigen}). Moreover, there were even inconsistencies between the
stability conditions derived from local analysis and global analysis (see
\cite{MK2008,MK2009}). For some special cases such as $b=1,\omega^{2}=\omega_{0}%
^{2}\left(  1+\beta r^{\gamma}\right)  $ with constants $\omega_{0}%
,\beta,\gamma$, the equation (\ref{noneigen}) can be reduced to a
Schr\"{o}dinger-like equation with an effective potential. Then the global unstable
solutions of (\ref{noneigen}) can be obtained under some conditions and with
numerical help. See also \cite{MK2008} for the special case $\omega^{2}%
=c_{1}r+\frac{c_{2}}{r}$. But even for these special cases, the sharp
criterion for MRI was not proved.

Our goal of this paper is to address three natural questions for MRI: 1) What
is the sharp criterion for MRI, that is, the existence of an unstable
eigenvalue to the nonlinear eigenvalue problem (\ref{noneigen})? 2) What is
the nature of MRI? Is it due to continuous or discrete spectrum? 3) Is MRI
true at the nonlinear level? We are able to answer these questions in a
rigorous way. First, we give a sharp instability criterion for general
vertical magnetic fields and angular velocities. Second, we show that MRI is
due to discrete unstable spectrum. Moreover, we give the precise counting of
unstable modes and the exponential trichotomy estimates of the linearized
semigroup. Thirdly, we prove that the sharp stability or instability criteria
also imply nonlinear stability or instability respectively.
To state our results more precisely, first we introduce some notations. Define
the spaces
\[
H_{mag}^{r}(R_{1},R_{2}):=\{h(r)\ |\ h(R_{1})=h(R_{2})=0,\Vert h(r)\Vert
_{H_{mag}^{r}}<\infty\},
\]
with the norm
\[
\Vert h(r)\Vert_{H_{mag}^{r}(R_{1},R_{2})}:=\left(  \int_{R_{1}}^{R_{2}}%
\left(\frac{1}{r}|\partial_{r}h(r)|^{2}+\frac{1}{r}|h(r)|^{2}\right)dr\right)  ^{\frac
{1}{2}},
\]
and
\[
H_{mag}^{1}(\Omega):=\{\varphi(r,z)\ |\ \Vert\varphi\Vert_{H_{mag}^{1}%
(\Omega)}<\infty\},
\]
with the norm
\[
\Vert\varphi\Vert_{H_{mag}^{1}(\Omega)}:=\left(  \int_{\Omega}\left(\frac{1}{r^{2}%
}|\partial_{z}\varphi|^{2}+\frac{1}{r^{2}}|\partial_{r}\varphi|^{2}\right)dx\right)
^{\frac{1}{2}},
\]
and
\[
Z:=\{\varphi(r,z)\in H_{mag}^{1}(\Omega)\ |\ \varphi(R_{1},z)=\varphi
(R_{2},z)=0\}.
\]
Consider the energy spaces $\mathbf{X}=X\times Y$ with $X=L^{2}(\Omega)\times
Z,\ Y=L_{\text{div}}^{2}(\Omega)\times L^{2}(\Omega),$ where $L^{2}(\Omega)$
is the cylindrically symmetric $L^{2}$ space on $\Omega$, and
\[
L_{\text{div}}^{2}(\Omega):=\{\vec{u}=u_{r}(r,z)\mathbf{e}_{r}+u_{z}%
(r,z)\mathbf{e}_{z}\in (L^{2}(\Omega))^{2}\big|\operatorname{div}\vec{u}%
=0,u_{r}(R_{1},z)=u_{r}(R_{2},z)=0\}.
\]
For convenience of notation, we denote
\[
\vec{u}=(u_{r},u_{z})=u_{r}(r,z)\mathbf{e}_{r}+u_{z}(r,z)\mathbf{e}_{z}.
\]

To study the linear MRI, we use the separable Hamiltonian framework developed
in \cite{LZ2019}, instead of solving the eigenvalue problem (\ref{noneigen})
directly. First, we find that the linearized axisymmetric MHD system
\eqref{linearmhd} can be written in a Hamiltonian form
\begin{equation}
\frac{d}{dt}%
\begin{pmatrix}
u_{1}\\
u_{2}%
\end{pmatrix}
=\mathbf{J}\mathbf{L}%
\begin{pmatrix}
u_{1}\\
u_{2}%
\end{pmatrix}
,\label{linear-Hamiltonian}%
\end{equation}
where $u_{1}=\left(  u_{\theta}+\frac{\partial_{r}\omega(r)}{\epsilon
b(r)}\varphi,\varphi\right)  $, $u_{2}=(\vec{u},B_{\theta})$ with $\vec
{u}=(u_{r},u_{z})$. The off-diagonal anti-self-dual operator $\mathbf{J}$ and
diagonal self-dual operator $\mathbf{L}$ are defined by
\[
\mathbf{J}:=%
\begin{pmatrix}
0 & \mathbb{B}\\
-\mathbb{B}^{\prime} & 0
\end{pmatrix}
:\mathbf{X}^{\ast}\rightarrow\mathbf{X},\quad\mathbf{L}:=%
\begin{pmatrix}
\mathbb{L} & 0\\
0 & A
\end{pmatrix}
:\mathbf{X}\rightarrow\mathbf{X}^{\ast},
\]
where
\[
\mathbb{B}%
\begin{pmatrix}
\vec{u}\\
B_{\theta}%
\end{pmatrix}
=%
\begin{pmatrix}
-2\omega(r)u_{r}+\epsilon b(r)\partial_{z}B_{\theta}\\
\epsilon rb(r)u_{r}%
\end{pmatrix}
:Y^{\ast}\rightarrow X,
\]%
\[
\mathbb{B}^{^{\prime}}%
\begin{pmatrix}
f_{1}\\
f_{2}%
\end{pmatrix}
=%
\begin{pmatrix}
\mathcal{P}%
\begin{pmatrix}
-2\omega(r)f_{1}+r\epsilon b(r)f_{2}\\
0
\end{pmatrix}
\\
-\epsilon b(r)\partial_{z}f_{1}%
\end{pmatrix}
:X^{\ast}\rightarrow Y,
\]

\[
\mathbb{L}=%
\begin{pmatrix}
Id_{1}, & 0\\
0, & L
\end{pmatrix}
:X\rightarrow X^{\ast},\quad A=Id_{2}:Y\rightarrow Y^{\ast},
\]
where
\begin{equation*}
L:=-\frac{1}{r}\partial_{r}(\frac{1}{r}\partial_{r}\cdot)-\frac{1}{r^{2}%
}\partial_{z}^{2}+\mathfrak{F}(r):Z\rightarrow Z^{\ast},
\end{equation*}
$\mathfrak{F}(r)$ is defined in (\ref{defn-F-r}), $Id_{1}:L^{2}(\Omega
)\rightarrow\left(  L^{2}(\Omega)\right)  ^{\ast}$ and $Id_{2}:Y\rightarrow
Y^{\ast}$ are the isomorphisms. The operator $\mathcal{P}$ is the Leray
projection from $\left(  L^{2}(\Omega)\right)  ^{2}\ $to $L_{\text{div}}%
^{2}(\Omega)$, which is defined by \eqref{projectP}. By Theorem
\ref{T:abstract} for general separable Hamiltonian PDEs, the unstable spectra
of (\ref{linear-Hamiltonian}) are all discrete and the number of unstable
modes equals $n^{-}\left(  \mathbb{L}|_{\overline{R\left(  \mathbb{B}\right)
}}\right)  $, that is, the number of negative directions of $\left\langle
\mathbb{L}\cdot,\cdot\right\rangle $ restricted to $\overline{R(\mathbb{B})}$
which is shown to be
\[
\overline{R(\mathbb{B})}=\left\{
\begin{pmatrix}
g_{1}\\
g_{2}%
\end{pmatrix}
\in X\bigg|g_{j}(r,z)=\sum_{0\neq k \in \mathbb{Z}}e^{ikz}\tilde{\varphi}_{k,j}(r),\quad
j=1,2\right\}  .
\]
It follows that $n^{-}\left(  \mathbb{L}|_{\overline{R\left(  \mathbb{B}%
\right)  }}\right)  =2\sum\limits_{k=1}^{\infty}n^{-}(\mathbb{L}_{k})$, where
the operator $\mathbb{L}_{k}:H_{mag}^{r}\rightarrow(H_{mag}^{r})^{\ast}$ is
defined by
\begin{equation}
\mathbb{L}_{k}:=-\frac{1}{r}\partial_{r}\left(  \frac{1}{r}\partial_{r}%
\cdot\right)  +\frac{k^{2}}{r^{2}}+\mathfrak{F}(r)\label{defn-L-k}%
\end{equation}
for any $k\in \mathbb{Z},$ with
\begin{equation}
\mathfrak{F}(r):=\frac{\partial_{r}(\omega^{2})}{\epsilon^{2}b(r)^{2}%
r}+\left(  \frac{\partial_{r}^{2}b(r)}{r^{2}b(r)}-\frac{\partial_{r}%
b(r)}{r^{3}b(r)}\right)  .\label{defn-F-r}%
\end{equation}
Since $\mathbb{L}_{-k}=\mathbb{L}_{k}>\mathbb{L}_{1}$ for $k>1$, we get the sharp stability
criterion $\mathbb{L}_{1}\geq0$.

Below, we state our main results for the linear stability of rotating flows
with vertical magnetic fields under the axisymmetric perturbations.

\begin{theorem}
[A sharp stability/instability criterion]\label{linearstability}Assume
that the steady state $(v_{0},H_{0})(x)$ is given by \eqref{steady1},
with $\omega(r)\in C^{3}(R_{1},R_{2})$ and $b(r)\in C^{3}(R_{1},R_{2})$ with a
positive lower bound.\newline1) If $R_{1}=0$, $\partial_{r}(\omega
^{2})=O(r^{\beta-3}),\partial_{r}b=O(r^{\beta-1})$, for some constant
$\beta>0$, as $r\rightarrow0$.\newline2)If $R_{2}=\infty$, $\partial
_{r}(\omega^{2})=O(r^{-3-2\alpha}),\partial_{r}b=O(r^{-1-2\alpha})$, for some
constant $\alpha>0$, as $r\rightarrow\infty$.\newline Then the linearized
operator $\mathbf{JL}$ defined by (\ref{defn-JL}) generates a $C^{0}$ group
$e^{t\mathbf{JL}}$ of bounded linear operators on $\mathbf{X}=X\times Y$ and
there exists a decomposition%
\[
\mathbf{X}=E^{u}\oplus E^{c}\oplus E^{s}\quad
\]
of closed subspaces $E^{u,s,c}$ satisfying the following properties:

i) $E^{c},E^{u},E^{s}$ are invariant under $e^{t\mathbf{JL}}$.

ii) $E^{u}\left(  E^{s}\right)  $ only consists of eigenvectors corresponding
to positive (negative) eigenvalues of $\mathbf{JL}$ and
\[
\dim E^{u}=\dim E^{s}=%
\begin{cases}
0,\quad\text{if}\quad n^{-}(\mathbb{L}_{1})=0,\text{(linear stability)}\\
2\sum\limits_{k=1}^{\infty}n^{-}(\mathbb{L}_{k})>0,\quad\text{if}\quad
n^{-}(\mathbb{L}_{1})\neq0,\text{(linear instability)}%
\end{cases}
\]
where the operator $\mathbb{L}_{k}$ is defined in (\ref{defn-L-k}) and
$n^{-}(\mathbb{L}_{k})$ denotes the number of negative directions of
$\langle\mathbb{L}_{k}\cdot,\cdot\rangle$. In particular, $n^{-}%
(\mathbb{L}_{k})=0$ when $k$ is large enough.

iii) The exponential trichotomy is true in the space $\mathbf{X}$, i.e. if $n^{-}(\mathbb{L}_{1})\neq0$,
then there exists $M>0$ such that
\begin{equation}
\left\vert e^{t\mathbf{JL}}|_{E^{s}}\right\vert \leq Me^{-\lambda_{u}%
t},\;t\geq0;\quad\left\vert e^{t\mathbf{JL}}|_{E^{u}}\right\vert \leq
Me^{\lambda_{u}t},\;t\leq0, \label{estimate-stable-unstable}%
\end{equation}
where $\lambda_{u}=\min\{\lambda\mid\lambda\in\sigma(\mathbf{JL}|_{E^{u}%
})\}>0$.

iii) The quadratic form $\left\langle \mathbf{L}\cdot,\cdot\right\rangle
$ vanishes on $E^{u,s}$, i.e., $\langle\mathbf{L}\mathbf{u},\mathbf{u}%
\rangle=0$ for all $\mathbf{u}\in E^{u,s}$, but is non-degenerate on
$E^{u}\oplus E^{s}$, and
\[
E^{c}=\left\{  \mathbf{u}\in\mathbf{X}\mid\left\langle \mathbf{\mathbf{L}%
u,v}\right\rangle =0,\ \forall\ \mathbf{v}\in E^{s}\oplus E^{u}\right\}  .
\]
There exists $M>0$ such that
\begin{equation}
|e^{t\mathbf{J}\mathbf{L}}|_{E^{c}}|\leq M(1+t^{2}),\text{ for all }%
t\in\mathbb{R}. \label{estimate-center}%
\end{equation}

\end{theorem}

Theorem \ref{linearstability} gives not only the sharp stability criterion for
general rotating flows with vertical magnetic fields, but also more detailed
information on the spectra of the linearized operator and exponential
trichotomy estimates for the linearized MHD system which play important roles
on the study of nonlinear dynamics.

Back to the
linearized system of \eqref{EPM} around the steady solution $(v_0,H_0)(x)$ in \eqref{steady1}, we obtain
\begin{align*}
\begin{cases}
\partial_t\begin{pmatrix}
 u\\
 B
\end{pmatrix}=-
\begin{pmatrix}
u\cdot\nabla v_0+v_0\cdot\nabla u-B\cdot\nabla H_0-H_0\cdot\nabla B+\nabla W\\
u\cdot\nabla H_0+v_0\cdot\nabla B-B\cdot\nabla v_0-H_0\cdot\nabla u
\end{pmatrix}:=L_{u,B}\begin{pmatrix}
 u\\
 B
\end{pmatrix}\\
\text{div}u=\text{div} B=0,
\end{cases}
\end{align*}
where the operator $L_{u,B}$ can be defined as acting only on $(u,B)$ since the linearized pressure $W$ can be determined by $(u,B)$ linearly.
Let
\begin{align}
H^s_{mhd}:=\{(v_1,v_2)|v_i\in (H^s(\Omega))^3,\text{div}(v_i)=0 \text{ in } \Omega, v_i\cdot \textbf{n}=0 \text{ on } \partial\Omega, i=1,2\},\quad
 s\geq 0.
\end{align}
From the special structure of linearized MHD equations, we can obtain the following exponential trichotomy properties on $H^s_{mhd}$.
\begin{Corollary}
\label{exp_dychotomy} Assume that the functions
$\omega(r)\in H^{s+2}[R_{1},R_{2}]$, $b(r)\in H^{s+2}[R_{1},R_{2}](R_{2}<\infty)$ with
a positive lower bound. 
Then
 the linearized
operator $L_{u,B}$ generates a $C^{0}$ group
$e^{tL_{u,B}}$ of bounded linear operators on $H^s_{mhd}$ and
there exists a decomposition%
\[
H^s_{mhd}=X^{u}\oplus X^{c}\oplus X^{s}\quad
\]
of closed subspaces $X^{u,s,c}$ satisfying the following properties:

i) $X^{c},X^{u},X^{s}$ are invariant under $e^{tL_{u,B}}$.

ii) $X^{u}\left(  X^{s}\right)  $ only consists of eigenvectors corresponding
to positive (negative) eigenvalues of $L_{u,B}$ and
\[
\dim X^{u}=\dim X^{s}=\dim E^{u}=\dim E^{s}.
\]

iii) The exponential trichotomy is true in the space $H^s_{mhd}$, i.e. if $n^{-}(\mathbb{L}_{1})\neq0$,
then there exists $M>0$ such that
\begin{equation*}
\left\vert e^{tL_{u,B}}|_{X^{s}}\right\vert \leq Me^{-\lambda_{u}%
t},\;t\geq0;\quad\left\vert e^{tL_{u,B}}|_{X^{u}}\right\vert \leq
Me^{\lambda_{u}t},\;t\leq0, 
\end{equation*}
where $\lambda_{u}=\min\{\lambda\mid\lambda\in\sigma(L_{u,B}|_{X^{u}%
})\}>0$.

iv)
There exists $M>0$ such that
\begin{equation*}
|e^{tL_{u,B}}|_{X^{c}}|\leq M (1+|t|^{s+2}),\text{ for all }%
t\in\mathbb{R}. 
\end{equation*}

\end{Corollary}

We make some comments to compare MRI and Rayleigh instability for rotating
flows. Besides the significant gap of the Rayleigh criterion (\ref{Rayleigh})
and the MRI criterion (\ref{MRI}) in the limit of vanishing magnetic fields,
there are other fundamental differences between these two instability
mechanisms. First, MRI is caused by unstable discrete spectrum and
perturbations of low frequency ($z$-direction). In particular, the unstable
subspace is finite-dimensional and the maximal growth rate is obtained at low
frequency. By contrast, the Rayleigh instability for rotating flows is caused
by unstable continuous spectrum (see Section \ref{Eulercase}). The unstable
subspace is infinite-dimensional and the maximal growth rate is obtained at
the high frequency limit.

Below, we list some corollaries of Theorem \ref{linearstability}. First, we
show that the condition (\ref{MRI}) is indeed the sharp stability criterion in
the limit $\epsilon\rightarrow0$.

\begin{Corollary}
\label{BHcorolly} Under the assumptions of Theorem \ref{linearstability}, \\
i) If $\partial_{r}(\omega^{2})>0$,
then for $\epsilon^{2}$ small enough the steady state $(v_{0},H_{0})(x)$ in
\eqref{steady1} is linearly stable to axisymmetric perturbations.\newline ii)
If there exists $r_{0}\in(R_{1},R_{2})$ such that $\partial_{r}(\omega
^{2})|_{r=r_{0}}<0$,
then for $\epsilon^{2}$ small enough the steady state $(v_{0},H_{0})(x)$ in
\eqref{steady1} is linearly unstable to axisymmetric perturbations.

\end{Corollary}

For $b=1$, Corollary \ref{BHcorolly} recovers the classical MRI criterion for
linear stability in the limit of vanishing magnetic fields. Next, we give the
threshold of the magnetic field strength for MRI.

Define the operator
\[
\widehat{L}:=-\frac{1}{r}\partial_{r}\left(  \frac{1}{r}\partial_{r}%
\cdot\right)  +\frac{1}{r^{2}}+\left(  \frac{\partial_{r}^{2}b}{r^{2}b}%
-\frac{\partial_{r}b}{r^{3}b}\right)  :H_{mag}^{r}\rightarrow(H_{mag}%
^{r})^{\ast},
\]
then $n^{-}(\widehat{L})<\infty$ by the proof of Lemma \ref{le2.1}.

\begin{Corollary}
\label{BHcorolly2} Under the assumptions of Theorem \ref{linearstability},\\
i) If $\widehat{L}>0,$ let $\epsilon_{min}^{2}%
:=\max\left\{  \sup\limits_{\tilde{\varphi}(r)\in H_{mag}^{r}}\frac
{-\int_{R_{1}}^{R_{2}}\frac{\partial_{r}(\omega^{2})}{b^{2}}|\tilde{\varphi
}|^{2}dr}{\langle\widehat{L}\tilde{\varphi},\tilde{\varphi}\rangle},0\right\}
$, then the steady state $(v_{0},H_{0})(x)$ in \eqref{steady1} is linearly
stable to axisymmetric perturbations
if and only if $\epsilon^{2}\geq\epsilon_{min}^{2}$.\newline ii)If
$n^{-}(\widehat{L})>0$ and $\partial_{r}(\omega^{2})>0$, let $\frac
{1}{\epsilon_{max}^{2}}:=\sup\limits_{\tilde{\varphi}(r)\in H_{mag}^{r}}%
\frac{-\langle\widehat{L}\tilde{\varphi},\tilde{\varphi}\rangle}{\int_{R_{1}%
}^{R_{2}}\frac{\partial_{r}(\omega^{2})}{b^{2}}|\tilde{\varphi}|^{2}dr},$ then
the steady state $(v_{0},H_{0})(x)$ in \eqref{steady1} is linearly stable to
axisymmetric perturbations
if and only if $\epsilon^{2}\leq\epsilon_{max}^{2}$. \newline iii) If
$n^{-}(\widehat{L})>0$ and $\partial_{r}(\omega^{2})$ changes sign, then for
$\epsilon$ small enough or large enough, the steady state $(v_{0},H_{0})(x)$
in \eqref{steady1} is linearly unstable to axisymmetric perturbations.

\end{Corollary}

The supremum in the definition of $\epsilon_{min}^{2}$ and $\epsilon_{max}%
^{2}\ $can be obtained (see the proof of Corollary \ref{BHcorolly2}). By
Corollary \ref{BHcorolly2} i), for the uniform background magnetic field (i.e.,
$b(r)=1$) as considered in the astrophysical literature, the stability
criterion becomes
\begin{equation*}
\left\vert \epsilon\right\vert \geq B_{0},\ \ \ B_{0}^{2}:=\max\left\{
\sup_{\tilde{\varphi}(r)\in H_{mag}^{r}}\frac{-\int_{R_{1}}^{R_{2}}%
\partial_{r}(\omega^{2})|\tilde{\varphi}(r)|^{2}dr}{\int_{R_{1}}^{R_{2}}%
\left(\frac{1}{r}|\partial_{r}\tilde{\varphi}(r)|^{2}+\frac{1}{r}|\tilde{\varphi
}(r)|^{2}\right)dr},0\right\}  .
\end{equation*}


Corollary \ref{BHcorolly2} ii) shows that MRI might be generated by the uneven
distribution of the background magnetic field (i.e., $b\left(  r\right)  \neq
constant$) even if $\partial_{r}(\omega^{2})>0$. Corollary \ref{BHcorolly2}
iii) shows that if $n^{-}(\widehat{L})>0$ and $\partial_{r}(\omega^{2})$
changes sign, then: for $\epsilon$ small enough, MRI is caused by the negative
sign of $\partial_{r}(\omega^{2})$; for $\epsilon$ large enough, MRI is caused
by the uneven distribution of the background magnetic field. The next
corollary also shows that how the uneven distribution of the background
magnetic field can influence MRI.

\begin{Corollary}
\label{BHcorolly3} Under the assumptions of Theorem \ref{linearstability},\\
i) If $n^{-}(\widehat{L})=0$ and $\partial_{r}(\omega
^{2})>0$, then for all $\epsilon>0$, the steady state $(v_{0},H_{0})(x)$ in
\eqref{steady1} is linearly stable to axisymmetric perturbations.
\newline ii) If $n^{-}(\widehat{L})=0$ and $\ker(\widehat{L})=\hat{\varphi
}(r)$, then for all $\epsilon$ large enough, the steady state $(v_{0}%
,H_{0})(x)$ in \eqref{steady1} is linearly stable to axisymmetric
perturbations if $\int_{R_{1}}^{R_{2}}\frac{\partial_{r}(\omega^{2})}{b^{2}%
}|\hat{\varphi}|^{2}dr>0$ and linearly unstable when $\int_{R_{1}}^{R_{2}%
}\frac{\partial_{r}(\omega^{2})}{b^{2}}|\hat{\varphi}|^{2}dr<0$.
\end{Corollary}

Our second main result is to show that the sharp linear stability criterion in
Theorem \ref{linearstability} is also true at the nonlinear level. First we
state the result of nonlinear conditional stability when the linear stability
condition $\mathbb{L}_{1}>0$ holds. Define the distance function
\begin{align}
d(t) &  =d((v,H)(t,x),(v_{0},H_{0})(x))\nonumber\\
&=\frac{1}{2}\int_{\Omega}(|v_{r}%
|^{2}+|v_{z}|^{2}+|H_{\theta}|^{2})dx\label{distance-function}\\
&\quad+\frac{1}{2}\int_{\Omega}|v_{\theta}-\mathbf{v}_{0}(r)|^{2}dx+\frac
{1}{2}\int_{\Omega}\left(\frac{1}{r^{2}}|\partial_{z}(\psi-\psi_{0})|^{2}+\frac
{1}{r^{2}}|\partial_{r}(\psi-\psi_{0})|^{2}\right)dx,\nonumber
\end{align}
where $\psi$ and $\psi_{0}$ defined by \eqref{ellipticpsi} are the magnetic
potential functions of $H$ and $H_{0}$ respectively.

\begin{theorem}
[Nonlinear stability]\label{nonlinearstable} Assume that the steady state
$(v_{0},H_{0})(x)$ is given by \eqref{steady1}, where $\omega(r)\in
C^{3}(R_{1},R_{2})$ and $b(r)\in C^{3}(R_{1},R_{2})$ with a positive lower
bound.\newline When $R_{1}=0$ or $R_{2}=\infty$, we further assume i) If
$R_{1}=0$, $\partial_{r}\omega=O(r^{\beta-3}),\partial_{r}b=O(r^{\beta-1})$,
for some constant $\beta\geq4$, as $r\rightarrow0$. ii) If $R_{2}=\infty$,
$\omega=O(r^{-1-\alpha}),\partial_{r}b=O(r^{-1-2\alpha})$, for some constant
$\alpha>1$, as $r\rightarrow\infty$.\newline Assume the linear stability
condition $\mathbb{L}_{1}>0$ with $\mathbb{L}_{1}$ defined by
\eqref{defn-L-k} for $k=1$ and the existence of global axisymmetric weak
solution (defined in Definition \ref{weaksolution}) satisfying:\newline1)The
total energy defined by \eqref{energy} for $R_{2}<\infty$ or \eqref{'7.1} for
$R_{2}=\infty$ is non-increasing with respect to $t$.\newline2)The distance
function $d(t)$ defined by \eqref{distance-function} is continuous with respect to $t$.\newline3)The
functionals $\int_{\Omega}f(\psi)dx$ and $\int_{\Omega}rv_{\theta}\Phi
(\psi)dx$ for $R_{2}<\infty$ or the functionals $\int_{\Omega}(rv_{\theta}%
\Phi(\psi)-r^{2}\omega(r)\Phi(\psi_{0}))dx$ and $\int_{\Omega}%
(f(\psi)-f(\psi_{0}))dx$ for $R_{2}=\infty$ are conserved.\newline Then the
steady state $(v_{0},H_{0})(x)$ in \eqref{steady1} of \eqref{EPM}-\eqref{bdc}
is nonlinearly stable, in the sense that if $d(0)=d((v^{0},H^{0}%
)(x),(v_{0},H_{0})(x))$ is small enough, then the weak solution $(v,H)(t,x)$
of \eqref{EPM}-\eqref{bdc} satisfies
\[
d(t)=d((v,H)(t,x),(v_{0},H_{0})(x))\lesssim d(0)
\]
for any $t\geq0$.
\end{theorem}

Our next result is to prove nonlinear instability under the linear instability
condition $n^{-}(\mathbb{L}_{1})\neq0$.

\begin{theorem}
[Nonlinear instability]\label{nonlinearunstable} Assume that the functions
$\omega(r)\in C^{\infty}[R_{1},R_{2})$, $b(r)\in C^{\infty}[R_{1},R_{2})$ with
a positive lower bound. In addition, if $R_{2}=\infty$, $\omega=O(r^{-1-\alpha
}),\partial_{r}b=O(r^{-1-2\alpha})$, for some constant $\alpha>0$, as
$r\rightarrow\infty$.\newline If $n^{-}(\mathbb{L}_{1})\neq0$ with
$\mathbb{L}_{1}$ defined by (\ref{defn-L-k}) for $k=1$, then the steady state $(v_{0}%
,H_{0})(x)$ in \eqref{steady1} of \eqref{EPM}-\eqref{bdc} is nonlinearly
unstable in the sense that for any $s>0$ large enough, there exists
$\epsilon_{0}>0$, such that for any small $\delta>0$, there exists a family of
axisymmetric solutions $(v^{\delta},H^{\delta})(t,x)$ to
\eqref{EPM}-\eqref{bdc} satisfying
\[
\Vert v^{\delta}(0,x)-v_{0}(x)\Vert_{H^{s}}+\Vert H^{\delta}(0,x)-H_{0}%
(x)\Vert_{H^{s}}\leq\delta
\]
and
\[
\sup_{0\leq t\leq T^{\delta}}\{\Vert v^{\delta}(t,x)-v_{0}(x)\Vert_{L^{2}%
}+\Vert H^{\delta}(t,x)-H_{0}(x)\Vert_{L^{2}}\}\geq\epsilon_{0},
\]
where $T^{\delta}=O(|\ln\delta|)$.
\end{theorem}

Below we discuss main ideas in the proof of Theorems \ref{nonlinearstable} and
\ref{nonlinearunstable}. To prove nonlinear stability for the case
$R_{2}<\infty$, we construct the Casimir-energy functional
\[
E_{c}(v,H_{\theta},\psi)=\frac{1}{2}\int_{\Omega}\left(  v^{2}+H^{2}\right)
\ dx+\int_{\Omega}rv_{\theta}\Phi(\psi)dx+\int_{\Omega}f(\psi)dx,
\]
 where $\Phi$ and $f$ are defined in
(\ref{extension1})-(\ref{extension2}). The Casimir invariants$\ \int_{\Omega
}rv_{\theta}\Phi(\psi)dx$ and $\int_{\Omega}f(\psi)dx$ can be proved by using
the following Hamiltonian structure of the nonlinear system \eqref{EPM}
\[%
\begin{cases}
\partial_{t}(rv_{\theta})=\{\varpi,rv_{\theta}\}+\{rH_{\theta},\psi\},\\
\partial_{t}\psi=\{\varpi,\psi\},
\end{cases}
\]
where the Poisson bracket $\{f,g\}:=\frac{1}{r}\left(  \partial_{r}%
f\partial_{z}g-\partial_{z}f\partial_{r}g\right)  $ and the stream function
$\varpi$ is such that $-\frac{\partial_{r}\varpi}{r}=v_{z}$, $\frac
{\partial_{z}\varpi}{r}=v_{r}$. By our construction, $(v_{0},0,\psi_{0})$ is a
critical point of $E_{c}(v,H_{\theta},\psi)$. The second variation of
$E_{c}(v,H_{\theta},\psi)$ at $(v_{0},0,\psi_{0})\ $can be written as
\begin{equation}
\langle\mathbb{L}\delta u_{1},\delta u_{1}\rangle+\langle A\delta u_{2},\delta
u_{2}\rangle,\label{2nd variation}%
\end{equation}
where $\delta u_{1}=\left(  \delta v_{\theta}+\frac{\partial_{r}\omega
(r)}{\epsilon b(r)}\delta\psi,\delta\psi\right)  $, $\delta u_{2}=(\delta
v_{r},\delta v_{z},\delta H_{\theta})$ are associated with the perturbation
$(v,H_{\theta},\psi)-(v_{0},0,\psi_{0})$. The quadratic form
(\ref{2nd variation}) is conserved for the linearized system
(\ref{linear-Hamiltonian}). The functional $E_{c}(v,H_{\theta},\psi)$ can be
shown to be \thinspace$C^{2}$ in the energy space. By the Taylor expansion,
we have
\begin{align*}
E_{c}(v,H_{\theta},\psi)-E_{c}(v_{0},0,\psi_{0}) &  =\frac{1}{2}\left\Vert
\delta v_{\theta}+\frac{\partial_{r}\omega}{\epsilon b}\delta\psi\right\Vert
_{L^{2}}^{2}+\frac{1}{2}\langle L\delta\psi,\delta\psi\rangle\\
&  \quad+\frac{1}{2}\langle A\delta u_{2},\delta u_{2}\rangle+o(\Vert
\delta\psi\Vert_{H_{mag}^{1}}^{2})+o(\Vert\delta v_{\theta}\Vert_{L^{2}}^{2}).
\end{align*}
To prove nonlinear stability, the key step is to get the positivity of
$\langle L\delta\psi,\delta\psi\rangle$. The linear stability condition
$\mathbb{L}_{k}>0$ for $k\neq0$ implies that $n^{-}\left(  \mathbb{L}\right)
=\sum\limits_{ k\in \mathbb{Z}}n^{-}(\mathbb{L}_{k})=n^{-}(\mathbb{L}_{0})$. We
need to deal with the negative directions and possible zero direction of
$\left\langle \mathbb{L}_{0}\cdot,\cdot\right\rangle \ $in the space of zero
modes, that is, the space $H_{mag}^{r}(R_{1},R_{2})$. For this purpose, let
$K=n^{-}(\mathbb{L}_{0})\ $denote $\left\{  h_{i}\left(  r\right)  \right\}
_{i=0}^{K}$ be the kernel and negative directions of $\mathbb{L}_{0}$
($i=1,\cdots,K$ if $\ker\mathbb{L}_{0}=\{0\}$). Then we construct invariants
$J_{i}(\psi)=\int_{\Omega}f_{i}(\psi)dx$ ($i=0,1,..,K$) with $f_{i}^{\prime
}(\psi_{0})=h_{i}\left(  r\right)  \mathfrak{F}(r)$. Expanding the invariants
$J_{i}(\psi)$ at $\psi_{0}$, we get
\[
J_{i}(\psi)-J_{i}(\psi_{0})=\int_{\Omega}\mathfrak{F}(r)h_{i}(r)(\psi-\psi
_{0})dx+O(d),
\]
which implies that the projection of $\delta\psi\ $to $h_{i}(r)$ is
$O(d(0)^{\frac{1}{2}})+O(d)$, i.e.,
\[
\left\vert \int_{\Omega}\mathfrak{F}(r)h_{i}(r)(\psi-\psi_{0})dx\right\vert
=O(d(0)^{\frac{1}{2}})+O(d).
\]
By using above estimates, we deduce that the projection $P_{K}\delta\psi$ of $\delta\psi$ to the
space spanned by $\{h_{i}(r)\}$ is $O(d(0)^{\frac{1}{2}})+O(d)$, and
consequently
\[
|\langle LP_{K}\delta\psi,P_{K}\delta\psi\rangle|=O(d(0))+O(d^{2})
\]
can be treated as a higher order small quantity. Since the positivity holds
for $\langle L(I-P_{K})\delta\psi,(I-P_{K})\delta\psi\rangle$, we are able to
get the desirable control for the term $\langle L\delta\psi,\delta\psi\rangle$
and the nonlinear stability can be shown. The similar arguments are also used
to prove nonlinear stability for the case $R_{2}=\infty$.

To prove the nonlinear instability under the linear instability condition
$n^{-}(\mathbb{L}_{1})\neq0$, we use the approach in \cite{GE2000,MR1978583}
to construct higher-order approximate solutions to \eqref{2.1'} for initial
perturbation along the most unstable mode of the linearized operator. Then
energy estimates are used to overcome the loss of derivative and show that the
leading order of the perturbation is given by the solution of the
linearized equations with exponential growth. To construct the higher-order approximate solutions, a
key step is to obtain the sharp exponential growth estimate for the linearized
MHD equations \eqref{hamiltonian-RS}, i.e.,
\begin{equation}
\Vert(u,B)(t)\Vert_{H^{s}}\lesssim e^{\Lambda t}\Vert(u,B)(0)\Vert_{H^{s}%
}\quad\text{for any}\quad s\geq0,\label{sharp}%
\end{equation}
where $\Lambda^{2}=\max\{\sigma(-\mathbb{B}^{^{\prime}}\mathbb{L}%
\mathbb{B}A)\}>0$ determines the maximal growth rate. The sharp estimate
(\ref{sharp}) can not be obtained by the energy methods such as in \cite{MR1978583} due to difficulties caused by the boundary integral terms and
non-commutativity of $r-$derivative and Leray Projection, i.e., $\partial_{r}^{\alpha}\mathcal{P}\neq\mathcal{P}\partial
_{r}^{\alpha}$. To prove  (\ref{sharp}), we strongly use the Hamiltonian
structure of the linearized equations \eqref{hamiltonian-RS}. We sketch the
main steps below:

Step 1: We find the Hamiltonian structure is still preserved if we apply
$\partial_{z}^{\alpha}$ $(|\alpha|\leq s)$ to \eqref{hamiltonian-RS}. So we
can use the exponential trichotomy estimates (\ref{estimate-stable-unstable})
and (\ref{estimate-center}) to get the estimate for $\Vert\partial_{z}%
^{\alpha}(u,B)(t)\Vert_{L^{2}}$, i.e.,
\[
\Vert\partial_{z}^{\alpha}(u,B)(t)\Vert_{L^{2}}\lesssim e^{\Lambda t}%
\Vert(u,B)(0)\Vert_{H^{s}},\quad|\alpha|\leq s.
\]

Step 2: The Hamiltonian structure is destroyed if we apply $\partial
_{r}^{\alpha}$ $(|\alpha|\leq s)$ to \eqref{hamiltonian-RS}. However, we
observe that
\[
\partial_{t}\left(
\begin{array}
[c]{c}%
\partial_{r}^{\alpha}u_{\theta}\\
\partial_{r}^{\alpha}B_{\theta}%
\end{array}
\right)  =A_{1}\left(
\begin{array}
[c]{c}%
\partial_{r}^{\alpha}u_{\theta}\\
\partial_{r}^{\alpha}B_{\theta}%
\end{array}
\right)  +f_{s},
\]
where $A_{1}$ is an anti-self-adjoint operator and $f_{s}$ defined
by\eqref{fs} contains the low-order $r-$derivative and at most $s-$order
$z-$derivative terms. Thus by induction assumptions and Duhamel's principle,
we shall obtain the growth estimate for $\Vert\partial_{r}^{\alpha}(u_{\theta
},B_{\theta})(t)\Vert_{L^{2}}$, i.e.,
\[
\left\Vert \left(
\begin{array}
[c]{c}%
\partial_{r}^{\alpha}u_{\theta}\\
\partial_{r}^{\alpha}B_{\theta}%
\end{array}
\right)  (t)\right\Vert _{L^{2}}\lesssim e^{\Lambda t}\Vert(u,B)(0)\Vert
_{H^{s}},\quad|\alpha|\leq s.
\]

Step 3: We observe that
\[
\partial_{t}\left(
\begin{array}
[c]{c}%
\partial_{r}^{\alpha}u_{r}\\
\partial_{r}^{\alpha}u_{z}\\
\partial_{r}^{\alpha}B_{r}\\
\partial_{r}^{\alpha}B_{z}%
\end{array}
\right)  =A_{2}\left(
\begin{array}
[c]{c}%
\partial_{r}^{\alpha}u_{r}\\
\partial_{r}^{\alpha}u_{z}\\
\partial_{r}^{\alpha}B_{r}\\
\partial_{r}^{\alpha}B_{z}%
\end{array}
\right)  +%
\begin{pmatrix}
\partial_{r}^{\alpha}\mathcal{P}%
\begin{pmatrix}
2\omega(r)u_{\theta}\\
0
\end{pmatrix}
\\
0\\
0
\end{pmatrix}
,\quad|\alpha|\leq s,
\]
where $A_{2}$ is an anti-self-adjoint operator. Hence by the estimate obtained
in Step 2 for $\Vert\partial_{r}^{\alpha}u_{\theta}(t)\Vert_{L^{2}}$ and
Duhamel's principle, we get the estimate for $\Vert\partial_{r}^{\alpha}%
(u_{r},u_{z},B_{r},B_{z})(t)\Vert_{L^{2}},$ i.e.,
\[
\left\Vert \left(
\begin{array}
[c]{c}%
\partial_{r}^{\alpha}u_{r}\\
\partial_{r}^{\alpha}u_{z}\\
\partial_{r}^{\alpha}B_{r}\\
\partial_{r}^{\alpha}B_{z}%
\end{array}
\right)  (t)\right\Vert _{L^{2}}\lesssim e^{\Lambda t}\Vert(u,B)(0)\Vert
_{H^{s}},\quad|\alpha|\leq s.
\]
Thus, we obtain the estimates for $\Vert\partial_{r}^{\alpha}(u,B)(t)\Vert
_{L^{2}}$, i.e.,
\[
\Vert\partial_{r}^{\alpha}(u,B)(t)\Vert_{L^{2}}\lesssim e^{\Lambda t}%
\Vert(u,B)(0)\Vert_{H^{s}}%
\]
with $|\alpha|\leq s.$ Combining together, we obtain the sharp exponential
growth estimate \eqref{sharp}.


Below, we will summarize the novelties of our methods used in this article:

$\bullet$ The biggest novelty of our approach is to totally bypass the
solution of the ODE \eqref{noneigen} which is very difficult to study.
Instead, we find a new separable Hamiltonian formulation of the linearized MHD equation
and utilize this structure to study eigenvalue problem and get semigroup estimates.
 Although the concept and
abstract formulation of separable Hamiltonian PDE were proposed by \cite{LZ2019} for a totally different problem, for the
current MHD problem, it is highly nontrivial to find such a separable Hamiltonian structure. To find it, we need to group the unknowns carefully and use
proper transformations to rewrite the equations. It is far from being obvious
to recognize that the linearized MHD equation can be written in a separable
Hamiltonian formulation due to the complexity of the equations.

$\bullet$ In the nonlinear stability proof, the main difficulty is to dealt with infinitely
many Casimir constraints of the perturbations. Under the sharp linear stability criterion, we find an approach to reduce the infinitely many constraints to
finitely many constraints which exactly match with the negative dimension of
the second order expansion of the energy-Casimir functional. To our knowledge,
such an approach to dealt with infinitely many constraints has not been used in
other papers.

$\bullet$ For the nonlinear instability proof, it is important to get the sharp
exponential growth estimate of the linearized semigroup. It is far from enough to only use the separable Hamiltonian structure of the linearized equation, especially
for the estimates of the higher-order derivatives. The usual energy estimates
are also too crude to use. Our novelty is to find additional unitary structures of the equations for higher-order derivatives. This enables us to
get the sharp semigroup estimates required for the nonlinear instability proof. Moreover, we get exponential trichotomy which will
be used for the construction of unstable and stable manifolds in a work under preparation.

In this paper, we consider steady magnetic fields with only the vertical
component. There are also works \cite{CP1995,PAD2013,PAD2017,EK1992,VEP1959} on the MRI induced by other
steady magnetic fields. However, there is no complete classification of the
kind of magnetic fields causing MRI. We believe that our methods might be
applicable to study MRI for more general magnetic fields.

The paper is organized as follows. In Section \ref{Linearstability}, we prove
the sharp stability criterion for rotating flows with vertical magnetic
fields. In Section \ref{nonlinearstability}, we prove the conditional
nonlinear stability under the linear stability condition. In Section
\ref{nonlinearinstability}, we prove nonlinear instability under the linear
instability condition. In Section \ref{Eulercase}, we study the nonlinear
stability/instability for rotating flows of Euler equations and make
comparison with the case of rotating flows with magnetic fields. In Appendix,
we provide detailed proofs of some technical lemmas and the conditional
nonlinear stability for the case of unbounded domain.

\noindent{\textbf{Notation}}. Throughout this paper, we use $C>0$ to represent
a generic constant. The notation $a\lesssim b$ represents $a\leq Cb$ and
$a\sim b$ means $a=Cb$ for some constant $C>0$. For simplicity, we denote
$L^{p}:=L^{p}(\Omega),H^{s}:=H^{s}(\Omega)$, $H_{mag}^{r}:=H_{mag}^{r}%
(R_{1},R_{2})$ and $H_{mag}^{1}:=H_{mag}^{1}(\Omega)$. Let $(\cdot
,\cdot)_{H_{mag}^{1}}$ be the inner product in $H_{mag}^{1}.$ To simplify
notations, we use  $f(R_{2})$ or $f(R_{2},z)$ for $\lim\limits_{r\rightarrow
\infty}f(r)$ or $\lim\limits_{r\rightarrow\infty}f(r,z)$ for the case
$R_{2}=\infty$. 

\section{Linear stability/instability for rotating shear flows with magnetic
fields}\label{Linearstability}
Roughly speaking, we shall use a
framework of separable Hamiltonian systems in \cite{LZ2019} to study the linear stability/instability for the linearized MHD system \eqref{linearmhd}. For completeness, we sketch the abstract theory in \cite{LZ2019}.

\subsection{Separable Linear Hamiltonian PDEs}

Consider a linear Hamiltonian PDEs of the separable form%
\begin{equation}\label{separate-hamiltonian}
\frac{d}{dt}\left(
\begin{array}
[c]{c}%
u_1\\
u_2
\end{array}
\right)  =\left(
\begin{array}
[c]{cc}%
0 & \mathbb{B}\\
-\mathbb{B}^{\prime} & 0
\end{array}
\right)  \left(
\begin{array}
[c]{cc}%
\mathbb{L} & 0\\
0 & A
\end{array}
\right)  \left(
\begin{array}
[c]{c}%
u_1\\
u_2
\end{array}
\right)  =\mathbf{JL}\left(
\begin{array}
[c]{c}%
u_1\\
u_2
\end{array}
\right)  , %
\end{equation}
where $u_1\in X,\ u_2\in Y$ and $X,Y$ are real Hilbert spaces. We briefly describe
the results in \cite{LZ2019} about general separable Hamiltonian PDEs
\eqref{separate-hamiltonian}. The triple $\left(  \mathbb{L},A,\mathbb{B}\right)  $ is assumed
to satisfy assumptions:

\begin{enumerate}
\item[(\textbf{G1})] The operator $\mathbb{B}:Y^{\ast}\supset D(\mathbb{B})\rightarrow X$ and
its dual operator $\mathbb{B}^{\prime}:X^{\ast}\supset D(\mathbb{B}^{\prime})\rightarrow Y\ $are
densely defined and closed (and thus $\mathbb{B}^{\prime\prime}=\mathbb{B}$).

\item[(\textbf{G2})] The operator $A:Y\rightarrow Y^{\ast}$ is bounded and
self-dual. 
Moreover, there exist $\tilde{\delta}>0$ and a closed subspace $Y_{+}\subset Y$ such
that
\[
Y=\ker A\oplus Y_{+},\quad\quad \langle Au_2,u_2\rangle\geq\tilde{\delta}\left\Vert u_2\right\Vert _{Y}^{2},\;\forall u_2\in
Y_{+}.
\]

\item[(\textbf{G3})] The operator $\mathbb{L}:X\rightarrow X^{\ast}$ is bounded and
self-dual (i.e., $\mathbb{L}^{\prime}=\mathbb{L}$ \textit{etc.}) and there exists a decomposition
of $X$ into the direct sum of three closed subspaces
\begin{equation*}
X=X_{-}\oplus\ker \mathbb{L}\oplus X_{+},\ \ n^{-}(\mathbb{L})\triangleq\dim
X_{-}<\infty\label{decom-X}%
\end{equation*}
satisfying

\begin{enumerate}
\item[(\textbf{G3.a})] $\left\langle \mathbb{L}u_1,u_1\right\rangle <0$ for all $u_1\in
X_{-}\backslash\{0\}$;

\item[(\textbf{G3.b})] there exists $\tilde{\delta}>0$ such that
\[
\left\langle \mathbb{L}u_1,u_1\right\rangle \geq\tilde{\delta}\left\Vert u_1\right\Vert ^{2}\ ,\text{
for any }u_1\in X_{+}.
\]
\end{enumerate}
\item[(\textbf{G4})] The above $X_{\pm}$  and $Y_{+}$ satisfy
\[
\ker (i_{X_{+}\oplus X_{-}})^{\prime}\subset D(\mathbb{B}^{\prime}), \quad\quad \ker (i_{Y_{+}})^{\prime}\subset D(\mathbb{B}).
\]

\end{enumerate}
\begin{remark}
\label{remark-G4}The assumption (\textbf{G4}) for $\mathbb{L}$ (or for $A$) is satisfied automatically if
$\dim\ker \mathbb{L}<\infty$ (or for $\dim\ker A<\infty$). (see Remark 2.3 in \cite{lin-zeng-hamiltonian} for details).
\end{remark}

\begin{theorem}
\label{T:abstract} \cite{LZ2019}Assume (\textbf{G1-4}) for
(\ref{separate-hamiltonian}). The operator $\mathbf{JL}$ generates a $C^{0}$
group $e^{t\mathbf{JL}}$ of bounded linear operators on $\mathbf{X}=X\times Y$
and there exists a decomposition%
\[
\mathbf{X}=E^{u}\oplus E^{c}\oplus E^{s}\quad
\]
of closed subspaces $E^{u,s,c}$ with the following properties:

i) $E^{c},E^{u},E^{s}$ are invariant under $e^{t\mathbf{JL}}$.

ii) $E^{u}\left(  E^{s}\right)  $ only consists of eigenvectors corresponding
to positive (negative) eigenvalues of $\mathbf{JL}$ and
\begin{equation*}
\dim E^{u}=\dim E^{s}=n^{-}\left(  \mathbb{L}|_{\overline{R\left(  \mathbb{B}A\right)  }}\right)
, \label{unstable-dimension-formula}%
\end{equation*}
where $n^{-}\left(  \mathbb{L}|_{\overline{R\left(  \mathbb{B}A\right)  }}\right)  $ denotes the
number of negative modes of
$\left\langle \mathbb{L}\cdot,\cdot\right\rangle |_{\overline{R\left(  \mathbb{B}A\right)  }}$.
If $n^{-}\left(  \mathbb{L}|_{\overline{R\left(  \mathbb{B}A\right)  }}\right)  >0$,
then there exists $M>0$ such that
\begin{equation*}
\left\vert e^{t\mathbf{JL}}|_{E^{s}}\right\vert \leq Me^{-\lambda_{u}%
t},\;t\geq0;\quad\left\vert e^{t\mathbf{JL}}|_{E^{u}}\right\vert \leq
Me^{\lambda_{u}t},\;t\leq0, %
\end{equation*}
where $\lambda_{u}=\min\{\lambda\mid\lambda\in\sigma(\mathbf{JL}|_{E^{u}%
})\}>0$.

iii) The quadratic form $\left\langle \mathbf{L}\cdot,\cdot\right\rangle
$ vanishes on $E^{u,s}$, i.e., $\langle\mathbf{L}\mathbf{u},\mathbf{u}%
\rangle=0$ for all $\mathbf{u}\in E^{u,s}$, but is non-degenerate on
$E^{u}\oplus E^{s}$, and
\[
E^{c}=\left\{  \mathbf{u}\in\mathbf{X}\mid\left\langle \mathbf{\mathbf{L}%
u,v}\right\rangle =0,\ \forall\ \mathbf{v}\in E^{s}\oplus E^{u}\right\}  .
\]
There exists $M>0$ such that
\begin{equation*}
|e^{t\mathbf{J}\mathbf{L}}|_{E^{c}}|\leq M(1+t^{2}),\text{ for all }%
t\in\mathbb{R}. %
\end{equation*}

iv) Suppose $\left\langle \mathbb{L}\cdot,\cdot\right\rangle $ is non-degenerate on
$\overline{R\left(  \mathbb{B}\right)  }$, then $|e^{t\mathbf{JL}}|_{E^{c}}|\leq M$ for
some $M>0$. Namely, there is Lyapunov stability on the center space $E^{c}$.

\end{theorem}

\subsection{Stability/instability of the linearized axisymmetric MHD equations}
We consider the axisymmetric solution of the system \eqref{EPM}. In the cylindrical coordinates, we denote
$$H(t,r,z)=H_r(t,r,z)\mathbf{e}_r+H_\theta(t,r,z)\mathbf{e}_\theta +H_z(t,r,z)\mathbf{e}_z$$  and  $$v(t,r,z)=v_r(t,r,z)\mathbf{e}_r+v_\theta(t,r,z)\mathbf{e}_\theta +v_z(t,r,z)\mathbf{e}_z.$$
Due to $\text{div} H=0$, we can define the magnetic potential $\psi(t,r,z)$ of $H_r(t,r,z)\mathbf{e}_r +H_z(t,r,z)\mathbf{e}_z$ by the following boundary-value problem,
\begin{align}\label{ellipticpsi}
\begin{cases}
-\frac{1}{r}\partial_r\left(\frac{1}{r}\partial_r\psi\right)-\frac{1}{r^2}\partial_{z}^2\psi=\frac{1}{r}\partial_rH_z-\frac{1}{r}\partial_{z}H_r,\\
\psi(t,R_1,z)=0,\\
\psi(t,R_2,z)=-\int_{R_1}^{R_2}rH_z(t,r,z)dr=-\int_{R_1}^{R_2}rH_z(0,r,z)dr=C(R_1,R_2).
\end{cases}
\end{align}
The last two equalities follow from
\begin{align*}
\partial_t\int_{R_1}^{R_2}rH_z(t,r,z)dr&=-\int_{R_1}^{R_2}\partial_r[r(v_rH_z-v_zH_r)]dr\\
&=r(v_rH_z-v_zH_r)|_{r=R_1}^{r=R_2}=0
\end{align*}
and
\begin{align*}
\partial_z\int_{R_1}^{R_2}rH_z(t,r,z)dr&=-\int_{R_1}^{R_2}\partial_r(rH_r)dr=rH_r|_{r=R_1}^{r=R_2}=0,
\end{align*}
where we used $\eqref{bdc}_1$, $\eqref{EPM}_3$.
By axial symmetry, we are able to extend the function $\psi$ from $\Omega\subset\mathbb{R}^3$ to $\Omega^{e}\subset\mathbb{R}^5$, where $\Omega^{e}=\{x=(x_1,x_2,x_3,x_4,x_5)\in\mathbb{R}^5|R_1\leq\sqrt{x_1^2+x_2^2+x_4^2+x_5^2}< R_2, x_3\in\mathbb{T}_{2\pi}\}.$
The 5-dimensions extension method is also used to deal with the elliptic problem in other models such as the magnetic stars \cite{FLJ2015} and the vortex rings \cite{Ni1980}.
In precise,
we give the extension of $\psi(t,r,z)$ by
$\psi^e(t,R,z)$, where $R=\sqrt{x_1^2+x_2^2+x_4^2+x_5^2}, z=x_3.$
Hence we have
\begin{align*}
-\Delta_5\left(\frac{\psi^e}{R^2}\right)&=
-\partial_{R}^2\left(\frac{\psi^e}{R^2}\right)-\frac{3}{R}\partial_R\left(\frac{\psi^e}{R^2}\right)-\partial_{z}^2\left(\frac{\psi^e}{R^2}\right)\\
&=
-\frac{1}{R}\partial_R\left(\frac{1}{R}\partial_R\psi^e\right)-\frac{1}{R^2}\partial_{z}^2\psi^e.
\end{align*}
Then the boundary-value problem \eqref{ellipticpsi} is equivalent to the following boundary-value problem for Poisson's equation in $5$-dimensions
\begin{align}\label{Poissonequation}
\begin{cases}
-\Delta_5\left(\frac{\psi^e}{R^2}\right)=\frac{1}{R}\partial_RH_z(t,R,z)-\frac{1}{R}\partial_{z}H_r(t,R,z),\\
\psi^e(t,R_1,z)=0,\\
\psi^e(t,R_2,z)=-\int_{R_1}^{R_2}rH_z(0,r,z)dr,
\end{cases}
\end{align}
where $\frac{1}{R}\partial_RH_z(t,R,z)-\frac{1}{R}\partial_{z}H_r(t,R,z)\in \dot{H}^{-1}(\Omega^e)$.
Thus, we can obtain the unique solution $\psi^e(t,R,z)$ with $\frac{\psi^e}{R^2}\in\dot{H}^1(\Omega^{e})$ to \eqref{Poissonequation}.
Meanwhile, we get the unique solution $\psi(t,r,z)$ to \eqref{ellipticpsi}.
Moreover, $\psi$ automatically satisfies
\begin{align}\label{timeevolution}
\begin{cases}
H_r(t,r,z)=\frac{\partial_z\psi}{r},\\
H_z(t,r,z)=-\frac{\partial_r\psi}{r},\\
\psi(t,R_1,z)=0,\\
\psi(t,R_2,z)=-\int_{R_1}^{R_2}rH_z(0,r,z)dr.
\end{cases}
\end{align}
By $\eqref{EPM}_2$ and \eqref{timeevolution}, we have
\begin{align*}
\partial_t \nabla\psi=\nabla[r(v_rH_z-v_zH_r)]=\nabla(-v_r\partial_r\psi-v_z\partial_z\psi).
\end{align*}
Together with $\eqref{timeevolution}_3$ and $\eqref{bdc}_1$, one has
\begin{align*}
\partial_t \psi-r[v_rH_z-v_zH_r]
&=\partial_t\psi(t,R_1,z)-R_1[v_r(R_1,z)H_z(R_1,z)
\\
&\quad-v_z(R_1,z)H_r(R_1,z)]=0.
\end{align*}
Then, the system \eqref{EPM} can be rewritten in the cylindrical coordinates as
\begin{align}\label{nonlinearMHDrz}
\begin{cases}
  \partial_t v_r+\partial_{r} p=\frac{\partial_z\psi}{r}\partial_r (\frac{\partial_z\psi}{r})-\frac{\partial_r\psi}{r^2} \partial_{zz} \psi-\frac{1}{r}H_{\theta}^2-v_r\partial_r v_r-v_z\partial_z v_r+\frac{1}{r}v_{\theta}^2,\\
\partial_t v_\theta=\frac{\partial_z\psi}{r} \partial_r H_\theta-\frac{\partial_r\psi}{r} \partial_z H_\theta+\frac{1}{r^2}H_{\theta}  \partial_z\psi-v_r\partial_r v_\theta-v_z\partial_z v_\theta-\frac{v_{\theta}  v_{r}}{r},\\
\partial_t v_z+\partial_{z} p=-\frac{\partial_z\psi}{r} \partial_r (\frac{\partial_r\psi}{r})+\frac{\partial_r\psi}{r^2} \partial_{zr} \psi-v_r\partial_r v_z-v_z\partial_z v_z,\\
\partial_t \psi=-v_r\partial_r\psi-v_z\partial_z\psi,\\
\partial_t H_\theta=\frac{\partial_z\psi}{r} \partial_r v_\theta-\frac{\partial_r\psi}{r} \partial_z v_\theta+\frac{H_{\theta}  v_{r}}{r}-v_r\partial_r H_\theta-v_z\partial_z H_\theta-\frac{v_{\theta}  \partial_z\psi}{r^2},\\
\frac{1}{r}\partial_r(rv_r)+\partial_zv_z=0.
\end{cases}
\end{align}
For steady state, we can take
\begin{align*}
 \psi_0(r)=-\epsilon\int_{R_1}^{r}sb(s)ds .
\end{align*}
Now let the perturbations be
\begin{align*}
u(t,x)&=v(t,x)-v_0(x),\quad B_\theta(t,x)=H_\theta(t,x), \\
W(t,x)&=p(t,x)-p_0(x),\quad \varphi(t,r,z)=\psi-\psi_0.
\end{align*}
Then the linearized MHD system around a given steady state $(v_0,H_0)$ in \eqref{steady1}
under cylindrical coordinates can be reduced to the following equations:
\begin{align}\label{linearmhd}
\begin{cases}
 \partial_t u_r=\epsilon b(r)\partial_{z} (\frac{\partial_z\varphi}{r})+2\omega(r)u_{\theta}-\partial_r W,\\
\partial_t u_\theta
=\epsilon b(r)\partial_zB_\theta-\frac{u_r}{r}\partial_r(r^2\omega(r)),\\
\partial_t u_z=\epsilon b(r)\partial_{z} (-\frac{\partial_r\varphi}{r})-\partial_{z} W +\frac{\epsilon\partial_rb(r)}{r}\partial_z\varphi ,\\
\partial_t \varphi=\epsilon rb(r)u_r ,\\
\partial_t B_\theta=\epsilon b(r)\partial_zu_\theta+\partial_r\omega(r)\partial_z \varphi,\\
\frac{1}{r}\partial_r(ru_r)+\partial_zu_z=0.
\end{cases}
\end{align}
We impose the system \eqref{linearmhd} with the initial and boundary conditions
\begin{align*}
\begin{cases}
(u_r,u_\theta,u_z, \varphi, B_\theta)(t,r,z)\mid_{t=0}=(u_r^0, u_\theta^0,u_z^0, \varphi^0, B_\theta^0)(r,z),\\
u_r(t,R_1,z)=\varphi(t,R_1,z)=0,\\
u_r(t,R_2,z)=\varphi(t,R_2,z)=0,\quad  \mbox{if}\quad R_2<\infty;\\
(u_r,u_\theta,u_z, \varphi, B_\theta)(t,r,z)\to (0,0,0,0,0) \quad \mbox{as} \quad r\to \infty, \quad \mbox{if}\quad R_2=\infty;\\
(u_r,u_\theta,u_z, \varphi, B_\theta)(t,r,z)=(u_r, u_\theta,u_z, \varphi, B_\theta)(t,r,z+2\pi).
\end{cases}
\end{align*}
In order to get the separable Hamiltonian structure of the linearized system \eqref{linearmhd},
let
\begin{align}\label{u1u2}
u_1=\left(u_{\theta}+\frac{\partial_{r}\omega(r)}{\epsilon b(r)}\varphi,\varphi\right) \text{ and } u_2=(\vec{u},B_\theta),
\end{align}
with
$\vec{u}=(u_r,u_z),$
then
\begin{align*}
\frac{du_1}{dt}=\begin{pmatrix} -2\omega(r)u_r+\epsilon b(r)\partial_zB_\theta\\
\epsilon r b(r)u_r
\end{pmatrix}.
\end{align*}
Then, we define the Leray projection operator $\mathcal{P}$.  Specifically,
for any $\vec{u}(r,z)=(u_r,u_z)=u_r\mathbf{e}_r+u_z\mathbf{e}_z\in L^2(\Omega)$,
$\mathcal{P}$ is defined by
\begin{align}\label{projectP}
\mathcal{P}\vec{u}:=(u_r-\partial_rp(r,z),u_z-\partial_zp(r,z))=(u_r-\partial_rp)\mathbf{e}_r+(u_z-\partial_zp)\mathbf{e}_z,
\end{align}
 where
$p(r,z)\in \dot{H}^1(\Omega)$ is a weak solution of the equation
$$\Delta p=\text{div}\vec{u}.$$
This is equivalent to
\begin{align}\label{POweakeq}
\int_\Omega \nabla p\cdot\nabla\xi dx=\int_\Omega \vec{u}\cdot\nabla\xi dx,\quad \forall \xi\in \dot{H}^1(\Omega).
\end{align}
The right-hand side of \eqref{POweakeq} defines a bounded linear functional on $\dot{H}^1(\Omega)$. Thus by the
Riesz representation theorem, there exists a unique $p$ satisfying \eqref{POweakeq},
and consequently $\text{div}(\mathcal{P}\vec{u})=\text{div}(\vec{u}-\nabla p)=0$.
Moreover, we have
\begin{align*}
\frac{d}{dt}\begin{pmatrix}u_r\\
u_z\end{pmatrix}&=\mathcal{P}\begin{pmatrix}\epsilon b\partial_{z} (\frac{\partial_z\varphi}{r})+2\omega u_{\theta}\\
\partial_{z}(-\epsilon b(r) \frac{\partial_r\varphi}{r} +\frac{\epsilon\partial_rb(r)}{r}\varphi)
\end{pmatrix}\\
&=\mathcal{P}\begin{pmatrix}\epsilon b\partial_{z} (\frac{\partial_z\varphi}{r})+2\omega u_{\theta}-\partial_r[\epsilon b(r) (-\frac{\partial_r\varphi}{r}) +\frac{\epsilon\partial_rb(r)}{r}\varphi]\\
0
\end{pmatrix}\\
&=\mathcal{P}\begin{pmatrix}2\omega u_{\theta}+\epsilon b (\frac{\partial_{z}^2\varphi}{r})
 +\epsilon b\partial_r(\frac{\partial_r\varphi}{r})+
\epsilon\frac{\partial_rb(r)}{r^2}\varphi-\epsilon\frac{\partial_{r}^2b(r)}{r}\varphi\\
0
\end{pmatrix}\\
&=-\mathcal{P}\begin{pmatrix}-2\omega (u_{\theta}+\frac{\partial_{r}\omega(r)}{\epsilon b(r)}\varphi)+r\epsilon b L\varphi\\
0
\end{pmatrix},
\end{align*}
by using the fact that 
\begin{align*}
\mathcal{P}\begin{pmatrix}\partial_{r}(-\epsilon b(r) \frac{\partial_r\varphi}{r} +\frac{\epsilon\partial_rb(r)}{r}\varphi)\\
\partial_{z}(-\epsilon b(r) \frac{\partial_r\varphi}{r} +\frac{\epsilon\partial_rb(r)}{r}\varphi)
\end{pmatrix}=0,
\end{align*}
where
 $$L:=-\frac{1}{r}\partial_r(\frac{1}{r}\partial_r\cdot)-\frac{1}{r^2}\partial_z^2
+\mathfrak{F}(r): Z\rightarrow Z^{*},$$
$\mathfrak{F}(r):=\frac{\partial_r(\omega^2)}{ \epsilon^2b(r)^2r}+ \left(\frac{\partial_r^2b(r)}{r^2b(r)}-\frac{\partial_rb(r)}{r^3b(r)}\right) $.
Then the linearized axisymmetric MHD system \eqref{linearmhd} takes the Hamiltonian form
\begin{equation}
\frac{d}{dt}%
\begin{pmatrix}
u_{1}\\
u_{2}%
\end{pmatrix}
=\mathbf{J}\mathbf{L}%
\begin{pmatrix}
u_{1}\\
u_{2}%
\end{pmatrix}
,\label{hamiltonian-RS}%
\end{equation}
where
the off-diagonal anti-self-dual operator $\mathbf{J}$ and diagonal self-dual operator $\mathbf{L}$ are defined by
\begin{equation}
\mathbf{J}:=%
\begin{pmatrix}
0 & \mathbb{B}\\
-\mathbb{B}^{\prime} & 0
\end{pmatrix}
:\mathbf{X}^{\ast}\rightarrow\mathbf{X},\quad\mathbf{L}:=%
\begin{pmatrix}
\mathbb{L} & 0\\
0 & A
\end{pmatrix}
:\mathbf{X}\rightarrow\mathbf{X}^{\ast},\label{defn-JL}%
\end{equation}
with
\begin{equation}\label{2.7}
\mathbb{B}\begin{pmatrix}
\vec{u}\\
B_\theta%
\end{pmatrix}=\begin{pmatrix} -2\omega(r)u_r+\epsilon b(r)\partial_zB_\theta\\
\epsilon r b(r)u_r
\end{pmatrix}: Y^*\rightarrow X,
\end{equation}
\begin{equation}\label{2.8}
\mathbb{B}^{'}
\begin{pmatrix} f_1\\
f_2
\end{pmatrix}=
\begin{pmatrix} \mathcal{P}\begin{pmatrix}-2\omega(r)f_1+r\epsilon b(r)f_2\\
0\end{pmatrix}\\
- \epsilon b(r)\partial_zf_1
\end{pmatrix}: X^*\rightarrow Y,
\end{equation}

\begin{equation}\label{defL}
\mathbb{L}=\begin{pmatrix} Id_1, &0\\
0,&L
\end{pmatrix}: X\rightarrow X^{*},
\end{equation}
\begin{equation}\label{2.10}
A=Id_2 : Y\rightarrow Y^{*}.
\end{equation}
Then we need to check that $\left(  \mathbb{L},A,\mathbb{B}\right)  $ in the Hamiltonian PDEs \eqref{hamiltonian-RS}
 satisfy the assumptions (\textbf{G1})-(\textbf{G4}). The properties (\textbf{G1})
and (\textbf{G2}) are easy to check by the definition.
To show that $\mathbb{L}$ satisfies (\textbf{G3}),
first we observe that
\begin{align}\label{2.11}
-\frac{1}{r}\partial_r\left(\frac{1}{r}\partial_r\varphi\right)-\frac{1}{r^2}\partial_z^2\varphi
&=-\partial_r^2\left(\frac{\varphi}{r^2}\right)-\frac{3}{r}\partial_r\left(\frac{\varphi}{r^2}\right)-\partial_z^2\left(\frac{\varphi}{r^2}\right)
=-\Delta_5\left(\frac{\varphi^e}{R^2}\right),
\end{align}
where $\varphi^e(t,R,z)$ is the extension of $\varphi(t,r,z)$ from $\Omega\subset\mathbb{R}^3$ to $\Omega^{e}\subset\mathbb{R}^5$
with $R=(x_1^2+x_2^2+x_4^2+x_5^2)^{1/2}$, $z=x_3$.

By using \eqref{2.11}, a routine computation gives
\begin{align*}
\int_{\Omega}\left(\frac{1}{r^2}|\partial_z\varphi|^2+\frac{1}{r^2}|\partial_r\varphi|^2\right)d_3x=\int_{\Omega^{e}}\left|\nabla_5\left(\frac{\varphi^{e}}{R^2}\right)\right|^2 d_5x,
\end{align*}
which implies
\begin{align}\label{quadratic-equivalent}
\langle L\varphi,\varphi\rangle
=\int_{\Omega^{e}}\left(\left|\nabla_5\left(\frac{\varphi^{e}}{R^2}\right)\right|^2
+R^2\mathfrak{F}(R)\left|\frac{\varphi^{e}}{R^2}\right|^2\right) d_5x.
\end{align}
Therefore, we can define a new operator in $5$-dimensions,
$$L^e:=-\Delta_5+R^2\mathfrak{F}(R):\dot{H}^1(\Omega^{e})\rightarrow\dot{H}^{-1}(\Omega^{e}).$$ By \eqref{quadratic-equivalent},
\begin{align}\label{equivalent}
\left\langle L^e\left(\frac{\varphi^{e}}{R^2}\right),\left(\frac{\varphi^{e}}{R^2}\right)\right\rangle=\langle L\varphi,\varphi\rangle.
\end{align}
Define the operator $\mathcal{L}:L^2(\Omega^{e})\rightarrow L^2(\Omega^{e})$ by
\begin{align*}
\mathcal{L}:&=\left(-\Delta_5\right)^{-\frac{1}{2}}L^e\left(-\Delta_5\right)^{-\frac{1}{2}}
\\&=I+\left(-\Delta_5\right)^{-\frac{1}{2}}R^2\mathfrak{F}(R)\left(-\Delta_5\right)^{-\frac{1}{2}}
\\&:=I+\mathcal{K},
\end{align*}
then assumption (\textbf{G3}) follows from the following lemma.
\begin{lemma}\label{le2.1} Under the assumptions of Theorem \ref{linearstability},
 the operator $\mathcal{L}$ is bounded and self-adjoint on $L^2(\Omega^{e})$ and the operator $\mathcal{K}:L^2(\Omega^{e})\rightarrow L^2(\Omega^{e})$ is compact.
\end{lemma}
\begin{proof}
For the case $0<R_1<R_2<\infty$, the proof of this lemma is easy by the compactness of $H_0^1(\Omega^{e})\hookrightarrow L^2(\Omega^{e})$. So we only prove the case of $0=R_1<R_2=\infty$.
To show that $\mathcal{K}$ is compact. It is suffices to show that for any sequence $\{u_k\}\subset L^2(\Omega^{e})$ satisfying $u_k\rightharpoonup u_\infty$ weakly in $L^2(\Omega^{e})$, we have $\mathcal{K}u_k\rightarrow \mathcal{K}u_\infty$ strongly in $L^2(\Omega^{e})$.
Let $\chi\in C_0^{\infty}(\Omega^{e};[0,1])$ be a radial cut-off function satisfying
$\chi(x)=1$ if $|x|\leq\frac{1}{2}$ and $\chi(x)=0$ if $|x|\geq1.$
For any $M>2$, define $\chi_M=\chi(\frac{x}{M}).$ For convenience, we set $v_k=(-\Delta_5)^{-\frac{1}{2}}(u_k-u_\infty)$, then $\|v_k\|_{\dot{H}^1(\Omega^{e})}\leq\|u_k-u_\infty\|_{L^2(\Omega^{e})}\leq C.$ Hence,
\begin{align*}
&\|\mathcal{K}(u_k-u_\infty)\|_{L^2(\Omega^{e})}=\|(-\Delta_5)^{-\frac{1}{2}}R^2\mathfrak{F}(-\Delta_5)^{-\frac{1}{2}}(u_k-u_\infty)\|_{L^2(\Omega^{e})}\\
&=\left\|\frac{1}{|\xi|}\left(R^2\mathfrak{F}v_k\right)^{\wedge}(\xi)\right\|_{L^2(\Omega^{e})}
\leq C\left\|\nabla_\xi\left(R^2\mathfrak{F}v_k\right)^{\wedge}(\xi)\right\|_{L^2(\Omega^{e})}\\
&\leq C\left\||x|R^2\mathfrak{F}v_k\right\|_{L^2(\Omega^{e})}\\
&\leq C\left(\left\||x|R^2\mathfrak{F}\chi_Mv_k\right\|_{L^2(\Omega^{e})}+\left\||x|R^2\mathfrak{F}(1-\chi_M)v_k\right\|_{L^2(\Omega^{e})}\right)\\
&\leq C\left(\left\||x|R^2\mathfrak{F}\chi_M\chi_{\frac{2}{M}}v_k\right\|_{L^2(|x|<\frac{2}{M})}+\left\||x|R^2\mathfrak{F}\chi_M(1-\chi_{\frac{2}{M}})v_k\right\|_{L^2(\frac{1}{M}<|x|<M)}
\right.\\ &\left.\quad\quad\quad+\left\||x|R^2\mathfrak{F}(1-\chi_M)v_k\right\|_{L^2(|x|>\frac{M}{2})}\right)\\
&\leq C\left(\left\||x|R^2\mathfrak{F}\chi_M\chi_{\frac{2}{M}}v_k\right\|_{L^2(|x|<\frac{2}{M})}+\left\||x|R^2\mathfrak{F}\chi_M(1-\chi_{\frac{2}{M}})v_k\right\|_{L^2(\frac{1}{M}<|x|<M)}
\right.\\ &\left.\quad\quad\quad+\left\||x|R^2\mathfrak{F}(1-\chi_M)v_k\right\|_{L^2(|x|>\frac{M}{2})}\right)\\
&\leq C\left(\frac{1}{M^{\beta}}\left\|\frac{1}{|x|}\chi_M\chi_{\frac{2}{M}}v_k\right\|_{L^2(|x|<\frac{2}{M})}+\left\||x|R^2\mathfrak{F}\chi_M(1-\chi_{\frac{2}{M}})v_k\right\|_{L^2(\frac{1}{M}<|x|<M)}
\right.\\ &\left.\quad\quad\quad+\frac{1}{M^{2\alpha}}\left\|\frac{1}{|x|}(1-\chi_{M})v_k\right\|_{L^2(|x|>\frac{M}{2})}\right)
\\
&\leq \frac{C}{M^{\beta}}\left\|\nabla(\chi_M\chi_{\frac{2}{M}}v_k)\right\|_{L^2(\Omega^{e})}+C\left\||x|R^2\mathfrak{F}\chi_M(1-\chi_{\frac{2}{M}})v_k\right\|_{L^2(\frac{1}{M}<|x|<M)}
\\&\quad+\frac{C}{M^{2\alpha}}\|\nabla[(1-\chi_M)v_k]\|_{L^2(\Omega^{e})}\\
&\leq \frac{C}{M^{\beta}}(\left\|\nabla v_k\right\|_{L^2(\Omega^{e})}+\|\nabla(\chi_M\chi_{\frac{2}{M}})\|_{L^5(\Omega^{e})}\|v_k\|_{L^{\frac{10}{3}}(\Omega^{e})})\\
&\quad+C\left\||x|R^2\mathfrak{F}\chi_M(1-\chi_{\frac{2}{M}})v_k\right\|_{L^2(\frac{1}{M}<|x|<M)}
\\&\quad+\frac{C}{M^{2\alpha}}\left(\|\nabla v_k\|_{L^2(\Omega^{e})}+\left\|\nabla\chi_{M}\right\|_{L^5(\Omega^{e})}\|v_k\|_{L^{\frac{10}{3}}(\Omega^{e})}\right)\\
&\leq \frac{C}{M^{\beta}}\|\nabla v_k\|_{L^2(\Omega^{e})}+C\left\||x|R^2\mathfrak{F}\chi_M(1-\chi_{\frac{2}{M}})v_k\right\|_{L^2(\frac{1}{M}<|x|<M)}\nonumber\\
&\quad+\frac{C}{M^{2\alpha}}\|\nabla v_k\|_{L^2(\Omega^{e})},
\end{align*}
when $M$ is sufficiently large, by using $\mathfrak{F}(R)=O(R^{-4-2\alpha})$ when $R\rightarrow\infty$ and $\mathfrak{F}(R)=O(R^{-4+\beta})$ when $R\rightarrow0$, for some constants $\alpha>0,\beta>0$ given by the assumption of Theorem \ref{linearstability}.

Since $M$ can be arbitrarily large, by the compactness of $H_0^1(\{\frac{1}{M}<|x|<M\})\hookrightarrow L^2(\Omega^{e})$, we conclude $\|\mathcal{K}(u_k-u_\infty)\|_{L^2(\Omega^{e})}\rightarrow0$ when $k\rightarrow\infty.$ Thus $\mathcal{K}:L^2(\Omega^{e})\rightarrow L^2(\Omega^{e})$ is a compact operator. Since $\mathcal{K}$ is symmetric and bounded on $L^2(\Omega^{e})$, the self-adjointness of $\mathcal{L}$ follows from the Kato-Rellich theorem.

\end{proof}

Since $\mathcal{K}$ is a compact operator, it has discrete spectra and zero being the only possible accumulation point. 
Moreover, the operator $\mathcal{L}$ is bounded from below. Thus, $n^{-}(\mathcal{L})=n^{-}(I+\mathcal{K})<\infty.$  
Since
\begin{align*}
\langle L^ef,f\rangle&=\langle\left(-\Delta_5\right)^{\frac{1}{2}}\mathcal{L}\left(-\Delta_5\right)^{\frac{1}{2}}f,f\rangle\\
&=\langle\mathcal{L}\left(-\Delta_5\right)^{\frac{1}{2}}f,\left(-\Delta_5\right)^{\frac{1}{2}}f\rangle,
\end{align*}
together with \eqref{equivalent}, one has $n^{-}(\mathbb{L})=n^{-}(L)=n^{-}(L^e)=n^{-}(\mathcal{L})<\infty.$ Thus $\mathbb{L}$ satisfies (\textbf{G3}).
Since $\dim\ker \mathbb{L}<\infty$ and $A=I_{3\times3}$, the assumption (\textbf{G4}) for $\mathbb{L}$ (or for $A$) is automatically satisfied by the Remark \ref{remark-G4}.
Therefore, we verify the assumptions (\textbf{G1})-(\textbf{G4}) for the operators $\mathbb{L},\mathbb{B},\mathbb{B}^{'},A$ given by $\eqref{2.7}-\eqref{2.10}$. 

\begin{lemma}\label{embedd}
Under the assumptions of Theorem \ref{linearstability}, we have
$$H^1_{mag}(\Omega)\hookrightarrow\hookrightarrow L^2_{|\mathfrak{F}(r)|}(\Omega),$$
$$H^r_{mag}(R_1,R_2)\hookrightarrow\hookrightarrow L^2_{r|\mathfrak{F}(r)|}(R_1,R_2).$$
\end{lemma}
\begin{proof}
For the case $0<R_1<R_2<\infty$, the proof of this lemma is clear.
So we consider the case $R_1=0$ and $R_2=\infty$.
For any bounded sequence $\{\phi_i(r,z)\}$ $(i=1,2,...)$ in $H^1_{mag}(\Omega)$,
there exists a subsequence, still denoted by $\{\phi_i\}$, satisfying $\phi_i\rightharpoonup\phi_\infty$ in $H^1_{mag}(\Omega)$.
Denote $\delta_i(r,z)=\phi_i-\phi_\infty$. Let $\chi(r)\in C_0^{\infty}((R_1,R_2);[0,1])$ be a cut-off function satisfying
$\chi(r)=1$ if $r\leq\frac{1}{2}$ and $\chi(r)=0$ if $r\geq1.$
For any $M>2$, define $\chi_M=\chi(\frac{r}{M}).$
By using $\mathfrak{F}(R)=O(R^{-4-2\alpha})$ when $R\rightarrow\infty$ and $\mathfrak{F}(R)=O(R^{-4+\beta})$ when $R\rightarrow0$, for some constants $\alpha>0,\beta>0$ in Theorem \ref{linearstability}, and Hardy's inequality,
we obtain
\begin{align*}
\|\delta_i\|_{L^2_{|\mathfrak{F}(r)|}(\Omega)}
&=\|\delta_i\chi_{\frac{2}{M}}\|_{L^2_{|\mathfrak{F}(r)|}(\Omega)}
+\|\delta_i\chi_M(1-\chi_{\frac{2}{M}})\|_{L^2_{|\mathfrak{F}(r)|}(\Omega)}+\|\delta_i(1-\chi_M)\|_{L^2_{|\mathfrak{F}(r)|}(\Omega)}\nonumber\\
&=2\pi
\int_0^{2\pi}\left(\int_{0}^{\frac{2}{M}}|\delta_i|^2 r|\mathfrak{F}(r)|\chi_{\frac{2}{M}}dr+\int_{\frac{1}{M}}^{M}|\delta_i|^2 r|\mathfrak{F}(r)|\chi_M(1-\chi_{\frac{2}{M}})dr\right)dz\\
&\quad+2\pi
\int_0^{2\pi}\int_{\frac{M}{2}}^\infty|\delta_i|^2 r|\mathfrak{F}(r)|(1-\chi_M)drdz\\
&\leq \frac{C}{M^{\beta}}\int_0^{2\pi}\int_{0}^{\frac{2}{M}}|\delta_i|^2 r^{-3}drdz+\int_0^{2\pi}\int_{\frac{1}{M}}^{M}|\delta_i|^2 r|\mathfrak{F}(r)|drdz\nonumber\\
&\quad+\frac{C}{M^{2\alpha}}\int_0^{2\pi}\int_{\frac{M}{2}}^\infty|\delta_i|^2 r^{-3}drdz\nonumber\\
& \leq \frac{C}{M^{\beta}}\int_0^{2\pi}\int_{0}^{\frac{2}{M}}\frac{1}{r}|\partial_r\delta_i|^2drdz+\int_0^{2\pi}\int_{\frac{1}{M}}^{M}|\delta_i|^2 r|\mathfrak{F}(r)|drdz\nonumber\\
&\quad+\frac{C}{M^{2\alpha}}\int_0^{2\pi}\int_{\frac{M}{2}}^\infty\frac{1}{r}|\partial_r\delta_i|^2drdz.
\end{align*}
Since $M$ can be arbitrarily large, by the compactness of $H_0^1(\{\frac{1}{M}<r<M\})\hookrightarrow L^2(\Omega)$, we obtain
 $\phi_i\rightarrow\phi_\infty$ in $L^2_{|\mathfrak{F}(r)|}(\Omega)$. Thus we have $H^1_{mag}(\Omega)\hookrightarrow\hookrightarrow L^2_{|\mathfrak{F}(r)|}(\Omega),$ which implies
$H^r_{mag}(R_1,R_2)\hookrightarrow\hookrightarrow L^2_{r|\mathfrak{F}(r)|}(R_1,R_2).$

\end{proof}

\begin{lemma}\label{ragab}
It holds that
\begin{align*}
\overline{R(\mathbb{B})}
&= \left\{\begin{pmatrix} g_1\\ g_2\end{pmatrix}(r,z)\in X|\int_{0}^{2\pi}g_i(r,z)dz=0,\quad (i=1,2)\right\}.
\end{align*}
\end{lemma}
\begin{proof}
For $(f_1,f_2)^{T}(r,z)\in\text{Ker} \mathbb{B}^{'},$ we have
\begin{equation*}
\mathbb{B}'
\begin{pmatrix} f_1(r,z)\\
f_2(r,z)
\end{pmatrix}=
\begin{pmatrix} \mathcal{P}\begin{pmatrix}-2\omega(r)f_1(r,z)+r\epsilon b(r)f_2(r,z)\\
0\end{pmatrix}\\
- \epsilon b(r)\partial_zf_1(r,z)
\end{pmatrix}=0.
\end{equation*}
This implies that $f_1(r,z) $ is a function independent of $z$, i.e., $f_1(r,z)=f_1(r)$.
We deduce from $\begin{pmatrix}-2\omega(r)f_1(r,z)+r\epsilon b(r)f_2(r,z)\\
0\end{pmatrix}=\nabla g(r,z)$ that $\partial_zg(r,z)=0$. This implies that $g(r,z) $ is a function independent of $z$, i.e., $g(r,z)=g(r)$. Hence, $f_2(r,z)=[g'(r)+2\omega(r)f_1(r)]/r \epsilon b(r)$ is also a function independent of $z$, thus we have $f_2(r,z)=f_2(r)$.
It follows that
$$\text{Ker} \mathbb{B}^{'}=\left\{(f_1, f_2)^{T}(r)\in X^*\right\}.$$ Consequently, we have
\begin{align*}
\overline{R(\mathbb{B})}&=(\text{Ker} \mathbb{B}^{'})^{\perp}
\\&= \left\{\begin{pmatrix} g_1\\ g_2\end{pmatrix}(r,z)\in X\big|\int_{R_1}^{R_2}\int_{0}^{2\pi}(f_1(r)g_{1}r+f_2(r)g_{2}r) drdz= 0,\forall \begin{pmatrix} f_1\\ f_2\end{pmatrix}(r)\in X^*\right\}\\
&= \left\{\begin{pmatrix} g_1\\ g_2\end{pmatrix}\in X\big|\int_{0}^{2\pi}g_i(r,z)dz=0,\quad i=1,2\right\}.
\end{align*}
\end{proof}

\begin{proof}[Proof of Theorem \ref{linearstability}]
By Theorem \ref{T:abstract}, the solutions of \eqref{hamiltonian-RS} are spectrally stable if and only if $\mathbb{L}|_{\overline{R(\mathbb{B})}}\geq0.$ Moreover, $n^{-}(\mathbb{L}|_{\overline{R(\mathbb{B})}})$ equals to the number of unstable modes.
Then we need to calculate $n^{-}(\mathbb{L}|_{\overline{R(\mathbb{B})}})$.
Define the operator $\mathbb{L}_k: H^r_{mag}\rightarrow (H^r_{mag})^*$ by
\begin{align}\label{defLk}
\mathbb{L}_k =-\frac{1}{r}\partial_r\left(\frac{1}{r}\partial_r \cdot \right)+\frac{k^2}{r^2}+\mathfrak{F}(r)
\end{align}
for any $k\in \mathbb{Z}.$
Then for any $(g_1,g_2)^{T}\in\overline{R(\mathbb{B})},$ by the Fourier expansion, taking
\begin{equation*}
g_2=\sum_{0\neq k\in \mathbb{Z}}e^{ikz}\tilde{\varphi}_k(r),
\end{equation*}
it follows that
\begin{align*}
\left\langle\mathbb{L}\begin{pmatrix} g_1\\
g_2
\end{pmatrix},
\begin{pmatrix} g_1\\
g_2
\end{pmatrix}\right\rangle\nonumber
&=2\pi\int_{R_1}^{R_2}\int_{0}^{2\pi}\left(|g_1|^2r+\frac{1}{r}|\partial_z g_2|^2+\frac{1}{r}|\partial_r g_2|^2+\mathfrak{F}(r)r|g_2|^2\right)drdz\nonumber\\
&= (2\pi)^2\sum\limits_{0\neq k\in \mathbb{Z}}\langle \mathbb{L}_k \tilde{\varphi}_k,\tilde{\varphi}_k\rangle+2\pi\int_{R_1}^{R_2}\int_{0}^{2\pi}|g_1|^2rdrdz.
\end{align*}
Thus, we have
$n^-(\mathbb{L}|_{\overline{R(\mathbb{B})}})
=2 \sum\limits_{k=1}^\infty n^-(\mathbb{L}_k)$.
Actually $n^-(\mathbb{L}|_{\overline{R(\mathbb{B})}})=2\sum\limits_{k=1}^N n^-(\mathbb{L}_k)$ for some $N<\infty$, since $n^-(\mathbb{L}_k)=0$ for $k$ large enough
by the definition of $\mathbb{L}_k$ in \eqref{defLk}.

 Since $\mathbb{L}_{-k}=\mathbb{L}_k>\mathbb{L}_1$ for $k>1$,  we deduce that if $\mathbb{L}_1\geq0$,
 then $n^-(\mathbb{L}|_{\overline{R(\mathbb{B})}})=0$.
On the other hand, if $n^-(\mathbb{L}_1)>0$, then $n^-(\mathbb{L}|_{\overline{R(\mathbb{B})}})\geq 2n^-(\mathbb{L}_1)>0$. Thus $\mathbb{L}_1\geq0$ is the sharp criterion for linear stability.
\end{proof}

\begin{proof}[Proof of Corollary \ref{BHcorolly}]
From Theorem \ref{linearstability}, a necessary and sufficient condition for the linearly stability is $\mathbb{L}_1\geq0$,
 that is, for any $\tilde{\varphi}(r)\in H^r_{mag}$,
\begin{align}\label{L1hh}
\int_{R_1}^{R_2}\frac{\partial_{r}(\omega^2)}{  b^2} |\tilde{\varphi}(r)|^2dr
\geq-\epsilon^2\int_{R_1}^{R_2}\left[\frac{1}{r}|\tilde{\varphi}(r)|^2+\frac{1}{r}|\partial_r \tilde{\varphi}(r)|^2+ \left(\frac{\partial_r^2b}{rb}-\frac{\partial_rb}{r^2b}\right)|\tilde{\varphi}(r)|^2\right]dr.
\end{align}
First, for the case i) $0<R_1<R_2<\infty$,
if $\partial_r(\omega^2)>0$ and $\epsilon^2$ is small enough, then
\begin{align*}\int_{R_1}^{R_2}\frac{\partial_{r}(\omega^2)}{  b^2} |\tilde{\varphi}(r)|^2dr\geq-\epsilon^2\int_{R_1}^{R_2}\left(\frac{\partial_r^2b}{rb}-\frac{\partial_rb}{r^2b}\right)|\tilde{\varphi}(r)|^2dr, \forall  \tilde{\varphi}(r)\in H^r_{mag},
\end{align*}
which implies \eqref{L1hh}.
Second, for the case ii) $R_1=0$ and $R_2=\infty, $
let $\chi(r)\in C_0^{\infty}((R_1,R_2);[0,1])$ be a cut-off function satisfying
$\chi(r)=1$ if $r\leq\frac{1}{2}$ and $\chi(r)=0$ if $r\geq1.$
For any $M>2$, we
define $\chi_M=\chi(\frac{r}{M}).$
By the same argument as in the proof of Lemma \ref{embedd},
one has
\begin{align*}
\epsilon^2\int_{0}^{\frac{2}{M}}\left|\frac{\partial_r^2b}{rb}-\frac{\partial_rb}{r^2b}\right||\tilde{\varphi}(r)\chi_{M}\chi_{\frac{2}{M}}|^2dr\leq
\frac{\epsilon^2C}{M^\beta}\int_{0}^{\frac{2}{M}}\frac{1}{r} |\partial_r\tilde{\varphi}(r)|^2dr, \forall  \tilde{\varphi}(r)\in H^r_{mag}
\end{align*}
and
\begin{align*}
\epsilon^2\int_{\frac{M}{2}}^{\infty}\left|\frac{\partial_r^2b}{rb}-\frac{\partial_rb}{r^2b}\right||\tilde{\varphi}(r)(1-\chi_{M})|^2dr\leq
\frac{\epsilon^2C}{M^{2\alpha}}\int_{\frac{M}{2}}^{\infty}\frac{1}{r} |\partial_r\tilde{\varphi}(r)|^2dr, \forall  \tilde{\varphi}(r)\in H^r_{mag},
\end{align*}
for large $M,$ where $C$ is a constant 
 independent of $M$.
For sufficiently large $M>\max\{3,C^{\frac{1}{\beta}}+1,C^{\frac{1}{2\alpha}}+1\}$ and $\epsilon^2$ small enough such that
$$\epsilon^2<\inf\limits_{r\in[\frac{1}{M},M]}\frac{\partial_{r}(\omega^2)}{  b^2} \left|\frac{\partial_r^2b}{rb}-\frac{\partial_rb}{r^2b}\right|^{-1},$$
we have
\begin{align*}
&\int_{R_1}^{R_2}\frac{\partial_{r}(\omega^2)}{  b^2} |\tilde{\varphi}(r)|^2dr
+\epsilon^2\int_{R_1}^{R_2}\left[\frac{1}{r}|\tilde{\varphi}(r)|^2+\frac{1}{r}|\partial_r \tilde{\varphi}(r)|^2+ \left(\frac{\partial_r^2b}{rb}-\frac{\partial_rb}{r^2b}\right)|\tilde{\varphi}(r)|^2\right]dr\\
&\geq \int_{\frac{1}{M}}^{M}\frac{\partial_{r}(\omega^2)}{  b^2} |\tilde{\varphi}(r)|^2dr+\int_{\frac{1}{M}}^{M}\epsilon^2\left(\frac{\partial_r^2b}{rb}-\frac{\partial_rb}{r^2b}\right)|\tilde{\varphi}(r)|^2dr
\nonumber\\
&\quad+\epsilon^2\left(1-\frac{C}{M^\beta}\right)\int_{0}^{\frac{2}{M}}\frac{1}{r}|\partial_r \tilde{\varphi}(r)|^2dr+\epsilon^2\left(1-\frac{C}{M^{2\alpha}}\right)\int_{\frac{2}{M}}^{\infty}\frac{1}{r}|\partial_r \tilde{\varphi}(r)|^2dr\geq0.
\end{align*}
Thus, if $\partial_r(\omega^2)>0$, and $\epsilon^2$ is small enough, \eqref{L1hh} holds, i.e., the steady state $(v_0,H_0)(x)$ in \eqref{steady1} is linearly stable to axisymmetric perturbations. For the case $R_1=0$, $R_2<\infty $ or $R_1>0$, $R_2=\infty $, the proof is similar, so we omit it.

If $\partial_r(\omega^2)|_{r=r_0}<0$ for some $r_0\in (R_1,R_2)$,
then $\partial_r(\omega^2)<0$ for all $r\in (r_0-\varepsilon,r_0+\varepsilon)$ for some $\varepsilon>0$ small enough.
Let $\xi(r)\in C_0^\infty(r_0-\varepsilon,r_0+\varepsilon)$ and $\xi(r_0)=1$, if $$\epsilon^2<-\frac{\int_{r_0-\varepsilon}^{r_0+\varepsilon}
\frac{\partial_r(\omega^2)}{ b(r)^2}|\xi|^2dr}{\int_{r_0-\varepsilon}^{r_0+\varepsilon}\left(\frac{1}{r}| \xi'|^2+\frac{1}{r}|\xi|^2+ \left|\frac{\partial_r^2b(r)}{rb(r)}|\xi|^2-\frac{\partial_rb(r)}{r^2b(r)}|\xi|^2\right|\right)dr},$$ then we have
\begin{align*}
\langle\mathbb{L}_1 \xi,\xi\rangle
&=\int_{r_0-\varepsilon}^{r_0+\varepsilon}\left[\frac{1}{r}| \xi'|^2+\frac{1}{r}|\xi|^2+
\frac{\partial_r(\omega^2)}{ \epsilon^2b(r)^2}|\xi|^2+ \left(\frac{\partial_r^2b(r)}{rb(r)}|\xi|^2-\frac{\partial_rb(r)}{r^2b(r)}|\xi|^2\right)\right]dr\\
&\leq\int_{r_0-\varepsilon}^{r_0+\varepsilon}\left(\frac{1}{r}| \xi'|^2+\frac{1}{r}|\xi|^2+
\frac{\partial_r(\omega^2)}{ \epsilon^2b(r)^2}|\xi|^2+ \left|\frac{\partial_r^2b(r)}{rb(r)}|\xi|^2-\frac{\partial_rb(r)}{r^2b(r)}|\xi|^2\right|\right)dr<0.
\end{align*}
Thus, $n^{-}(\mathbb{L}_1)>0$ and there is linear instability.
\end{proof}

\begin{proof}[Proof of Corollary \ref{BHcorolly2}]
i) If
$
\widehat{L}=-\frac{1}{r}\partial_r\left(\frac{1}{r}\partial_r\cdot\right)
+\frac{1}{r^2}+\left(\frac{\partial_r^2b}{r^2b}-\frac{\partial_rb}{r^3b}\right)>0,
$
then by \eqref{L1hh}, a necessary and sufficient condition for the linearly stability is that for any $\tilde{\varphi}(r)\in H^r_{mag}$,
\begin{align*}
\epsilon^2&\geq-\frac{\int_{R_1}^{R_2}\frac{\partial_{r}(\omega^2)}{  b^2} |\tilde{\varphi}(r)|^2dr}{\langle \widehat{L}\tilde{\varphi}, \tilde{\varphi} \rangle},
\end{align*}
i.e.,
\begin{align*}
\epsilon^2&\geq \epsilon^2_{min}=\max\left\{\sup\limits_{\tilde{\varphi}(r)\in H^r_{mag}}\frac{-\int_{R_1}^{R_2}\frac{\partial_r(\omega^2)}{b^2}|\tilde{\varphi}|^2dr}{\langle \widehat{L}\tilde{\varphi}, \tilde{\varphi} \rangle},0\right\}.
\end{align*}
In the remaining of the proof, we check
the supremum  in the definition of $\epsilon^2_{min}$ can be reached.

If $\partial_r(\omega^2)\geq0$ for $r\in(R_1,R_2)$, then $\epsilon_{min}=0$. Considering the case there exists $r_0\in (R_1,R_2)$ such that $\partial_r(\omega^2)|_{r=r_0}<0,$
then
\begin{align*}
\epsilon^2_{min}&=\sup\limits_{\tilde{\varphi}(r)\in H^r_{mag}}\frac{-\int_{R_1}^{R_2}\frac{\partial_r(\omega^2)}{b^2}|\tilde{\varphi}|^2dr}{\langle \widehat{L}\tilde{\varphi}, \tilde{\varphi} \rangle}>0,
\end{align*}
i.e.,
\begin{align}\label{maxmizer}
1=\sup\limits_{\tilde{\varphi}(r)\in H^r_{mag}}\frac{-\int_{R_1}^{R_2}\frac{\partial_r(\omega^2)}{b^2}|\tilde{\varphi}|^2dr
-\epsilon^2_{min}\int_{R_1}^{R_2}\left(\frac{\partial_r^2b}{rb}-\frac{\partial_rb}{r^2b}\right)|\tilde{\varphi}|^2dr}
{\epsilon^2_{min}\int_{R_1}^{R_2}\frac{1}{r}|\partial_r\tilde{\varphi}|^2+\frac{1}{r}|\tilde{\varphi}|^2dr},
\end{align}
equivalently,
\begin{align*}
\sup\limits_{\tilde{\varphi}(r)\in H^r_{mag}}\int_{R_1}^{R_2}-\frac{\partial_r(\omega^2)}{b^2}|\tilde{\varphi}|^2dr
-\epsilon^2_{min}\int_{R_1}^{R_2}\left(\frac{\partial_r^2b}{rb}-\frac{\partial_rb}{r^2b}\right)|\tilde{\varphi}|^2dr=1,
\end{align*}
with the constraint
\begin{align*}
\epsilon^2_{min}\int_{R_1}^{R_2}\frac{1}{r}|\partial_r\tilde{\varphi}|^2+\frac{1}{r}|\tilde{\varphi}|^2dr=1.
\end{align*}
 Then we can take a maximizing sequence $\tilde{\varphi}_n(r)\in H^r_{mag}$ satisfying
\begin{align*}
\epsilon^2_{min}\int_{R_1}^{R_2}\left(\frac{1}{r}|\partial_r\tilde{\varphi}_n|^2+\frac{1}{r}|\tilde{\varphi}_n|^2\right)dr=1
\end{align*}
and
\begin{align*}
\lim_{n\rightarrow\infty}\int_{R_1}^{R_2}-\frac{\partial_r(\omega^2)}{b^2}|\tilde{\varphi}_n|^2dr
-\epsilon^2_{min}\int_{R_1}^{R_2}\left(\frac{\partial_r^2b}{rb}-\frac{\partial_rb}{r^2b}\right)|\tilde{\varphi}_n|^2dr
=1.
\end{align*}
Note that there exists $\varphi^{*}(r)\in H^r_{mag}$ and a subsequence of $\tilde{\varphi}_n\in H^r_{mag}$, for convenience, we still denote this subsequence by $\tilde{\varphi}_n$, such that $\tilde{\varphi}_n\rightharpoonup \varphi^{*}$ in $H^r_{mag}$.
By the compactness of $H^r_{mag}(R_1,R_2)\hookrightarrow L^2_{r|\mathfrak{F}(r)|}(R_1,R_2)$,
it follows that
\begin{align*}
\int_{R_1}^{R_2}-\frac{\partial_r(\omega^2)}{b^2}|\varphi^{*}|^2dr
-\epsilon^2_{min}\int_{R_1}^{R_2}\left(\frac{\partial_r^2b}{rb}-\frac{\partial_rb}{r^2b}\right)|\varphi^{*}|^2dr=1.
\end{align*}
Next, 
\begin{align*}
&\epsilon^2_{min}\int_{R_1}^{R_2}\left(\frac{1}{r}|\partial_r\varphi^{*}|^2+\frac{1}{r}|\varphi^{*}|^2\right)dr
\leq\liminf_{n}\epsilon^2_{min}\int_{R_1}^{R_2}\left(\frac{1}{r}|\partial_r\tilde{\varphi}_n|^2+\frac{1}{r}|\tilde{\varphi}_n|^2\right)
dr=1.
\end{align*}
From the definition of $\epsilon^2_{min}$, we get
\begin{align*}
\frac{-\int_{R_1}^{R_2}\frac{\partial_r(\omega^2)}{b^2}|\varphi^{*}|^2dr
-\epsilon^2_{min}\int_{R_1}^{R_2}\left(\frac{\partial_r^2b}{rb}-\frac{\partial_rb}{r^2b}\right)|\varphi^{*}|^2dr}
{\epsilon^2_{min}\int_{R_1}^{R_2}\left(\frac{1}{r}|\partial_r\varphi^{*}|^2+\frac{1}{r}|\varphi^{*}|^2\right)dr}\leq1,
\end{align*}
which implies $\epsilon^2_{min}\int_{R_1}^{R_2}\left(\frac{1}{r}|\partial_r\varphi^{*}|^2+\frac{1}{r}|\varphi^{*}|^2\right)dr\geq1.$ Thus, $$\epsilon^2_{min}\int_{R_1}^{R_2}\left(\frac{1}{r}|\partial_r\varphi^{*}|^2+\frac{1}{r}|\varphi^{*}|^2\right)dr=1$$ and $\varphi^{*}$ is the maximizer of \eqref{maxmizer}. Therefore, we complete the proof of Corollary \ref{BHcorolly2} i).

ii)
If
$
\partial_{r}(\omega^2)>0,
$
by \eqref{L1hh}, a necessary and sufficient condition for the linearly stability is,
 for any $\tilde{\varphi}(r)\in H^r_{mag}$,
\begin{align*}
\frac{1}{\epsilon^2}&\geq-\frac{\langle \widehat{L}\tilde{\varphi}, \tilde{\varphi} \rangle}{\int_{R_1}^{R_2}\frac{\partial_{r}(\omega^2)}{  b^2} |\tilde{\varphi}|^2dr},
\end{align*}
i.e.,
\begin{align*}
\frac{1}{\epsilon^2}\geq \frac{1}{\epsilon^2_{\max}}&=\sup\limits_{\tilde{\varphi}(r)\in H^r_{mag}}\frac{-\langle \widehat{L}\tilde{\varphi}, \tilde{\varphi} \rangle}{\int_{R_1}^{R_2}\frac{\partial_r(\omega^2)}{b^2}|\tilde{\varphi}|^2dr},
\end{align*}
and hence, a necessary and sufficient condition for the linearly stability is $\epsilon^2\leq\epsilon^2_{\max}$.
We can check
the supremum in the definition of $\epsilon^2_{\max}$ can be reached similar to the proof of Corollary \ref{BHcorolly2} i).

iii)
By Corollary \ref{BHcorolly}, if $n^{-}(\widehat{L})>0$ and $\partial_r(\omega^2)$ changes sign, then for $\epsilon$ sufficiently small,
the steady state $(v_0,H_0)(x)$ in \eqref{steady1} is spectrally unstable to axisymmetric perturbations.
Let $h$ be a negative direction of $\widehat{L}$,
then
\begin{align*}
\langle \mathbb{L}_1h, h \rangle&=\langle \widehat{L}h, h\rangle+\frac{1}{\epsilon^2}\int_{R_1}^{R_2}\frac{\partial_r(\omega^2)}{ b^2r}h^2  rdr
\rightarrow \langle \widehat{L}h, h\rangle<0,
\end{align*}
as $\epsilon\rightarrow\infty.$ Therefore, the steady state $(v_0,H_0)(x)$ in \eqref{steady1} is spectrally unstable to axisymmetric perturbations
if $\epsilon$ is sufficiently large.
\end{proof}

\begin{proof}[Proof of Corollary \ref{BHcorolly3}]
i) It is clear from \eqref{L1hh}.

ii)
To denote the dependence of $\mathbb{L}_{1}$ on $\epsilon$, we use the
notation
\[
\mathbb{L}_{1}^{\epsilon}:=-\frac{1}{r}\partial_{r}\left(  \frac{1}{r}%
\partial_{r}\cdot\right)  +\frac{1}{r^{2}}+\frac{\partial_{r}(\omega^{2}%
)}{\epsilon^{2}b(r)^{2}r}+\left(  \frac{\partial_{r}^{2}b(r)}{r^{2}b(r)}%
-\frac{\partial_{r}b(r)}{r^{3}b(r)}\right)  .
\]
Then we have
\[
\lim_{\epsilon\rightarrow\infty}\|\mathbb{L}_{1}^{\epsilon}-\widehat{L}\|=0.
\]
Consider the eigenvalue problem
\begin{align}\label{eigenvalue}
-\frac{1}{r}\partial_{r}\left(  \frac{1}{r}\partial_{r}\tilde{\varphi}_{\epsilon}\right)
+\frac{1}{r^{2}}\tilde{\varphi}_{\epsilon}+\lambda_{\epsilon}\left[  \frac{\partial
_{r}(\omega^{2})}{\epsilon^{2}b(r)^{2}r}+\left(  \frac{\partial_{r}^{2}%
b(r)}{r^{2}b(r)}-\frac{\partial_{r}b(r)}{r^{3}b(r)}\right)  \right]
\tilde{\varphi}_{\epsilon}=0.
\end{align}
Here, $\lambda_{\epsilon}=1$ ($>1$ or $<1$ respectively)\ corresponds to a
zero (positive or negative respectively) direction of the quadratic form
$\left\langle \mathbb{L}_{1}^{\epsilon}\cdot,\cdot\right\rangle $.
Then we have
\[
\mathbb{L}_{1}^{\epsilon}\tilde{\varphi}_{\epsilon}+(\lambda_{\epsilon}-1)\left[  \frac{\partial
_{r}(\omega^{2})}{\epsilon^{2}b(r)^{2}r}+\left(  \frac{\partial_{r}^{2}%
b(r)}{r^{2}b(r)}-\frac{\partial_{r}b(r)}{r^{3}b(r)}\right)  \right]
\tilde{\varphi}_{\epsilon}=0,
\]
which implies
\begin{align*}
\widehat{L}\tilde{\varphi}_{\epsilon}+(\mathbb{L}_{1}^{\epsilon}-\widehat{L})\tilde{\varphi}_{\epsilon}+(\lambda_{\epsilon}-1)\left[  \frac{\partial
_{r}(\omega^{2})}{\epsilon^{2}b(r)^{2}r}+\left(  \frac{\partial_{r}^{2}%
b(r)}{r^{2}b(r)}-\frac{\partial_{r}b(r)}{r^{3}b(r)}\right)  \right]
\tilde{\varphi}_{\epsilon}=0.
\end{align*}
Now we choose
$\tilde{\varphi}_{\epsilon}$ by the normalization $$\int_{R_1}^{R_2}\left[\left|  \frac{\partial
_{r}(\omega^{2})}{\epsilon^{2}b(r)^{2}}\right|+\left|\left(  \frac{\partial_{r}^{2}%
b(r)}{rb(r)}-\frac{\partial_{r}b(r)}{r^{2}b(r)}\right)  \right|\right]|\tilde{\varphi}_{\epsilon}|^2dr=1.$$ Then it follows from \eqref{eigenvalue} that
\begin{align*}
\|\tilde{\varphi}_{\epsilon}\|_{H^r_{mag}}^2=-\lambda_{\epsilon}\int_{R_1}^{R_2} \left[\frac{\partial
_{r}(\omega^{2})}{\epsilon^{2}b(r)^{2}}+\left(  \frac{\partial_{r}^{2}
b(r)}{rb(r)}-\frac{\partial_{r}b(r)}{r^{2}b(r)}\right)\right]  |\tilde{\varphi}_{\epsilon}|^2dr\leq C,
\end{align*}
and consequently there exists $\tilde{\varphi}_{\infty}\in H^r_{mag}$ such that
\begin{align*}
\tilde{\varphi}_{\epsilon}\rightharpoonup \tilde{\varphi}_{\infty} \quad \text{in} \quad{H^r_{mag}} \quad \text{as }\epsilon\rightarrow\infty.
\end{align*}
By the compactness of $H^r_{mag}\hookrightarrow L^2_{\left|  \frac{\partial
_{r}(\omega^{2})}{\epsilon^{2}b(r)^{2}}\right|+\left|\left(  \frac{\partial_{r}^{2}%
b(r)}{rb(r)}-\frac{\partial_{r}b(r)}{r^{2}b(r)}\right)  \right|}$, we have
\begin{align*}\int_{R_1}^{R_2}\left[\left|  \frac{\partial
_{r}(\omega^{2})}{\epsilon^{2}b(r)^{2}}\right|+\left|\left(  \frac{\partial_{r}^{2}%
b(r)}{rb(r)}-\frac{\partial_{r}b(r)}{r^{2}b(r)}\right)  \right|\right]|\tilde{\varphi}_{\infty}|^2dr=1.
\end{align*}
So $\tilde{\varphi}_{\infty}\neq0.$ By taking the limit of \eqref{eigenvalue} as $\epsilon\rightarrow\infty,$ $\widehat{L}\tilde{\varphi}_{\infty}=0$. Thus $\tilde{\varphi}_{\infty}=c\hat{\varphi}$, with $c=\left(\int_{R_1}^{R_2}\left[\left|  \frac{\partial
_{r}(\omega^{2})}{\epsilon^{2}b(r)^{2}}\right|+\left|\left(  \frac{\partial_{r}^{2}%
b(r)}{rb(r)}-\frac{\partial_{r}b(r)}{r^{2}b(r)}\right)  \right|\right]|\hat{\varphi}|^2dr\right)^{-\frac{1}{2}}.$
Multiplying \eqref{eigenvalue} by $\epsilon^2r\hat{\varphi}$, and integrating on $(R_1,R_2)$, one has
\begin{align*}
&\epsilon^2\langle\widehat{L}\tilde{\varphi}_{\epsilon}, \hat{\varphi}\rangle+\int_{R_1}^{R_2}\lambda_{\epsilon}\frac{\partial
_{r}(\omega^{2})}{b(r)^{2}r}\tilde{\varphi}_{\epsilon} r\hat{\varphi}dr
+\int_{R_1}^{R_2}\epsilon^2(\lambda_{\epsilon}-1)\left(  \frac{\partial_{r}^{2}%
b(r)}{r^{2}b(r)}-\frac{\partial_{r}b(r)}{r^{3}b(r)}\right)
\tilde{\varphi}_{\epsilon} r\hat{\varphi} dr=0.
\end{align*}
Letting $\epsilon\rightarrow\infty,$ by $\langle\widehat{L}\hat{\varphi},\hat{\varphi} \rangle=0$ and $\lim\limits_{\epsilon\rightarrow\infty}\lambda_{\epsilon}=1,$ it follows that
\begin{align*}
\lim_{\epsilon\rightarrow\infty}\epsilon^2(\lambda_{\epsilon}-1)=\frac{\int_{R_{1}}^{R_{2}}\frac{\partial_{r}(\omega^{2})}{b^{2}}%
|\hat{\varphi}|^{2}dr}{-\int_{R_{1}}^{R_{2}}\left(  \frac{\partial_{r}^{2}%
b(r)}{rb(r)}-\frac{\partial_{r}b(r)}{r^{2}b(r)}\right)|\hat{\varphi}|^{2}dr}=\frac{\int_{R_{1}}^{R_{2}}\frac{\partial_{r}(\omega^{2})}{b^{2}}%
|\hat{\varphi}|^{2}dr}{\int_{R_{1}}^{R_{2}}\left(\frac{1}{r}|\partial_{r}\hat{\varphi}
|^{2}+\frac{1}{r}|\hat{\varphi}|^{2}\right)dr}.
\end{align*}
Consequently, when $\epsilon$ is large enough, the kernel of $\widehat{L}$ is
perturbed to be a positive (negative) direction of $\mathbb{L}_{1}^{\epsilon}$
when $\int_{R_{1}}^{R_{2}}\frac{\partial_{r}(\omega^{2})}{b^{2}}|\hat{\varphi}|^{2}dr>0$ ($<0$).
Therefore, we complete the proof of Corollary \ref{BHcorolly3} ii).

\end{proof}

\section{Nonlinear stability}\label{nonlinearstability}

In the section, we show the conditional nonlinear stability under the linear stability condition $\mathbb{L}_1>0$.
We impose the system \eqref{nonlinearMHDrz} with the initial and boundary conditions
\begin{align*}
\begin{cases}
(v_r, v_\theta,v_z, \psi, H_\theta)(0,r,z)=(v_r^0, v_\theta^0,v_z^0, \psi^0, H_\theta^0)(r,z),\\
 v_r(t,R_1,z)=0,\psi(t,R_1,z)=0 ,\\
 v_r(t,R_2,z)=0,\psi(t,R_2,z)=-\epsilon\int_{R_1}^{R_2}rH_z(0,r,z)dr,\quad  \mbox{if}\quad R_2<\infty ;\\
(v_r,v_\theta,v_z, \psi, H_\theta)(t,r,z)\to (0,r\omega(r),0,\psi_0,0) \quad \mbox{as} \quad r\to \infty, \quad \mbox{if}\quad R_2=\infty;\\
(v_r,v_\theta,v_z, \psi, H_\theta)(t,r,z)=(v_r, v_\theta,v_z, \psi, H_\theta)(t,r,z+2\pi).
\end{cases}
\end{align*}
Without loss of generality, we can take the initial magnetic potential $\psi^0(r,z)$ satisfying
\begin{align*}
\begin{cases}
\psi^0(R_1,z)=0, \\
 \psi^0(R_2,z)=-\epsilon\int_{R_1}^{R_2}sb(s)ds=\psi_0(R_2).
\end{cases}
\end{align*}
Indeed, for general perturbations with $\psi^0(r,z)$ sufficiently close to $\psi_0(r)$ in $H^{1}_{mag},$ we can find $\tilde{\epsilon}$ near $\epsilon$ such that $\psi^0(R_2,z)=-\tilde{\epsilon}\int_{R_1}^{R_2}sb(s)ds=\tilde{\psi}_0(R_2)$, where $\tilde{\psi}_0(r)=-\tilde{\epsilon}\int_{R_1}^{r}sb(s)ds$, and $\mathbb{L}_1^{\tilde{\epsilon}}>0$ by the continuity of $\mathbb{L}_1^{\epsilon}$ to $\epsilon$. Then by the nonlinear stability of $(v_0,0,\tilde{\psi}_0)$ and the triangle inequality, we have
\begin{align*}
d(t)&=d((v,H_\theta,\psi)(t,x),(v_0,0,\psi_0)(x))\\
&\leq d((v,H_\theta,\psi)(t,x),(v_0,0,\tilde{\psi_0})(x))+d((v_0,0,\psi_0)(x),(v_0,0,\tilde{\psi}_0)(x))\lesssim d(0).
\end{align*}

 To this end, we introduce
the definition of a weak solution to MHD equations \eqref{EPM}.

\begin{definition}\label{weaksolution}
We call $(v, H) \in C((0 , T ); L^2)$ is an axisymmetric weak
solution of the ideal MHD system \eqref{EPM} with initial-boundary conditions \eqref{1.2'}-\eqref{bdc}, if for any $t \in (0 , T )$ the vector fields
$(v, H)(t, \cdot)$ are divergence free in the sense of distributions, \eqref{EPM} holds in the sense of distributions, i.e.,
\begin{align*}
\int_0^t\int_{\Omega}&\left[\partial_t \xi \cdot v+\nabla\xi:(v\otimes v+H\otimes H)\right] dxdt
=\int_{\Omega} \xi(t,x) \cdot v(t,x) dx-\int_{\Omega} \xi(0,x) \cdot v(0,x) dx,\\
\int_0^t\int_{\Omega}&\left[\partial_t \xi \cdot H+\nabla\xi:(v\otimes H+H\otimes v)\right] dxdt=\int_{\Omega} \xi(t,x) \cdot H(t,x) dx-\int_{\Omega} \xi(0,x) \cdot H(0,x) dx
\end{align*}
hold for all divergence free test functions $\xi \in C_0^{\infty} ((0 , T ) \times \Omega)$.
\end{definition}

Here, the existence of global weak solutions of the MHD equations is assumed in the proof of nonlinear stability.
In the following, we will consider the nonlinear stability for MHD equations in the bounded domain $\Omega=\{x\in\mathbb{R}^3|R_1\leq r=\sqrt{x_1^2+x_2^2}< R_2 , x_3\in\mathbb{T}_{2\pi}\}$. Without loss of generality, we take $R_1=0$ and $R_2=1$. For unbounded case, we give the details of nonlinear stability in the appendix.

We define the total energy by
\begin{align}\label{energy}
E(t)&=E(v,H_\theta,\psi)
=\frac{1}{2}\int_{\Omega}\left(|v_r|^2+|v_z|^2+|H_\theta|^2+|v_\theta|^2+\frac{1}{r^2}|\partial_z\psi|^2+\frac{1}{r^2}|\partial_r\psi|^2\right)
 dx.
\end{align}
It is conserved for strong solutions of \eqref{nonlinearMHDrz}. We assume $E(t)$ is non-increasing for the weak solution.

\begin{lemma} 
\label{fpsicon}
For any $\Phi(\tau),f(\tau)\in C^1(\mathbb{R})$, 
the functionals $\int_{\Omega}f(\psi) dx$
and $\int_{\Omega}rv_{\theta} \Phi(\psi) dx$ are conserved for strong solutions of MHD equations \eqref{nonlinearMHDrz}.
\end{lemma}
\begin{proof}
Define the Poisson bracket $\{f,g\}=\frac{1}{r}\left(\partial_rf\partial_zg-\partial_z f\partial_rg\right)$ and the function $\varpi$ satisfies $-\frac{\partial_r\varpi}{r}=v_z$,
 $\frac{\partial_z\varpi}{r}=v_r$, then $\eqref{nonlinearMHDrz}_2$ and $\eqref{nonlinearMHDrz}_4$ can be rewritten as
\begin{align}\label{PoissonBracket}
\begin{cases}
\partial_t(rv_\theta)=\{\varpi,rv_\theta\}+\{rH_\theta,\psi\},\\
\partial_t \psi=\{\varpi,\psi\}.
\end{cases}
\end{align}
 By \eqref{PoissonBracket}, we have
\begin{align*}
\frac{d}{d t}\int_{\Omega}f(\psi) dx&=\int_{\Omega}f^{'}(\psi)\partial_t\psi dx
=2\pi\int_{0}^{2\pi}\int_{0}^{1}f^{'}(\psi)\{\varpi,\psi\}rdrdz\\
&=2\pi\int_{0}^{2\pi}\int_{0}^{1}f^{'}(\psi)\left(\partial_r\varpi\partial_z\psi-\partial_z \varpi\partial_r\psi\right)drdz
=0.
\end{align*}
 Similarly, by \eqref{PoissonBracket}, we obtain
\begin{align*}
&\frac{d}{d t}\int_{\Omega}rv_\theta\Phi(\psi) dx
=\int_{\Omega}\partial_t(rv_\theta)\Phi(\psi)+rv_\theta\Phi^{'}(\psi)\partial_t\psi dx\\
&=2\pi\int_{0}^{2\pi}\int_{0}^{1}\left[\Phi(\psi)\{\varpi,rv_\theta\}+\Phi(\psi)\{rH_\theta,\psi\}+rv_\theta\Phi^{'}(\psi)\{\varpi,\psi\}\right] rdrdz\\
&=2\pi\int_{0}^{2\pi}\int_{0}^{1}\left[\Phi(\psi)\{\varpi,rv_\theta\}+rH_\theta\{\Phi(\psi),\psi\}+rv_\theta\{\varpi,\Phi(\psi)\}\right] rdrdz\\
&=2\pi\int_{0}^{2\pi}\int_{0}^{1}\left[\Phi(\psi)\{\varpi,rv_\theta\}+\Phi(\psi)\{rv_\theta,\varpi\}\right] rdrdz=0.
\end{align*}

\end{proof}

Since the steady solution $\psi_0$ satisfies $\frac{-\partial_r\psi_0}{r}= \epsilon b(r),$
we define the Casimir-energy functional \cite{MR794110}
\begin{align*}
E_{c}(v,H_\theta,\psi)=E(v,H_\theta,\psi)+\int_{\Omega}rv_\theta\Phi(\psi)dx+\int_{\Omega}f(\psi) dx,
\end{align*}
where
\begin{align}\label{extension1}
\Phi(\tau)=\begin{cases}
 -\omega(h^{-1}(\tau)),\quad  \tau\in (-\epsilon\int_0^1sb(s)ds,0) ,\\
\text{some $C_0^\infty$ extension function, }  \tau\notin (-\epsilon\int_0^1sb(s)ds,0)
\end{cases}
\end{align}
and
\begin{align}\label{extension2}
f^{'}(\tau)=\begin{cases}
 \frac{-\epsilon b^{'}(h^{-1}(\tau))}{h^{-1}(\tau)}+\frac{(h^{-1}(\tau))^2}{2}\partial_\tau(\omega^2(h^{-1}(\tau))),  \tau\in (-\epsilon\int_0^1sb(s)ds,0) ,\\
\text{some $C_0^\infty$ extension function, }  \tau\notin (-\epsilon\int_0^1sb(s)ds,0),
\end{cases}
\end{align}
with  $\tau=h(r)=-\epsilon\int_0^rsb(s)ds$.
Here, $\Phi(\tau)$ and $f^{'}(\tau)$ are well-defined in $\mathbb{R}$.

A natural method to prove nonlinear stability is to do the Taylor expansion of $E_c(v,H_\theta,\psi)$ at $(v_0,0,\psi_0)$. For this, we first show that $E_c(v,H_\theta,\psi)$ is $C^2$.
\begin{lemma}\label{Fderivate}
1) For any $f'(\tau)\in C_0^2(\mathbb{R})$  and satisfies \eqref{order5bound}-\eqref{order7bound}, then $F_1(\psi):=\int_\Omega f(\psi)dx \in C^2(H^1_{mag})$.\\
2) For any $\Phi(\tau)\in C_0^3(\mathbb{R})$ satisfies \eqref{order1bound}-\eqref{order4bound}, then $F_2(v_\theta,\psi):=\int_\Omega rv_\theta \Phi(\psi)dx \in C^2(L^2\times H^1_{mag})$.
\end{lemma}
\begin{proof}
The proof of Lemma \ref{Fderivate} is given in Appendix.
\end{proof}

By Lemma \ref{Fderivate}, we have
\begin{align}\label{secondorderF1}
\notag&F_1(\psi)-F_1(\psi_0)-F_1'(\psi_0)(\psi-\psi_0)-\frac{1}{2}\langle F_1''(\psi_0)(\psi-\psi_0),(\psi-\psi_0)\rangle\\
&=\int_{\Omega}[f(\psi)-f(\psi_0)-f^{'}(\psi_0)(\psi-\psi_0)-\frac{1}{2}f^{''}(\psi_0)(\psi-\psi_0)^2] dx\\
\notag&=o(\|\psi-\psi_0\|^2_{H^1_{mag}})
\end{align}
and
\begin{align}\label{secondorderF2}
\notag&F_2(v_\theta,\psi)-F_2(\mathbf{v}_0,\psi_0)-\partial_{\psi}F_2(\mathbf{v}_0,\psi_0)(\psi-\psi_0)-\partial_{v_\theta}F_2(\mathbf{v}_0,\psi_0)(v_\theta-\mathbf{v}_0)\\
\notag&-\langle\partial_{\psi v_\theta}F_2(\mathbf{v}_0,\psi_0)(\psi-\psi_0),(v_\theta-\mathbf{v}_0)\rangle-\frac{1}{2}\langle\partial_{\psi\psi}F_2(\mathbf{v}_0,\psi_0)(\psi-\psi_0),(\psi-\psi_0)\rangle\\
\notag&=\int_{\Omega}r(v_\theta-\mathbf{v}_0)[\Phi(\psi)-\Phi(\psi_0)-\Phi^{'}(\psi_0)(\psi-\psi_0)]dx
\\&\quad+\int_{\Omega}r\mathbf{v}_0[\Phi(\psi)-\Phi(\psi_0)-\Phi^{'}(\psi_0)(\psi-\psi_0)-\frac{1}{2}\Phi^{''}(\psi_0)(\psi-\psi_0)^2]dx\\
\notag&=o(\|\psi-\psi_0\|^2_{H^1_{mag}})+o(\|v_\theta-\mathbf{v}_0\|^2_{L^2}).
\end{align}
The steady state $(v_0,0,\psi_0)$ has the following variational structure.
First, $(v_0=\omega r\mathbf{e}_\theta,0,\psi_0)$ is a critical point of $E_c$, i.e.,
\begin{align}\label{firstvariational}
&\langle \nabla E_c(v_0,0,\psi_0),(\delta v,\delta H_\theta,\delta\psi)\rangle\\
\notag&=\int_\Omega\left[(r\omega+r\Phi(\psi_0))\delta v_\theta
+\left(\omega r^2\Phi'(\psi_0)+f'(\psi_0)+\epsilon\frac{\partial_rb}{r}\right)\delta\psi\right] dx=0,
\end{align}
by \eqref{extension1}-\eqref{extension2}, where $(\delta v,\delta H_\theta,\delta\psi):=(v,H_{\theta},\psi)-(v_{0},0,\psi_{0})$.

From \eqref{extension1} and \eqref{extension2}, it follows that
\begin{align}\label{secondsteadyeqbyE1}
r\Phi^{'}(\psi_0)=\frac{\partial_r\omega}{ \epsilon  b}
\end{align}
and
\begin{align}\label{secondsteadyeqbyE2}
r\mathbf{v}_0\Phi^{''}(\psi_0)+f^{''}(\psi_0)
=
\frac{(\partial_r\omega)^2}{ \epsilon^2 b^2}
+\frac{\partial_r(\omega^2)}{ \epsilon^2 b^2r}+ \left(\frac{\partial_r^2b}{r^2b}-\frac{\partial_rb}{r^3b}\right) .
\end{align}
By \eqref{secondsteadyeqbyE1}-\eqref{secondsteadyeqbyE2}, we can write the second-order variation of $E_c(v,H_\theta,\psi)$
at $(v_0,0,\psi_0)$ as
\begin{align}\label{secondvariational}
&\langle \nabla^2E_c(v_0,0,\psi_0)(\delta v,\delta H_\theta,\delta\psi),(\delta v,\delta H_\theta,\delta\psi)\rangle
\notag\\
&=\int_{\Omega}\left(|\delta v_r|^2+|\delta v_z|^2+|\delta H_\theta|^2+|\delta v_\theta|^2+\frac{1}{r^2}|\partial_z\delta\psi|^2+\frac{1}{r^2}|\partial_r\delta\psi|^2\right)dx
\notag\\&\quad+\int_{\Omega}(\omega r^2\Phi''(\psi_0)+f''(\psi_0))|\delta\psi|^2dx +\int_{\Omega}2r\Phi'(\psi_0)\delta\psi\delta v_\theta dx
\notag\\&=\int_{\Omega}\left(|\delta v_r|^2+|\delta v_z|^2+|\delta H_\theta|^2+|\delta v_\theta|^2+\frac{1}{r^2}|\partial_z\delta\psi|^2+\frac{1}{r^2}|\partial_r\delta\psi|^2\right)dx
\notag\\
&\quad+\int_{\Omega}\left[\frac{(\partial_r\omega)^2}{ \epsilon^2 b^2}dx
+\frac{\partial_r(\omega^2)}{ \epsilon^2 b^2r}+ \left(\frac{\partial_r^2b}{r^2b}-\frac{\partial_rb}{r^3b}\right)\right]|\delta\psi|^2dx +\int_{\Omega}2\frac{\partial_r\omega}{ \epsilon  b}\delta\psi\delta v_\theta dx
\notag\\
&=\int_{\Omega}\left(|\delta v_r|^2+|\delta v_z|^2+|\delta H_\theta|^2+\frac{1}{r^2}|\partial_z\delta\psi|^2+\frac{1}{r^2}|\partial_r\delta\psi|^2\right)dx
\notag\\
&\quad+\int_{\Omega}\left[\mathfrak{F}(r) |\delta\psi|^2+\left(\delta v_\theta+\frac{\partial_r\omega}{ \epsilon b}\delta\psi\right)^2\right]
dx=\langle A \delta u_2,\delta u_2\rangle +\left\langle \mathbb{L}\delta u_1,\delta u_1\right\rangle,
\end{align}where $\delta u_{1}=\left(  \delta v_{\theta}+\frac{\partial_{r}\omega
(r)}{\epsilon b(r)}\delta\psi,\delta\psi\right)  $, $\delta u_{2}=(\delta
v_{r},\delta v_{z},\delta H_{\theta})$ are associated with the perturbation
$(v,H_{\theta},\psi)-(v_{0},0,\psi_{0})$.
Combining \eqref{secondorderF1}-\eqref{firstvariational} and \eqref{secondvariational}, it turns out that
\begin{align}\label{EcE01}
&\notag E_c(v,H_\theta,\psi)-E_c(v_0,0,\psi_0)\notag\\
&=\frac{1}{2}\langle \nabla^2E_c(v_0,0,\psi_0)(v-v_0, H_\theta,\psi-\psi_0),(v-v_0, H_\theta,\psi-\psi_0)\rangle
\notag \\&\quad+ o(\|\psi-\psi_0\|^2_{H^1_{mag}})+o(\|v_\theta-\mathbf{v}_0\|^2_{L^2})\notag\\
&=\frac{1}{2}\int_{\Omega}\left(|v_r|^2+|v_z|^2+|H_\theta|^2+\frac{1}{r^2}|\partial_z(\psi-\psi_0)|^2+\frac{1}{r^2}|\partial_r(\psi-\psi_0)|^2\right)dx
\notag\\&\quad+\frac{1}{2}\int_{\Omega}\mathfrak{F}(r) |\psi-\psi_0|^2dx+\frac{1}{2}\int_{\Omega}\left[(v_\theta-\mathbf{v}_0)+\frac{\partial_r\omega}{ \epsilon b}(\psi-\psi_0)\right]^2
dx\notag \\&\quad+ o(\|\psi-\psi_0\|^2_{H^1_{mag}})+o(\|v_\theta-\mathbf{v}_0\|^2_{L^2}).
\end{align}
Define the distance function
\begin{align}\label{distance}
d(t)&=d_1(t)+d_2(t)+d_3(t),
\end{align}
where
\begin{align*}
d_1(t)&=d_1(v_r,v_z,H_\theta)=\frac{1}{2}\int_{\Omega}(|v_r|^2+|v_z|^2+|H_\theta|^2)dx,\\
d_2(t)&=d_2(v_\theta)=\frac{1}{2}\int_{\Omega}|v_\theta-\mathbf{v}_0(r)|^2dx,\\
d_3(t)&=d_3(H_r,H_z)=\frac{1}{2}\int_{\Omega}(|H_r|^2+|H_z-\epsilon b|^2)
 dx\\
 &=d_3(\psi)=\frac{1}{2}\int_{\Omega}\left(\frac{1}{r^2}|\partial_z(\psi-\psi_0)|^2+\frac{1}{r^2}|\partial_r(\psi-\psi_0)|^2\right)
 dx.
\end{align*}

Then the upper bound of the relative
energy-Casimir functional can be controlled by the distance function.
\begin{lemma}\label{EcEleq}
 There exists some $C>0$, such that
\begin{align*}
E_c(v,H_\theta,\psi)-E_c(v_0,0,\psi_0)\leq Cd +o(d).
\end{align*}
\end{lemma}
\begin{proof}
Applying \eqref{EcE01}-\eqref{distance}, and $H^1_{mag}\hookrightarrow\hookrightarrow L_{ |\mathfrak{F}(r)| }^2$, we get
\begin{align*}
&E_c(v,H_\theta,\psi)-E_c(v_0,0,\psi_0)
\\
&=d_1+d_3+\frac{1}{2}\int_{\Omega}\mathfrak{F}(r) (\psi-\psi_0)^2dx
+\frac{1}{2}\int_{\Omega}\left[(v_\theta-\mathbf{v}_0)+\frac{\partial_r\omega}{ \epsilon b}(\psi-\psi_0)\right]^2
dx+o(d)
\\
&=d_1+d_3+d_2+\frac{1}{2}\int_{\Omega}\frac{(\partial_r\omega)^2}{ \epsilon^2 b^2}(\psi-\psi_0)^2dx
+\int_{\Omega}(v_\theta-\mathbf{v}_0)\frac{\partial_r\omega}{ \epsilon b}(\psi-\psi_0)dx
\\&\quad+\frac{1}{2}\int_{\Omega} \mathfrak{F}(r) (\psi-\psi_0)^2dx+o(d)
\\&\leq d_1+C(d_3
+d_2)
+o(d)
\leq Cd +o(d),
\end{align*}
since
\begin{align*}
&\frac{1}{2}\int_{\Omega}\frac{(\partial_r\omega)^2}{ \epsilon^2 b^2}(\psi-\psi_0)^2dx
+\int_{\Omega}(v_\theta-\mathbf{v}_0)\frac{\partial_r\omega}{ \epsilon b}(\psi-\psi_0)dx
+\frac{1}{2}\int_{\Omega} \mathfrak{F}(r) (\psi-\psi_0)^2dx
\\&\leq C\left[\int_{\Omega}(v_\theta-\mathbf{v}_0)^2dx+\int_{\Omega}\frac{(\partial_r\omega)^2}{ \epsilon^2 b^2}(\psi-\psi_0)^2dx
+\int_{\Omega} \mathfrak{F}(r) (\psi-\psi_0)^2dx\right]
\\&\leq C(d_2+d_3).
\end{align*}
\end{proof}

Next, we will give the lower bound estimate of $E_c(v,H_\theta,\psi)-E_c(v_0,0,\psi_0)$ in three steps.
\begin{Step}
To begin, by the Taylor expansion of $E_c(v,H_\theta,\psi)$ at $(v_0,0,\psi_0),$ we deduce from \eqref{EcE01}-\eqref{distance} that
\begin{align}\label{casimirsminus}
&E_c(v,H_\theta,\psi)-E_c(v_0,0,\psi_0)
\notag\\&=d_1+d_3+\frac{1}{2}\int_{\Omega}\mathfrak{F}(r) (\psi-\psi_0)^2dx
\nonumber+\frac{1}{2}\int_{\Omega}\left[(v_\theta-\mathbf{v}_0)+\frac{\partial_r\omega}{ \epsilon b}(\psi-\psi_0)\right]^2
dx+o(d)\nonumber\\
&=d_1+\frac{1}{2}\langle L(\psi-\psi_0),(\psi-\psi_0)\rangle+\frac{1}{2}\int_{\Omega}\left[(v_\theta-\mathbf{v}_0)+\frac{\partial_r\omega}{ \epsilon b}(\psi-\psi_0)\right]^2dx+o(d).
\end{align}

Under the linear stability condition $\mathbb{L}_1>0,$ we have $n^{-}(\mathbb{L})=n^{-}(L)=\sum\limits_{k\in \mathbb{Z}}n^{-}(\mathbb{L}_k)=n^{-}(\mathbb{L}_0)$, since $\mathbb{L}_{-k}=\mathbb{L}_k>\mathbb{L}_1$ for $k>1$.
Thus, in order to obtain the coercivity estimate of $\langle L(\psi-\psi_0),(\psi-\psi_0)\rangle$, we need to control the projection of $\psi-\psi_0$ to negative and kernel directions of $\langle \mathbb{L}_0\cdot,\cdot\rangle$.
We
denote $K:=n^-(\mathbb{L}_0)$. By solving the following eigenvalue problem
\begin{align*}
\mathbb{L}_0 h+(\lambda-1)\mathfrak{F}(r)h=-\frac{1}{r}\partial_r\left(\frac{1}{r}\partial_r h \right)+\lambda\mathfrak{F}(r)h=0
\end{align*}
in the space $Z$, we obtain $K$ negative directions of $\langle \mathbb{L}_0\cdot,\cdot\rangle$ denoted by $\{h_i\}_{i=1}^{K}\in C^2(0,1)$ for $0\neq\lambda_i<1$.
Below, we give the detailed proof for the case when $\ker(\mathbb{L}_0)\neq \{0\}$ and the proof for the case $\ker(\mathbb{L}_0)= \{0\}$ is similar but simpler.
For the 2nd order ODE operator $\mathbb{L}_0$, it is easy to see that $\ker(\mathbb{L}_0)$ is at most one-dimensional. We denote
the kernel direction of $\langle \mathbb{L}_0\cdot,\cdot\rangle$ by $h_{0}\in C^2(0,1)$.
By the direct decomposition of $L$ on $Z$, we have the following lemma.
\begin{lemma}\label{lemma3.3}
Assume $\mathbb{L}_1>0$ with $\mathbb{L}_1$ defined by \eqref{defLk} for $k=1$.
If $0\neq\varphi$ satisfies
$
\int_\Omega \mathfrak{F}(r)h_i(r)\varphi dx=0$  for $i=0,1,...,K$,
then
\begin{align}\label{LNpositive}
\langle L\varphi,\varphi\rangle\geq\tilde{\delta}\|\varphi\|_{H^1_{mag}}^2
\end{align}
for some constant $\tilde{\delta}>0$.
\end{lemma}
\begin{proof}
Let
$$\hat{Z}=\{\phi\in Z| \langle L\phi, h_i\rangle=0, ( i=1,2,...K), \text{ and } ( \phi, h_0)_{H^1_{mag}}=0\}.$$
First, we note that $
\langle L\varphi,h_i\rangle=\int_\Omega (1-\lambda_i)\mathfrak{F}(r)h_i(r)\varphi dx=0$  $(i=1,...,K)$,
and $( \varphi,h_0)_{H^1_{mag}}=\int_\Omega \mathfrak{F}(r)h_0(r)\varphi dx=0$ because of $-\frac{1}{r}\partial_r\left(\frac{1}{r}\partial_r h_0 \right)+\mathfrak{F}(r)h_0=0.$
By the proof of Lemma A.2 in \cite{lin-zeng-hamiltonian}, it is sufficient to show $\langle L\varphi,\varphi\rangle>0$.
From $n^-(L|_{\hat{Z}})=n^-(L)-K=0$, we know that $L|_{\hat{Z}}\geq0$.
We show $\ker(L|_{\hat{Z}})=\{0\}$, and thus $L|_{\hat{Z}}>0$. Suppose otherwise.
Then for any $h\in \ker(L|_{\hat{Z}})$, we have $\langle Lh,g\rangle=0, \forall g\in \hat{Z}$ and consequently $\langle Lh,\tilde{g}\rangle=0, \forall \tilde{g}\in H^1_{mag}$, due to $H^1_{mag}=span\{h_1,\cdot\cdot\cdot,h_k\}\oplus \ker L\oplus\hat{Z}$. Hence, $h=ch_0$ and by $( h, h_0)_{H^1_{mag}}=0$, we have $c=0$.
\end{proof}

\end{Step}

\begin{Step}
In this step, we will separate $(\psi-\psi_0)$ to the non-positive direction $P_K(\psi-\psi_0)$ and the positive direction $(I-P_K)(\psi-\psi_0)$ of $L$,
where the operator $P_K$ is defined by \eqref{defPK}. Then we obtain the lower bound estimate of $\langle L(I-P_K)(\psi-\psi_0),(I-P_K)(\psi-\psi_0)\rangle$
and the upper bound estimate of $\langle LP_K(\psi-\psi_0),P_K(\psi-\psi_0)\rangle$.

We construct additional $K+1$ invariants
\begin{align*}
J_i(\psi)=\int_\Omega f_i(\psi)dx
\end{align*}
for $0\leq i\leq K$, where $f_i(s)\in C^3(\mathbb{R})$ and
\begin{align*}
f_i(s)=
\int_0^s h_i\left(r(\tau)\right)\mathfrak{F}(r(\tau))d\tau,\quad \text{for } 0\leq i\leq K,
\end{align*}
$r(\tau)$ is the inverse function of $\tau(r)=-\epsilon\int_0^rsb(s)ds$.
Thus
\begin{align}\label{derivateofQ}
f'_i(\psi_0)&=
 h_i\left(r\right)\mathfrak{F}(r),\quad \text{for } 0\leq i\leq K,
\end{align}
and $J_i(\psi)\in C^2(H^1_{mag})$ directly by Lemma \ref{Fderivate}.
Then, we introduce the finite-dimensional projection operator
\begin{align}\label{defPK}
P_K \varphi&=\sum_{i=0}^{K}\frac{\int_\Omega \mathfrak{F}(r)h_i(r)\varphi dx}{\int_\Omega \mathfrak{F}(r)h_i^2(r) dx} h_i(r)
\end{align}
for any $\varphi\in Z.$
By a direct computation, we have for any $\varphi\in Z$,
\begin{align*}
\langle L(I-P_K)\varphi,h_i\rangle&=\langle \mathbb{L}_0h_i,(I-P_K)\varphi\rangle
\nonumber\\&=\int_\Omega (1-\lambda_i)\mathfrak{F}(r)h_i(r)(I-P_K)\varphi dx=0
\end{align*}
for $i=1,...,K$,
and
\begin{align*}
( h_0(r),(I-P_K)\varphi)_{H^1_{mag}}=0.
\end{align*}

Using \eqref{derivateofQ}, 
$J_i(\psi)\in C^2(H^1_{mag})$ and Taylor expansion, we can now derive
\begin{align*}
J_i(\psi)-J_i(\psi_0)
&=\int_\Omega f_i'(\psi_0)(\psi-\psi_0) dx+O(d)\nonumber\\
&=\int_\Omega \mathfrak{F}(r)h_i(r)(\psi-\psi_0) dx+O(d).
\end{align*}
Hence, by $J_i(\psi(t))=J_i(\psi(0))$, one has
\begin{align}\label{Internal-product}
\notag\left|\int_\Omega \mathfrak{F}(r)h_i(r)(\psi-\psi_0) dx\right|&=|J_i(\psi(t))-J_i(\psi_0)-O(d)|\nonumber\\
\notag&\leq |J_i(\psi(t))-J_i(\psi_0)|+ O(d)\nonumber\\
\notag&= |J_i(\psi(0))-J_i(\psi_0)|+ O(d)\nonumber\\
&
\leq O(d(0)^{\frac{1}{2}})+O(d),
\end{align}
for $i=0,...,K.$

For the estimate of $\langle L\cdot,\cdot\rangle$ acting on positive directions, combing \eqref{LNpositive} and \eqref{Internal-product}, we have
\begin{align}\label{positive}
\left\langle L(I-P_K)(\psi-\psi_0),(I-P_K)(\psi-\psi_0)\right\rangle&\geq\tilde{\delta}\|(I-P_K)(\psi-\psi_0)\|_{H^1_{mag}}^2\nonumber\\
&\geq\tilde{\delta}\left[\frac{1}{2}\|\psi-\psi_0\|_{H^1_{mag}}^2-\|P_K(\psi-\psi_0)\|_{H^1_{mag}}^2\right]
\nonumber\\
&\geq\tilde{\delta}\left[\frac{1}{2}d_3-O( d(0))-O(d^2)\right],
\end{align}
since
$\|P_K(\psi-\psi_0)\|_{L_{ |\mathfrak{F}(r)| }^2}^2\lesssim\|P_K(\psi-\psi_0)\|_{H^1_{mag}}^2.$

For the estimate of $\langle L\cdot,\cdot\rangle$ acting on the projection in the kernel and negative directions, on account of \eqref{Internal-product}, we have
\begin{align}\label{Quadratic1}
&|\langle LP_K(\psi-\psi_0),P_K(\psi-\psi_0)\rangle|\nonumber\\
&\quad
=\sum_{i=0}^{K}\left(\frac{\int_\Omega \mathfrak{F}(r)h_i(r)(\psi-\psi_0) dx}{\int_\Omega \mathfrak{F}(r)h_i^2(r) dx}\right)^2 |\langle Lh_i,h_i\rangle|\nonumber\\
&\quad=\sum_{i=0}^{K}\left(\frac{\int_\Omega \mathfrak{F}(r)h_i(r)(\psi-\psi_0) dx}{\int_\Omega \mathfrak{F}(r)h_i^2(r) dx}\right)^2 |\langle (1-\lambda_i)\mathfrak{F}(r)h_i(r),h_i(r)\rangle|\nonumber\\
&\quad
\leq O( d(0))+O(d^2).
\end{align}

\end{Step}

\begin{Step} Finally, by $\langle L(I-P_K)\varphi,h_i\rangle=\int_\Omega \mathfrak{F}(r)h_i(r)(I-P_K)\varphi dx=0$, we can deduce from \eqref{casimirsminus} that
\begin{align}\label{casimirsminus*}
&E_c(v,H_\theta,\psi)-E_c(v_0,0,\psi_0)
\notag\\&=d_1+\frac{1}{2}\int_{\Omega}\left[(v_\theta-\mathbf{v}_0)+\frac{\partial_r\omega}{ \epsilon b}(\psi-\psi_0)\right]^2dx+\frac{1}{2}\langle LP_K(\psi-\psi_0),P_K(\psi-\psi_0)\rangle\nonumber\\
&\quad+\frac{1}{2}\left\langle L(I-P_K)(\psi-\psi_0),(I-P_K)(\psi-\psi_0)\right\rangle+o(d).
\end{align}
Note that
\begin{align}\label{3.18}
&\frac{1}{2}\int_{\Omega}\left[(v_\theta-\mathbf{v}_0)+\frac{\partial_r\omega}{ \epsilon b}(\psi-\psi_0)\right]^2
dx\geq\frac{1}{c}d_2-\frac{1}{c}d_3
\end{align}
for some $c>0$ large enough.

Plugging \eqref{positive}, \eqref{Quadratic1} and \eqref{3.18} into \eqref{casimirsminus*}, and by taking $c$ sufficiently large, such that $c>\frac{4}{\tilde{\delta}}$, we obtain
\begin{align*}
&E_c(v,H_\theta,\psi)-E_c(v_0,0,\psi_0)
\notag\\
&\geq d_1
+\frac{1}{c}d_2-\frac{1}{c}d_3-o(d)- O( d(0))- O(d^2(0))-O(d^2)\notag\\
&\quad+\frac{1}{2}\left\langle L(I-P_K)(\psi-\psi_0),(I-P_K)(\psi-\psi_0)\right\rangle \notag\\
&\geq d_1
+\frac{1}{c}d_2-o(d)+\left(\frac{1}{4}\tilde{\delta}-\frac{1}{c}\right)d_3- O( d(0))- O(d^2(0))-O(d^2)
\notag\\&\geq \tau d-o(d)- O( d(0)),
\end{align*}
where $\tau=\min\{1,\frac{1}{c},\frac{1}{4}\tilde{\delta}-\frac{1}{c}\}>0$.
\end{Step}

Thus, we obtain the lower bound estimate of $E_c(v,H_\theta,\psi)-E_c(v_0,0,\psi_0)$ in the following lemma.
\begin{lemma}\label{EcEgeq}
There exists a constant $\tau>0$, such that
\begin{align*}
E_c(v,H_\theta,\psi)-E_c(v_0,0,\psi_0)\geq   \tau d-o(d)- O( d(0)).
\end{align*}
\end{lemma}

\begin{proof}[Proof of Theorem \ref{nonlinearstable} when $R_2<\infty$]
By assumption 1) in Theorem \ref{nonlinearstable}, we have
$$E(t)=E(v(t,x),H_\theta(t,x),\psi(t,x))\leq E(v^0(x),0,\psi^0(x)).$$
By Lemma \ref{EcEleq} and Lemma \ref{EcEgeq}, it follows that
\begin{align*}
 Cd(0)+o(d(0))
&\geq [E_c(v^0(x),H_\theta^0(x),\psi^0(x))-E_c(v_0(x),0,\psi_0(x))]\\
&\geq [E_c(v(t,x),H_\theta(t,x),\psi(t,x))-E_c(v_0(x),0,\psi_0(x))]\\
&\geq \tau d-o(d)- O( d(0)).
\end{align*}
Then we have $ Cd(0)+o(d(0))\geq \tau d -o(d)$.
By the continuity assumption of $d$ on $t$ (assumption 2) in Theorem \ref{nonlinearstable}), we conclude that
 if $d(0)$ is small enough,
$d(t)\lesssim d(0)$ for any $t\geq0$.
\end{proof}


\section{Nonlinear instability} 
\label{nonlinearinstability}

In this section, we study the nonlinear instability under the linear instability condition $n^{-}(\mathbb{L}_1)\neq0$. 
We only prove Theorem \ref{nonlinearunstable} for the case $R_2<\infty$, $\epsilon=1,  b(r)=b$, where $b$ is a constant. For the general case including $R_2=\infty$, we can prove Theorem \ref{nonlinearunstable} by similar arguments with additional conditions $\omega=O(r^{-1-\alpha}), \partial_r b=O(r^{-1-2\alpha}) (\alpha>0,r\rightarrow\infty)$, which ensure $\nabla(v_0,H_0)\in L^2.$
Here, let the perturbations be
$
u(t,r,z)=v(t,r,z)-v_0(r,z), B(t,r,z)=H(t,r,z)-H_0(r,z),
W(t,r,z)=p(t,r,z)-p_0(r,z).
$
We shall rewrite the MHD system \eqref{EPM} in following perturbation form:
\begin{align}\label{2.1'}
\begin{cases}
\partial_t u+u\cdot\nabla v_0+v_0\cdot\nabla u-B\cdot\nabla H_0-H_0\cdot\nabla B+\nabla W=B\cdot\nabla B-u\cdot\nabla u,\\
\partial_t B+u\cdot\nabla H_0+v_0\cdot\nabla B-B\cdot\nabla v_0-H_0\cdot\nabla u=B\cdot\nabla u-u\cdot\nabla B,\\
\text{div}u=\text{div} B=0,
\end{cases}
\end{align}
with the initial and boundary conditions
\begin{align}\label{linearized-cartesian-boundary}
\begin{cases}
(u,B)(0,r,z)=(u^0,B^0)(r,z),\quad (r,z)\in\Omega,\\
u\cdot\mathbf{n}|_{\partial\Omega}= B\cdot\mathbf{n}|_{\partial\Omega}=0,\\
(u,B)(t,r,z)\to (0,0) \quad \mbox{as} \quad r\to \infty,\quad \mbox{if}\quad R_2=\infty;\\
u(t,r,z)=u(t,r,z+2\pi), B(t,r,z)=B(t,r,z+2\pi).
\end{cases}
\end{align}
Then we obtain the linearized system of the MHD system \eqref{EPM} around the axisymmetric steady solution $(v_0,H_0)(x)$ in \eqref{steady1} as
\begin{align}\label{linearized-cartesian}
\begin{cases}
\partial_t\begin{pmatrix}
 u\\
 B
\end{pmatrix}=-
\begin{pmatrix}
u\cdot\nabla v_0+v_0\cdot\nabla u-B\cdot\nabla H_0-H_0\cdot\nabla B+\nabla W\\
u\cdot\nabla H_0+v_0\cdot\nabla B-B\cdot\nabla v_0-H_0\cdot\nabla u
\end{pmatrix}:=L_{u,B}\begin{pmatrix}
 u\\
 B
\end{pmatrix}\\
\text{div}u=\text{div} B=0,
\end{cases}
\end{align}
with the initial and boundary conditions
 \eqref{linearized-cartesian-boundary}.
Moreover, in the cylindrical coordinates, \eqref{linearized-cartesian} can be rewritten as
\begin{align}\label{linearized-cylindrical}
\begin{cases}
\partial_t u_r=b\partial_zB_r+2\omega u_\theta-\partial_r W,\\
\partial_t u_\theta=b\partial_zB_\theta-\frac{\partial_r(r^2\omega)}{r}u_r,\\
\partial_t u_z=b\partial_zB_z-\partial_z W,\\
\partial_t B_r=b\partial_zu_r,\\
\partial_t B_\theta=b\partial_zu_\theta+rB_r\partial_r\omega,\\
\partial_t B_z=b\partial_zu_z,\\
\frac{1}{r}\partial_r(ru_r)+\partial_zu_z=\frac{1}{r}\partial_r(rB_r)+\partial_zB_z=0.
\end{cases}
\end{align}

\subsection{Semigroup estimates for the linearized MHD equations }

According to Theorem \ref{linearstability}, if $b^2<B_0^2,$ we have $n^{-}\left(  \mathbb{L}|_{\overline{R\left(  \mathbb{B}\right)
}}\right)>0$. Define $\Lambda^2:=\max\{\sigma(-\mathbb{B}^{'}\mathbb{L}\mathbb{B}A)\}>0$, then $\Lambda>0$ is the maximal exponential growth rate for the linearized MHD equations \eqref{linearmhd}. By Theorem \ref{linearstability}, we can construct a maximal growing normal mode ($\Lambda=\lambda_{k_0}$):
\begin{align}\label{lineargrowmode}
&u_r(t,r,z)=\tilde{u}_r(r)\cos (k_0z)  e^{\lambda_{k_0} t},\nonumber\\
&u_\theta(t,r,z)=\tilde{u}_\theta(r)\cos (k_0z)  e^{\lambda_{k_0} t},\nonumber\\
&u_z(t,r,z)=\tilde{u}_z(r)\sin (k_0z)  e^{\lambda_{k_0} t},\nonumber\\
&\varphi(t,r,z)=\tilde{\varphi}(r)\cos (k_0z)  e^{\lambda_{k_0} t},\nonumber\\
&B_\theta(t,r,z)=\tilde{B}_\theta(r)\sin (k_0z)  e^{\lambda_{k_0} t},\nonumber\\
&W(t,r,z)=\tilde{W}(r)\cos (k_0z)  e^{\lambda_{k_0} t}.
\end{align}
By substituting \eqref{lineargrowmode} into \eqref{linearmhd},
the spectral problem then reduces to
\begin{align}\label{spectralequation}
&\left(\lambda_{k_0}+\frac{k_0^2b^2}{\lambda_{k_0}}\right)r\partial_r\left(\frac{1}{r}\partial_r\tilde{\varphi}(r)\right)
=k_0^2\left[\lambda_{k_0}+\frac{k_0^2b^2}{\lambda_{k_0}}
+\frac{4\omega^2}{\lambda_{k_0}}\left(1-\frac{k_0^2b^2}{\lambda_{k_0}^2+k_0^2b^2}\right)+\frac{r\partial_r(\omega^2)}{\lambda_{k_0}}
\right]\tilde{\varphi}(r),
\end{align}
with boundary conditions
\begin{align*}
\tilde{\varphi}(R_1)=\tilde{\varphi}(R_2)=0.
\end{align*}
Due to $\omega(r)\in C^{\infty}(\Omega),$ we can deduce from the above second-order equation \eqref{spectralequation} that $\tilde{\varphi}(r)\in C^{\infty}(\Omega)$ by a standard bootstrap argument. Moreover, we get $\tilde{u}_r(r),\tilde{u}_\theta(r),\tilde{u}_z(r),\tilde{B}_\theta(r),\tilde{W}(r)\in C^{\infty}(\Omega)$ by the same token.

\begin{lemma}\label{linearexponential}
Let $(u,B)(t)$ be the solution to the linearized equations \eqref{linearized-cartesian} with the initial and boundary conditions
 \eqref{linearized-cartesian-boundary}. Then, for any $s\geq0$, there holds
\begin{align*}
\|(u,B)(t)\|_{H^s}\lesssim e^{\Lambda t}\|(u,B)(0)\|_{H^{s}}.
\end{align*}
\end{lemma}
\begin{proof}
Note that in the cylindrical coordinates, the linearized system \eqref{linearized-cartesian} with the initial and boundary conditions
 \eqref{linearized-cartesian-boundary} can be rewritten as \eqref{linearized-cylindrical} with the initial and boundary conditions
 \eqref{linearized-cartesian-boundary}, which is equivalent to the linearized system \eqref{hamiltonian-RS} with $B_r=\frac{\partial_z\varphi}{r}$ and $B_z=-\frac{\partial_r\varphi}{r}$. It is clear that
\begin{equation*}
\frac{d}{dt}\begin{pmatrix} \partial_{z}^{\alpha}u_1\\
\partial_{z}^{\alpha}u_2\end{pmatrix}=\begin{pmatrix} 0, &\mathbb{B}\\
-\mathbb{B}',& 0
\end{pmatrix} \begin{pmatrix} \mathbb{L}, &0\\
0,& A
\end{pmatrix}\begin{pmatrix}\partial_{z}^{\alpha}u_1\\
\partial_{z}^{\alpha}u_2
\end{pmatrix}=\begin{pmatrix}\mathbb{B}A\partial_{z}^{\alpha}u_2\\
-\mathbb{B}^{'}\mathbb{L}\partial_{z}^{\alpha}u_1
\end{pmatrix}
\end{equation*}
with $|\alpha|\leq s$ and $(u_1,u_2)$ be defined by \eqref{u1u2}.
Using (\ref{estimate-stable-unstable}) and (\ref{estimate-center}),
 we obtain
 \begin{align*}
\|(u_1,u_2)(t)\|_{\mathbf{X}}\lesssim e^{\Lambda t}\|(u_1,u_2)(0)\|_{\mathbf{X}}
\end{align*}
and
 \begin{align*}
\|\partial_{z}^{\alpha}(u_1,u_2)(t)\|_{\mathbf{X}}\lesssim e^{\Lambda t}\|\partial_{z}^{\alpha}(u_1,u_2)(0)\|_{\mathbf{X}},\quad |\alpha|\leq s.
\end{align*}
Together with the facts $\|(u,B)(t)\|_{L^2}\sim\|(u_1,u_2)(t)\|_{\mathbf{X}}$ and $\|\partial_{z}^{\alpha}(u,B)(t)\|_{L^2}\sim\|\partial_{z}^{\alpha}(u_1,u_2)(t)\|_{\mathbf{X}}$, we get
  \begin{align*}
\|(u,B)(t)\|_{L^2}\lesssim e^{\Lambda t}\|(u,B)(0)\|_{L^2}
\end{align*}
and
\begin{align}\label{zinvariant}
\|\partial_{z}^{\alpha}(u,B)(t)\|_{L^2}\lesssim e^{\Lambda t}\|\partial_{z}^{\alpha}(u,B)(0)\|_{L^2},\quad |\alpha|\leq s.
\end{align}
Next, we need to prove $\|\partial_{r}^{\alpha}(u,B)(t)\|_{L^2}\lesssim e^{\Lambda t}\|(u,B)(0)\|_{H^{s}}$ $(|\alpha|=s)$ by an induction on $s$.
For $s=1$, taking $\partial_r$ on both side of $\eqref{linearized-cylindrical}_2$ and $\eqref{linearized-cylindrical}_5$, we obtain
\begin{equation*}
\partial_t\left(
\begin{array}
[c]{c}%
\partial_ru_\theta\\
\partial_rB_\theta
\end{array}
\right)  =\left(
\begin{array}
[c]{cc}%
0 & b\partial_z\\
b\partial_z & 0
\end{array}
\right)  \left(
\begin{array}
[c]{c}%
\partial_ru_\theta\\
\partial_rB_\theta
\end{array}
\right)  +\left(
\begin{array}
[c]{c}%
-\partial_r(ru_r)\frac{\partial_r(r^2\omega)}{r^2}-ru_r\partial_r\left(\frac{\partial_r(r^2\omega)}{r^2}\right)\\
\partial_r(rB_r)\partial_r\omega+rB_r\partial_r^2\omega
\end{array}
\right)  ,
\end{equation*}
together with $\eqref{linearized-cylindrical}_7$, it follows
\begin{equation*}
\partial_t\left(
\begin{array}
[c]{c}%
\partial_ru_\theta\\
\partial_rB_\theta
\end{array}
\right)=\left(
\begin{array}
[c]{cc}%
0 & b\partial_z\\
b\partial_z & 0
\end{array}
\right)  \left(
\begin{array}
[c]{c}%
\partial_ru_\theta\\
\partial_rB_\theta
\end{array}
\right) +\left(
\begin{array}
[c]{c}%
\partial_zu_z\frac{\partial_r(r^2\omega)}{r}-ru_r\partial_r\left(\frac{\partial_r(r^2\omega)}{r^2}\right)\\
-\partial_zB_zr\partial_r\omega+rB_r\partial_r^2\omega
\end{array}
\right). %
\end{equation*}
Denote
\begin{equation*}
A_1=\left(
\begin{array}
[c]{cc}%
0 & b\partial_z\\
b\partial_z & 0
\end{array}
\right)
\end{equation*}
and
\begin{equation*}
f_1=\left(
\begin{array}
[c]{c}%
\partial_zu_z\frac{\partial_r(r^2\omega)}{r}-ru_r\partial_r\left(\frac{\partial_r(r^2\omega)}{r^2}\right)\\
-\partial_zB_z r\partial_r\omega+rB_r\partial_r^2\omega
\end{array}
\right).
\end{equation*}
By Duhamel's principle, we have
\begin{equation*}
\left\|\left(
\begin{array}
[c]{c}%
\partial_ru_\theta\\
\partial_rB_\theta
\end{array}
\right) (t)\right\|_{L^2}\leq \left\|e^{tA_1}\left(
\begin{array}
[c]{c}%
\partial_ru_\theta\\
\partial_rB_\theta
\end{array}
\right) (0)\right\|_{L^2}+\int_0^t \|e^{(t-s)A_1}f_1(s)\|_{L^2}ds.
\end{equation*}
Since $A_1$ is anti-self adjoint on $L^2$, by Stone theorem, we know that $\|e^{tA_1}\|_{L^2\rightarrow L^2}=1$.
Together with \eqref{zinvariant}, we deduce from above inequality that
\begin{equation}\label{r-u-theta}
\left\|\left(
\begin{array}
[c]{c}%
\partial_ru_\theta\\
\partial_rB_\theta
\end{array}
\right) (t)\right\|_{L^2}
\lesssim e^{\Lambda t}\|(u,B)(0)\|_{H^{1}}.
\end{equation}
Then, taking $\partial_r$ on both side of $\eqref{linearized-cylindrical}_1,\eqref{linearized-cylindrical}_3,\eqref{linearized-cylindrical}_4$ and $\eqref{linearized-cylindrical}_6$, we obtain
\begin{equation*}
\partial_t\left(
\begin{array}
[c]{c}%
\partial_ru_r\\
\partial_ru_z\\
\partial_rB_r\\
\partial_rB_z
\end{array}
\right)  =\left(
\begin{array}
[c]{cccc}%
0 & 0 & b\partial_z & 0 \\
0 & 0 &0 & b\partial_z \\
b\partial_z & 0 &0 & 0\\
0 & b\partial_z & 0 &0
\end{array}
\right)  \left(
\begin{array}
[c]{c}%
\partial_ru_r\\
\partial_ru_z\\
\partial_rB_r\\
\partial_rB_z
\end{array}
\right)  +\begin{pmatrix} \partial_r\mathcal{P}\begin{pmatrix}2\omega(r)u_\theta\\
0\end{pmatrix}\\
0\\
0
\end{pmatrix}.
\end{equation*}
Let
\begin{equation*}
A_2=\left(
\begin{array}
[c]{cccc}%
0 & 0 & b\partial_z & 0 \\
0 & 0 &0 & b\partial_z \\
b\partial_z & 0 &0 & 0\\
0 & b\partial_z & 0 &0
\end{array}
\right)
\end{equation*}
and
\begin{equation*}
f_2=\begin{pmatrix} \partial_r\mathcal{P}\begin{pmatrix}2\omega(r)u_\theta\\
0\end{pmatrix}\\
0\\
0
\end{pmatrix}.
\end{equation*}
By Duhamel's principle, we get
\begin{equation*}
\left\|\left(
\begin{array}
[c]{c}%
\partial_ru_r\\
\partial_ru_z\\
\partial_rB_r\\
\partial_rB_z
\end{array}
\right)  (t)\right\|_{L^2}\leq \left\|e^{tA_2}\left(
\begin{array}
[c]{c}%
\partial_ru_r\\
\partial_ru_z\\
\partial_rB_r\\
\partial_rB_z
\end{array}
\right)  (0)\right\|_{L^2}+\int_0^t \|e^{(t-s)A_2}f_2(s)\|_{L^2}ds.
\end{equation*}
Since $A_2$ is anti-self adjoint on $L^2$, by Stone theorem and
\eqref{r-u-theta}, one has
\begin{equation}\label{r-u-r}
\left\|\left(
\begin{array}
[c]{c}%
\partial_ru_r\\
\partial_ru_z\\
\partial_rB_r\\
\partial_rB_z
\end{array}
\right)  (t)\right\|_{L^2}\lesssim e^{\Lambda t}\|(u,B)(0)\|_{H^{1}}.
\end{equation}
Combining \eqref{r-u-theta} and \eqref{r-u-r}, we obtain
 \begin{align*}
\|\partial_{r}(u,B)(t)\|_{L^2}
\lesssim e^{\Lambda t}\|(u,B)(0)\|_{H^{1}}.
\end{align*}
Assume it holds for $|\alpha|\leq s-1 (s\geq 1)$, we will prove $|\alpha|=s$ by induction.
Similarly, we get
\begin{equation}\label{uo-uB-r}
\partial_t\left(
\begin{array}
[c]{c}%
\partial_r^\alpha u_\theta\\
\partial_r^\alpha B_\theta
\end{array}
\right)=A_1  \left(
\begin{array}
[c]{c}%
\partial_r^\alpha u_\theta\\
\partial_r^\alpha B_\theta
\end{array}
\right) +f_s %
\end{equation}
and
\begin{equation}\label{uo-uuBB-r}
\partial_t\left(
\begin{array}
[c]{c}%
\partial_r^\alpha u_r\\
\partial_r^\alpha u_z\\
\partial_r^\alpha B_r\\
\partial_r^\alpha B_z
\end{array}
\right)  =A_2  \left(
\begin{array}
[c]{c}%
\partial_r^\alpha u_r\\
\partial_r^\alpha u_z\\
\partial_r^\alpha B_r\\
\partial_r^\alpha B_z
\end{array}
\right)  +\begin{pmatrix} \partial_r^\alpha \mathcal{P}\begin{pmatrix}2\omega(r)u_\theta\\
0\end{pmatrix}\\
0\\
0
\end{pmatrix},
\end{equation}
with
\begin{equation}\label{fs}
f_s=\left(
\begin{array}
[c]{c}%
\sum\limits_{|\beta|+|\gamma|= s-1,0\leq|\beta|,|\gamma|\leq s-1}
\left[\partial_r^{\beta}\partial_zu_z\partial_r^{\gamma}\left(\frac{\partial_r(r^2\omega)}{r}\right)
-\partial_r^\beta (ru_r)\partial_r^{\gamma}\partial_r\left(\frac{\partial_r(r^2\omega)}{r^2}\right)\right]\\
\sum\limits_{|\beta|+|\gamma|= s-1,0\leq|\beta|,|\gamma|\leq s-1}\left[-\partial_r^\beta(\partial_zB_z) \partial_r^{\gamma}(r\partial_r\omega)+\partial_r^\beta (rB_r)\partial_r^{\gamma}\partial_r^2\omega\right]
\end{array}
\right).
\end{equation}
By the induction hypothesis and Duhamel's principle, one gets from \eqref{uo-uB-r}-\eqref{fs} that
\begin{equation*}
\left\|\left(
\begin{array}
[c]{c}%
\partial_r^\alpha u_\theta\\
\partial_r^\alpha B_\theta
\end{array}
\right) (t)\right\|_{L^2}
\lesssim e^{\Lambda t}\|(u,B)(0)\|_{H^{s}}
\end{equation*}
 and
 \begin{equation*}
\left\|\left(
\begin{array}
[c]{c}%
\partial_r^\alpha u_r\\
\partial_r^\alpha u_z\\
\partial_r^\alpha B_r\\
\partial_r^\alpha B_z
\end{array}
\right)  (t)\right\|_{L^2}\lesssim e^{\Lambda t}\|(u,B)(0)\|_{H^{s}}.
\end{equation*}
Therefore, it follows that
\begin{align*}
\|\partial_{r}^{\alpha}(u,B)(t)\|_{L^2}
\lesssim e^{\Lambda t}\|(u,B)(0)\|_{H^{s}}
\end{align*}
with $|\alpha|=s.$
Hence,
\begin{align*}
\|(u,B)(t)\|_{H^s}
\lesssim e^{\Lambda t}\|(u,B)(0)\|_{H^{s}}.
\end{align*}
We complete the proof of the lemma.

\end{proof}

\begin{proof}[Proof of Corollary \ref{exp_dychotomy}]
We define a linear map $\Gamma: \mathbf{X}\rightarrow H^s_{mhd}$ by $(v_1,v_2):=\Gamma \xi\in H^s_{mhd}$ where $\xi=(\xi_1,\xi_2,\xi_3,\xi_4,\xi_5)\in \mathbf{X}$,
\begin{align*}
v_1=\xi_3\textbf{e}_r+\xi_4\textbf{e}_z+\left(\xi_1-\frac{\partial_r\omega}{\epsilon b}\xi_2\right)\textbf{e}_\theta
\text{ and } v_2=\frac{\partial_z\xi_2}{r}\textbf{e}_r-\frac{\partial_r\xi_2}{r}\textbf{e}_z+\xi_5\textbf{e}_\theta.
\end{align*}
By Theorem \ref{linearstability} and assumptions $\omega,b\in H^{s+2}$, we denote
\begin{align*}
E^u:=\text{span}\{\xi^i,i=1,2,...,N\} \text{ and } E^s:=\text{span}\{\eta^i,i=1,2,...,N\},
\end{align*}
and
\begin{align*}
X_u:=\text{span}\{\hat{\xi}^i,i=1,2,...,N\}\subset H^s_{mhd} \text{ and } X_{s}:=\text{span}\{\hat{\eta}^i,i=1,2,...,N\}\subset H^s_{mhd},
\end{align*}
where $\hat{\xi}^i$ and $\hat{\eta}^i$ satisfy
$\hat{\xi}^i=\Gamma\xi^i$, $\hat{\eta}^i=\Gamma \eta^i$,
for all $i=1,2,...,N$, and $N=2\sum\limits_{k=1}^{\infty}n^{-}(\mathbb{L}_{k})<+\infty$
is the dimension of the unstable/stable spaces $E^u$/$E^s$.
Then we have
 $$X^c=\{\hat{\zeta}|\hat{\zeta}=\Gamma \zeta, \zeta\in E^c\}\cap H^s_{mhd}.$$

For $s=0$, Corollary \ref{exp_dychotomy} can be directly inferred from Theorem \ref{linearstability}.
For $s>0$, Corollary \ref{exp_dychotomy} i)-iii) are clear since $X_{s}$ and $X_u$ are finite-dimensional.
Corollary \ref{exp_dychotomy} iv) can be directly obtained by the proof of Lemma \ref{linearexponential} with the initial data in the center subspace.
\end{proof}

\subsection{A higher-order approximate solution}
In this subsection, we shall construct the higher-order axisymmetric approximate solution $(u^a,B^a,W^a)(t,r,z)$ to \eqref{2.1'}-\eqref{linearized-cartesian-boundary} with the form
\begin{align}\label{approximate}
\begin{cases}
u^a(t,r,z)=\sum\limits_{j=1}^{N}\delta^{j}\phi_{j}(t,r,z),\\
B^a(t,r,z)=\sum\limits_{j=1}^{N}\delta^{j}\xi_{j}(t,r,z),\\
W^a(t,r,z)=\sum\limits_{j=1}^{N}\delta^{j}\eta_{j}(t,r,z),
\end{cases}
\end{align}
which satisfies
\begin{align}\label{4.1'}
\begin{cases}
 &\partial_t u^a+u^a\cdot\nabla v_0+v_0\cdot\nabla u^a-B^a\cdot\nabla H_0-H_0\cdot\nabla B^a+\nabla W^a
 \\&\quad\quad\quad\quad\quad=B^a\cdot\nabla B^a-u^a\cdot\nabla u^a+R_{N,1}^a,
 \\&
\partial_t B^a+u^a\cdot\nabla H_0+v_0\cdot\nabla B^a-B^a\cdot\nabla v_0-H_0\cdot\nabla u^a
\\&\quad\quad\quad\quad\quad=B^a\cdot\nabla u^a-u^a\cdot\nabla B^a+R_{N,2}^a,
\\&
\text{div}\phi_{j}=\text{div} \xi_{j}=0\quad (1\leq j\leq N),
\end{cases}
\end{align}
with the initial and boundary conditions
\begin{align}\label{aubdcdiffer}
\begin{cases}
(u^a,B^a)(0,r,z)=(u^{a0},B^{a0})(r,z)=\delta(\phi_{1},\xi_{1})(0,r,z),\quad (r,z)\in\Omega,\\
u^a\cdot\mathbf{n}|_{\partial\Omega}= B^a\cdot\mathbf{n}|_{\partial\Omega}=0,\\
(u^a,B^a)(t,r,z)\to (0,0) \quad as \quad r\to \infty,\quad \mbox{if}\quad R_2=\infty;\\
u^a(t,r,z)=u^a(t,r,z+2\pi), B^a(t,r,z)=B^a(t,r,z+2\pi),
\end{cases}
\end{align}
where $\delta>0$ be
an arbitrary small parameter and $(\phi_{1},\xi_{1})$ be the maximal growing normal mode with $\phi_{1}=(u_r,u_\theta,u_z),\xi_{1}=(\frac{\partial_z\varphi}{r},B_\theta,-\frac{\partial_r\varphi}{r})$, and $(u_r,u_\theta,u_z,\varphi,B_\theta)$ is given by \eqref{lineargrowmode}.
\begin{lemma}\label{4.1}
Let $N>0$ be a fixed integer and $\delta>0$ be
an arbitrary small parameter. 
Then
there exists an axisymmetric approximate solution $(u^a,B^a,W^a)(t,r,z)$ of the form \eqref{approximate} to \eqref{4.1'}-\eqref{aubdcdiffer}. Moreover, for every integer $s\geq0$, there is $\theta>0$ 
small enough,
for
$0\leq t\leq T^\delta$
with
$\theta=\delta e^{\Lambda T^{\delta}}$,
 the $j$-th order coefficients $\phi_{j}(t,r,z),\xi_{j}(t,r,z),\eta_{j}(t,r,z)$, the $N+1$-th order remainders $R_{N,1}^a(t,r,z),R_{N,2}^a(t,r,z)$ satisfy
\begin{align}\label{4.4}
\|(\phi_{j},\xi_{j},\nabla\eta_{j})(t)\|_{H^s}\leq C_{s,N}e^{j\Lambda t}, \quad \text{for} \quad 1\leq j\leq N,
\end{align}
\begin{align}\label{4.11}
\|(R_{N,1}^a,R_{N,2}^a)(t)\|_{H^s}\leq C_{s,N}\delta^{N+1}e^{(N+1)\Lambda t}.
\end{align}
\end{lemma}
\begin{proof}
By the linear semigroup estimates in Lemma \ref{linearexponential}, this proof can be obtained by the same method in \cite{MR1869291}.
Roughly speaking, we can construct $\phi_{j},\xi_{j},\eta_{j},R_{j,1}^a,R_{j,2}^a$ by induction on $j$.
For $j=1$, let $(\phi_{1},\xi_{1},\eta_{1})$ be the maximal growing normal mode with $\phi_{1}=(u_r,u_\theta,u_z),\xi_{1}=(\frac{\partial_z\varphi}{r},B_\theta,-\frac{\partial_r\varphi}{r}),\eta_{1}=W$, where $(u_r,u_\theta,u_z,\varphi,B_\theta,W)$ is given by \eqref{lineargrowmode}.
 Clearly $\phi_{j},\xi_{j}$ and $\eta_{j}$ satisfy \eqref{4.4}, and hence $R_{N,1}^a$ and $R_{N,2}^a$ satisfy \eqref{4.11}.

Assume that we have constructed $\phi_{j},\xi_{j},\eta_{j}$ and $R_{j,1}^a,R_{j,2}^a$ which satisfy
\eqref{4.4} and \eqref{4.11} for $j<N$. Based on it, now we turn to construct $\phi_{j+1},\xi_{j+1},\eta_{j+1}$ and $R_{j+1,1}^a,R_{j+1,2}^a$. Let
\begin{align*}
\begin{cases}
u_j=\sum\limits_{k=1}^{j}\delta^{k}\phi_{j},\\
B_j=\sum\limits_{k=1}^{j}\delta^{k}\xi_{j},\\
W_j=\sum\limits_{k=1}^{j}\delta^{k}\eta_{j}.
\end{cases}
\end{align*}
Putting $u_{j},B_{j},W_{j}$ into \eqref{4.1'}, we define
\begin{align*}
&F_{j+1}(\delta)=B_{j}\cdot\nabla B_{j}-u_{j}\cdot\nabla u_{j}, \quad G_{j+1}(\delta)=B_{j}\cdot\nabla u_{j}-u_{j}\cdot\nabla B_{j}
\end{align*}
as the nonlinear part corresponding to $u_j,B_{j}$. For
$0\leq t\leq T^\delta$ with $\theta$ small, we can expand $F_{j+1}(\delta)$ and $G_{j+1}(\delta)$ around $\delta=0.$
Let $(\phi_{j+1},\xi_{j+1},\eta_{j+1})$ be the solution of the system
\begin{align*}
\begin{cases}
 &\partial_t \phi_{j+1}+\phi_{j+1}\cdot\nabla v_0+v_0\cdot\nabla \phi_{j+1}-\xi_{j+1}\cdot\nabla H_0-H_0\cdot\nabla \xi_{j+1}+\nabla \eta_{j+1}
 =-\frac{F_{j+1}^{(j+1)}(0)}{(j+1)!},
 \\&
\partial_t \xi_{j+1}+\phi_{j+1}\cdot\nabla H_0+v_0\cdot\nabla \xi_{j+1}-\xi_{j+1}\cdot\nabla v_0-H_0\cdot\nabla \phi_{j+1}
=-\frac{G_{j+1}^{(j+1)}(0)}{(j+1)!},
\\&
\text{div}\phi_{j+1}=\text{div} \xi_{j+1}=0,
\end{cases}
\end{align*}
with $\phi_{j+1}(0,r,z)=(0,0,0),\xi_{j+1}(0,r,z)=(0,0,0)$ and
\begin{align*}
-\frac{F_{j+1}^{(j+1)}(0)}{(j+1)!}=\sum_{j_1+j_{2}=j+1, 1\leq j_1,j_{2}\leq j
}\xi_{j_1}\cdot\nabla \xi_{j_2}-\phi_{j_1}\cdot\nabla \phi_{j_2},
\end{align*}
\begin{align*}
-\frac{G_{j+1}^{(j+1)}(0)}{(j+1)!}=\sum_{j_1+j_{2}=j+1, 1\leq j_1,j_{2}\leq j
}\xi_{j_1}\cdot\nabla \phi_{j_2}-\phi_{j_1}\cdot\nabla \xi_{j_2}.
\end{align*}
Consequently, by induction, we have
\begin{align*}
\left\|\frac{F_{j+1}^{(j+1)}(0)}{(j+1)!}\right\|_{H^s}\leq C_{s,N}e^{(j+1)\Lambda t},\\
\left\|\frac{G_{j+1}^{(j+1)}(0)}{(j+1)!}\right\|_{H^s}\leq C_{s,N}e^{(j+1)\Lambda t}.
\end{align*}
Together with the estimate in Lemma \ref{linearexponential}, Duhamel's principle and $j\geq1,$ we obtain
\begin{align*}
\|(\phi_{j+1},\xi_{j+1})\|_{H^s}&\leq C\int_{0}^te^{\Lambda (t-\tau)}\left\|\frac{h_{j+1}^{(j+1)}(0)}{(j+1)!}\right\|_{H^{s}}d\tau
\\&\leq C_{s,N}\int_{0}^te^{\Lambda (t-\tau)}e^{(j+1)\Lambda \tau}d\tau
\\&\leq C_{s,N}e^{\Lambda t}\int_{0}^te^{[(j+1)\Lambda-\Lambda ]\tau}d\tau
\\&\leq C_{s,N}e^{\Lambda t}\frac{1}{(j+1)\Lambda-\Lambda }e^{[(j+1)\Lambda-\Lambda ]t}
\\&\leq C_{s,N}\frac{1}{(j+1)\Lambda-\Lambda }e^{(j+1)\Lambda t}\leq C_{s,N} e^{(j+1)\Lambda t}
\end{align*}
and
\begin{align*}
\|\nabla\eta_{j+1}\|_{H^s}\leq C\|(\phi_{j+1},\xi_{j+1})\|_{H^s}
\leq C_{s,N} e^{(j+1)\Lambda t}.
\end{align*}
Then, 
it follows from the definition of $(u^a,B^a,W^a)$ in \eqref{approximate} that
\begin{align*}
\begin{cases}
 &\partial_t u^a+u^a\cdot\nabla v_0+v_0\cdot\nabla u^a-B^a\cdot\nabla H_0-H_0\cdot\nabla B^a+\nabla W^a
 =-\sum\limits_{j=1}^{N}\frac{F_{j+1}^{(j+1)}(0)}{(j+1)!},
 \\&
\partial_t B^a+u^a\cdot\nabla H_0+v_0\cdot\nabla B^a-B^a\cdot\nabla v_0-H_0\cdot\nabla u^a
=-\sum\limits_{j=1}^{N}\frac{G_{j+1}^{(j+1)}(0)}{(j+1)!},
\\&
\text{div}u^a=\text{div} B^a=0.
\end{cases}
\end{align*}
Thus
\begin{align*}
R_{N,1}^a=F(\delta)-\sum_{j=1}^{N}\delta^{j+1}\frac{F_{j+1}^{(j+1)}(0)}{(j+1)!},\\
R_{N,2}^a=G(\delta)-\sum_{j=1}^{N}\delta^{j+1}\frac{G_{j+1}^{(j+1)}(0)}{(j+1)!},
\end{align*}
with
\begin{align*}
&F(\delta)=B^a\cdot\nabla B^a-u^a\cdot\nabla u^a, \quad G(\delta)=B^a\cdot\nabla u^a-u^a\cdot\nabla B^a.
\end{align*}
Hence, one has
\begin{align*}
\|(R_{N,1}^a,R_{N,2}^a)\|_{H^s}\leq C_{s,N}\delta^{N+1}e^{(N+1)\Lambda t}.
\end{align*}
The proof of the lemma is completed.

\end{proof}

\subsection{Energy estimates}\label{energyes}

As a prime, we shall give the local solvability for the system \eqref{2.1'} with the initial and boundary conditions
 \eqref{linearized-cartesian-boundary}.
\begin{lemma}\label{Local-existence}
Let $s\geq3$ and the initial data $(u^0,B^0)(r,z)\in H^s$. Then there is a $T>0$ such that there exists a unique local solution $(u,B,W)(t,r,z)\in C([0,T];H^s)$ to the system \eqref{2.1'} with the initial and boundary conditions
 \eqref{linearized-cartesian-boundary}.
\end{lemma}
\begin{proof}
This can be proved by standard methods \cite{Kato1975,Majda1984}.
\end{proof}

Let $(u,B,W)(t,r,z)\in C([0,T];H^s)$ be a local axisymmetric solution
as constructed in Lemma \ref{Local-existence}. Let $(u^a,B^a,W^a)(t,r,z)$ be an axisymmetric approximate
solution as in Lemma \ref{4.1}.
In precise, $(u^a,B^a,W^a)(t,r,z)$ satisfies
\begin{align*}
\begin{cases}
 &\partial_t u^a+u^a\cdot\nabla v_0+v_0\cdot\nabla u^a-B^a\cdot\nabla H_0-H_0\cdot\nabla B^a+\nabla W^a
 \\&\quad\quad\quad\quad\quad=B^a\cdot\nabla B^a-u^a\cdot\nabla u^a+R_{N,1}^a,
 \\&
\partial_t B^a+u^a\cdot\nabla H_0+v_0\cdot\nabla B^a-B^a\cdot\nabla v_0-H_0\cdot\nabla u^a
\\&\quad\quad\quad\quad\quad=B^a\cdot\nabla u^a-u^a\cdot\nabla B^a+R_{N,2}^a,
\\&
\text{div}u^a=\text{div} B^a=0.
\end{cases}
\end{align*}
Denote
$$u^d=u-u^a, \quad B^d=B-B^a, \quad W^d=W-W^a.$$
Thus, $(u^d,B^d,W^d)(t,r,z)$ satisfies
\begin{align}\label{5.1}
\begin{cases}
  &\partial_t u^d+u^d\cdot\nabla v_0+v_0\cdot\nabla u^d-B^d\cdot\nabla H_0-H_0\cdot\nabla B^d+\nabla W^d
  \\&\quad\quad\quad\quad\quad=B^d\cdot\nabla B+B^a\cdot\nabla B^d-u^d\cdot\nabla u-u^a\cdot\nabla u^d-R_{N,1}^a,\\
&\partial_t B^d+u^d\cdot\nabla H_0+v_0\cdot\nabla B^d-B^d\cdot\nabla v_0-H_0\cdot\nabla u^d\\&
\quad\quad\quad\quad\quad=B^d\cdot\nabla u+B^a\cdot\nabla u^d-u^d\cdot\nabla B-u^a\cdot\nabla B^d-R_{N,2}^a,\\
&\text{div}u^d=\text{div} B^d=0,
\end{cases}
\end{align}
with the initial and boundary conditions
\begin{align}\label{ubdcdiffer}
\begin{cases}
(u^d,B^d)(0,r,z)=(u^{d0},B^{d0})(r,z),\quad (r,z)\in\Omega,\\
u^d\cdot\mathbf{n}|_{\partial\Omega}= B^d\cdot\mathbf{n}|_{\partial\Omega}=0,\\
(u^d,B^d)(t,r,z)\to (0,0) \quad as \quad r\to \infty,\quad \mbox{if}\quad R_2=\infty;\\
u^d(t,r,z)=u^d(t,r,z+2\pi), B^d(t,r,z)=B^d(t,r,z+2\pi).
\end{cases}
\end{align}
\begin{lemma}\label{le4.5}
Let $s\geq3$ and $(u^a,B^a,W^a)$,$R_{N}^a=(R_{N,1}^a$, $R_{N,2}^a)$ $\in L_{loc}^{\infty}(H^s)$ as in Lemma \ref{4.1}. Then, there exists a constant $C>0$ such that \begin{align}\label{differenceest}
&\frac{d}{dt}\|(u^d,B^d)\|_{H^s}^2\leq C (1+\|(u^a,B^a)\|_{H^{s+1}}+\|(u^d,B^d)\|_{H^s})\|(u^d,B^d)\|_{H^s}^2
+\|R_{N}^a\|_{H^s}^2.
\end{align}
\end{lemma}
\begin{proof}
We first obtain, by applying the operator $\partial^{\alpha} (|\alpha|\leq s)$
to $\eqref{5.1}_1$, $\eqref{5.1}_2$, then taking the $L^2$
inner product of the resulting equations
with $\partial^{\alpha}u^d,\partial^{\alpha}B^d$ respectively, and summing them up, that
\begin{align}\label{difference}
\frac{1}{2}\frac{d}{dt}\|(u^d,B^d)\|_{H^s}^2\nonumber&=-\int_{\Omega}\partial^\alpha(u^d\cdot\nabla v_0+v_0\cdot\nabla u^d-B^d\cdot\nabla H_0-H_0\cdot\nabla B^d+\nabla W^d)\cdot\partial^\alpha u^d dx
\nonumber\\&\quad-\int_{\Omega}\partial^\alpha(u^d\cdot\nabla H_0+v_0\cdot\nabla B^d-B^d\cdot\nabla v_0-H_0\cdot\nabla u^d)\cdot\partial^\alpha B^d dx
  \nonumber\\&\quad+\int_{\Omega}\partial^\alpha(B^d\cdot\nabla B+B^a\cdot\nabla B^d-u^d\cdot\nabla u-u^a\cdot\nabla u^d-R_{N,1}^a)\cdot\partial^\alpha u^d dx
  \nonumber\\&\quad+\int_{\Omega}\partial^\alpha(B^d\cdot\nabla u+B^a\cdot\nabla u^d-u^d\cdot\nabla B-u^a\cdot\nabla B^d-R_{N,2}^a)\cdot\partial^\alpha B^d dx.
\end{align}
Then we need to estimate the right-hand side of \eqref{difference}.
It is easy to prove that
\begin{align*}
&\int_{\Omega}\partial^\alpha(u^d\cdot\nabla v_0)\cdot\partial^\alpha u^d dx\leq C\|u^d\|_{H^s}^2;\\
&\int_{\Omega}\partial^\alpha(B^d\cdot\nabla H_0)\cdot\partial^\alpha u^d dx\leq C\|(u^d,B^d)\|_{H^s}^2;\\
&\int_{\Omega}\partial^\alpha(u^d\cdot\nabla u^a)\cdot\partial^\alpha u^d+\partial^\alpha(B^d\cdot\nabla u^a)\cdot\partial^\alpha B^d dx\leq C\|u^a\|_{H^{s+1}}\|(u^d,B^d)\|_{H^s}^2;\\
&\int_{\Omega}\partial^\alpha(u^d\cdot\nabla H_0)\cdot\partial^\alpha B^d dx\leq C\|(u^d,B^d)\|_{H^s}^2;\\
&\int_{\Omega}\partial^\alpha(B^d\cdot\nabla v_0)\cdot\partial^\alpha B^d dx\leq C\|B^d\|_{H^s}^2;\\
&\int_{\Omega}\partial^\alpha(u^d\cdot\nabla B^a)\cdot\partial^\alpha B^d+\partial^\alpha(B^d\cdot\nabla B^a)\cdot\partial^\alpha u^d dx\leq C\|B^a\|_{H^{s+1}}\|(u^d,B^d)\|_{H^s}^2;\\
&\int_{\Omega}\partial^\alpha (R_{N,1}^a)\cdot\partial^\alpha u^d dx\leq C(\|u^d\|_{H^s}^2+\|R_{N,1}^a\|_{H^s}^2);\\
&\int_{\Omega}\partial^\alpha (R_{N,2}^a)\cdot\partial^\alpha B^d dx\leq C(\|B^d\|_{H^s}^2+\|R_{N,2}^a\|_{H^s}^2).
\end{align*}
Taking $\text{div}$ on $\eqref{5.1}_1$, we get the Neumann problem
\begin{align*}
\begin{cases}
-\Delta W^d=
f,\quad\quad\quad in \quad\Omega,\\
\frac{\partial W^d}{\partial \mathbf{n}}=
g,  \quad\quad\quad\quad on \quad \partial\Omega,
\end{cases}
\end{align*}
with $f=\text{div}[u^d\cdot\nabla v_0+v_0\cdot\nabla u^d-B^d\cdot\nabla H_0-H_0\cdot\nabla B^d-B^d\cdot\nabla B-B^a\cdot\nabla B^d+u^d\cdot\nabla u+u^a\cdot\nabla u^d+R_{N,1}^a]$ and $g=-[u^d\cdot\nabla v_0+v_0\cdot\nabla u^d-B^d\cdot\nabla H_0-H_0\cdot\nabla B^d-B^d\cdot\nabla B-B^a\cdot\nabla B^d+u^d\cdot\nabla u+u^a\cdot\nabla u^d+R_{N,1}^a]\cdot \mathbf{n}$. 
A standard elliptic estimate gives
\begin{align*}
\|\nabla W^d\|_{H^{s}}\leq C\|(u^a,B^a)\|_{H^{s}}\|(u^d,B^d)\|_{H^{s}}+C\|(u^d,B^d,R_{N,1}^a)\|_{H^{s}}.
\end{align*}
Thus, one has, by using $\text{div}v_0=\text{div}u^d=\text{div} u^a=0$
and the boundary conditions \eqref{ubdcdiffer}, that
\begin{align*}
&\int_{\Omega}\partial^\alpha(\nabla W^d)\cdot\partial^\alpha u^d dx\leq C\|(u^a,B^a)\|_{H^{s}}\|(u^d,B^d)\|_{H^{s}}^2+C\|(u^d,B^d,R_{N,1}^a)\|^2_{H^{s}};\\
&\int_{\Omega}\partial^\alpha(v_0\cdot\nabla u^d)\cdot\partial^\alpha u^d dx
\leq C\|u^d\|_{H^s}^2;\\
&\int_{\Omega}\partial^\alpha(u^d\cdot\nabla u^d)\cdot\partial^\alpha u^d dx
\leq C\|u^d\|_{H^s}^3;\\
&\int_{\Omega}\partial^\alpha(u^a\cdot\nabla u^d)\cdot\partial^\alpha u^d dx
\leq C\|u^a\|_{H^s}\|u^d\|_{H^s}^2;\\
&\int_{\Omega}\partial^\alpha(v_0\cdot\nabla B^d)\cdot\partial^\alpha B^d dx
\leq C\|B^d\|_{H^s}^2;\\
&\int_{\Omega}\partial^\alpha(u^d\cdot\nabla B^d)\cdot\partial^\alpha B^d dx
\leq C\|u^d\|_{H^s}\|B^d\|_{H^s}^2;\\
&\int_{\Omega}\partial^\alpha(u^a\cdot\nabla B^d)\cdot\partial^\alpha B^d dx
\leq C\|u^a\|_{H^s}\|B^d\|_{H^s}^2.
\end{align*}
It remains to estimate the higher-order terms in $\eqref{5.1}_{1}$ involving
$$B^d\cdot\nabla\partial^\alpha B^d\cdot\partial^\alpha u^d,\quad
B^a\cdot\nabla\partial^\alpha B^d\cdot\partial^\alpha u^d,\quad
H_0\cdot\nabla\partial^\alpha B^d\cdot\partial^\alpha u^d,$$
and
the higher-order terms in $\eqref{5.1}_{2}$ involving
$$B^d\cdot\nabla\partial^\alpha u^d\cdot\partial^\alpha B^d,\quad
B^a\cdot\nabla\partial^\alpha u^d\cdot\partial^\alpha B^d,\quad
H_0\cdot\nabla\partial^\alpha u^d\cdot\partial^\alpha B^d.$$
Notice that the boundary conditions \eqref{ubdcdiffer}
and $\text{div} B^d=0$. We find that
\begin{align*}
&\int_{\Omega}(B^d\cdot\nabla\partial^\alpha B^d\cdot\partial^\alpha u^d+B^d\cdot\nabla\partial^\alpha u^d\cdot\partial^\alpha B^d)dx
\\&=
-\int_{\Omega}\text{div}B^d \partial^\alpha B^d\cdot\partial^\alpha u^ddx=0.
\end{align*}
Similarly, the boundary conditions \eqref{ubdcdiffer}
and $\text{div} B^a=0, \text{div} H_0=0$ imply
\begin{align*}
&\int_{\Omega}(B^a\cdot\nabla\partial^\alpha B^d\cdot\partial^\alpha u^d+B^a\cdot\nabla\partial^\alpha u^d\cdot\partial^\alpha B^d)dx
=0
\end{align*}
and 
\begin{align*}
&\int_{\Omega}(H_0\cdot\nabla\partial^\alpha B^d\cdot\partial^\alpha u^d+H_0\cdot\nabla\partial^\alpha u^d\cdot\partial^\alpha B^d)dx
=0.
\end{align*}
So, we have
\begin{align*}
&\int_{\Omega}[\partial^\alpha(H_0\cdot\nabla B^d)\cdot\partial^\alpha u^d
  +\partial^\alpha(H_0\cdot \nabla u^d)\cdot\partial^\alpha B^d] dx\leq C\|(u^d,B^d)\|_{H^s}^2;\\
&\int_{\Omega}[\partial^\alpha(B^d\cdot\nabla B^d)\cdot\partial^\alpha u^d
  +\partial^\alpha(B^d\cdot\nabla u^d)\cdot\partial^\alpha B^d] dx\leq C\|u^d\|_{H^s}\|B^d\|_{H^s}^2;\\
&\int_{\Omega}[\partial^\alpha(B^a\cdot\nabla B^d)\cdot\partial^\alpha u^d
  +\partial^\alpha(B^a\cdot\nabla u^d)\cdot\partial^\alpha B^d] dx\leq C\|B^a\|_{H^{s+1}}\|(u^d,B^d)\|_{H^s}^2.
  \end{align*}
Hence, putting the above estimates into \eqref{difference}, leads to \eqref{differenceest}.
In conclusion, we finish the proof of Lemma \ref{le4.5}.

\end{proof}

\subsection{Proof of nonlinear instability}
Roughly speaking, our goal is to find a family of unstable axisymmetric solutions to \eqref{2.1'}-\eqref{linearized-cartesian-boundary}.
Let $(u^a,B^a,W^a)(t)$ be the axisymmetric approximate solution as in Lemma \ref{4.1}, for which $N$ will be chosen later.  For any $\delta>0$, let $(u^\delta,B^\delta,W^\delta)(t)$ be the unique axisymmetric solution of \eqref{2.1'}-\eqref{linearized-cartesian-boundary} with initial data $(u^a,B^a,W^a)(0,r,z)=\delta(\phi_{1},\xi_{1},\eta_1)(0,r,z)$, which was established by Lemma \ref{Local-existence}. 
Note that
\begin{align*}
(u^d,B^d,W^d)(t)=(u^\delta-u^a,B^\delta-B^a,W^\delta-W^a)(t),
\end{align*}
with
\begin{align*}
(u^d,B^d,W^d)(0)=0.
\end{align*}
By Lemma \ref{le4.5}, we get 
\begin{align}\label{detaildifference}
\frac{d}{dt}\|(u^d,B^d)\|_{H^s}^2&\leq C(1+\|(u^a,B^a)\|_{H^{s+1}}+\|(u^d,B^d)\|_{H^s})\|(u^d,B^d)\|_{H^s}^2\nonumber
+\|R_{N}^a\|_{H^s}^2\nonumber\\
&\leq C(1+\|(u^a,B^a)\|_{H^{s+1}}+\|(u^d,B^d)\|_{H^s})\|(u^d,B^d)\|_{H^s}^2\nonumber\\
&\quad+C_{s,N}\delta^{2(N+1)}e^{2(N+1)\Lambda t}.
\end{align}
On the one hand, let 
\begin{align*}
T=\sup\left\{t\big| \|(u^a,B^a)(t)\|_{H^{s+1}}\leq\frac{\zeta}{2}, \quad\|(u^d,B^d)(t)\|_{H^s}\leq\frac{\zeta}{2}\right\},
\end{align*}
where
$\zeta>0$ is a small constant which ensures local existence.
According to $(u^d,B^d)(0)=0$ and $\|(u^a,B^a)(0)\|_{H^{s+1}}=O(\delta)$, $T$ is
well defined for $\delta$ small enough.
Note that $T^\delta$ satisfies $\theta=\delta e^{\Lambda T^{\delta}}.$
Then, we claim $T^{\delta}\leq T$ for $\theta$ small enough. In fact, if otherwise, suppose that $T^{\delta}> T$, through the construction of the axisymmetric approximate solution in \eqref{approximate}, one has
\begin{align*}
\|(u^a,B^a)(t)\|_{H^{s+1}}&\leq C\sum_{j=1}^{N}\delta^{j}\|\phi_{j}\|_{H^{s+1}}+C\sum_{j=1}^{N}\delta^{j}\|\xi_{j}\|_{H^{s+1}}
\leq \sum_{j=1}^{N}C_j\delta^{j}e^{j\Lambda t}
\\&\leq \sum_{j=1}^{N}C_j(\delta e^{\Lambda T^{\delta}})^{j}=\sum_{j=1}^{N}C_j\theta^j\leq\frac{\zeta}{4}
\end{align*}
for $0\leq t\leq T.$
Together with the definition of $T$ and \eqref{detaildifference}, we obtain
\begin{align*}
\frac{d}{dt}\|(u^d,B^d)(t)\|_{H^s}^2
&\leq C(1+\frac{\zeta}{2})\cdot\|(u^d,B^d)\|_{H^s}^2
+C_{s,N}\delta^{2(N+1)}e^{2(N+1)\Lambda t}
\end{align*}
for $0\leq t\leq T$.
Then choosing $N$ sufficiently large such that
\begin{align*}
 C(1+\frac{\zeta}{2})<2(N+1)\Lambda,
\end{align*}
and using Gronwall inequality, we find
\begin{align*}
\|(u^d,B^d)(t)\|_{H^s}
&\leq C\delta^{(N+1)}e^{(N+1)\Lambda t}\leq C\theta^{N+1}<\frac{\zeta}{2}
\end{align*}
for $0\leq t\leq T$.
It leads to a contradiction to the definition of $T$. Consequently, we infer $T^{\delta}\leq T$ for $\theta$ small enough.

Based on the above statements, it is time to verify the nonlinear instability. Letting $t=T^{\delta}$, recalling the construction of the axisymmetric approximate solution in \eqref{approximate}, we infer
\begin{align*}
\|(u^a,B^a)(T^{\delta})\|_{L^2}
&\geq \delta\|(\phi_1,\xi_1)(T^{\delta})\|_{L^2}-\sum_{j=2}^{N}\delta^{j}\|(\phi_j,\xi_{j})(T^{\delta})\|_{L^2}
\\&\geq C\delta e^{\Lambda T^\delta}-\sum_{j=2}^{N}C_j\delta^{j}e^{j\Lambda T^\delta}
\\&=C \theta-\sum_{j=2}^{N}C_j\theta^j\geq\frac{C\theta}{2}.
\end{align*}
Therefore, we obtain
\begin{align*}
\|(u^\delta,B^\delta)(T^{\delta})\|_{L^2}
&\geq \|(u^a,B^a)(T^{\delta})\|_{L^2}-\|[(u^\delta,B^\delta)-(u^a,B^a)](T^{\delta})\|_{L^2}
\\&\geq \frac{C\theta}{2}-\|(u^d,B^d)(T^{\delta})\|_{L^2}
\\&\geq \frac{C\theta}{2}-C\theta^{N+1}
\geq \frac{\theta}{4}=\epsilon_0>0.
\end{align*}
We have thus proved Theorem \ref{nonlinearunstable}.

\section{Euler equations}\label{Eulercase}

Consider the Euler system for the incompressible
fluids
\begin{align}\label{Euler}
\begin{cases}
\partial_tv+v\cdot\nabla v+\nabla p=0,\\
\text{div}v=0.
\end{cases}
\end{align}
Here, $v(t,r,z)$ is the velocity, $p(t,r,z)$ is the pressure, where $t\geq0, (r,\theta,z)\in\Omega=\{R_1\leq r\leq R_2, (\theta,z)\in\mathbb{T}_{2\pi}\times\mathbb{T}_{2\pi}\}$, with $0\leq R_1<R_2<\infty.$ We impose the initial and boundary conditions
\begin{align}\label{Eulerboundary}
\begin{cases}
 v(0,r,z)=v^0(r,z),\\
 v_r(t,R_1,z)=0, \\
 v_r(t,R_2,z)=0, \\
 v(t,r,z)=v(t,r,z+2\pi).
\end{cases}
\end{align}
Let the steady solution $v_0(r,z)$ be the rotating flow
\begin{align}\label{Eulersteady}
 v_0(r,z)=\omega(r)r \mathbf{e}_\theta,
\end{align}
with angular velocity $\omega\in C^{2}(R_1,R_2)$,
and the perturbations be
\begin{align*}
 u(t,r,z)=v(t,r,z)-v_0(r,z), \quad W(t,r,z)=p(t,r,z)-p_0(r,z).
\end{align*}
The linearized Euler system around a given steady state $v_0(r,z)$
in cylindrical coordinates can be reduced to the following equations:
\begin{align}\label{linearizedEuler}
\begin{cases}
  \partial_t u_r=2\omega(r)u_{\theta}-\partial_{r} W,\\
\partial_t u_\theta=-\frac{u_r}{r}\partial_r(r^2\omega(r)),\\
\partial_t u_z=-\partial_{z} W,
\end{cases}
\end{align}
with the initial and boundary conditions
\begin{align}\label{newEulerboundary}
\begin{cases}
 u(0,r,z)=u^0(r,z),\\
 u_r(t,R_1,z)=0, \\
 u_r(t,R_2,z)=0, \\
 u(t,r,z)=u(t,r,z+2\pi).
\end{cases}
\end{align}
 Here, we list the nonlinear Rayleigh stability and instability results of the Euler system \eqref{Euler} in the following two theorems and give the sketches of the proofs in Section \ref{RSI}-\ref{newRSI}.
\begin{theorem}\label{nonEuler}
Assume that
the function $\omega(r)\in C^{2}(R_1,R_2)$, 
 the linear stability condition $\frac{\partial_r(\omega^2r^4)}{r^3}>0$
 and the existence of global axisymmetric weak solution to \eqref{Euler}-\eqref{Eulerboundary} satisfying:\\
1)The total energy defined by \eqref{Eulerenergy} is non-increasing with respect to $t$.\\
2)The distance function $d(t)$ defined by \eqref{distanceEuler} is continuous with respect to $t$.\\
3)The functional $\int_{\Omega}g(rv_\theta) dx$ is conserved for the weak solutions.\\
Then the steady state $v_0(x)$ in \eqref{Eulersteady} of \eqref{Euler}-\eqref{Eulerboundary} is nonlinearly stable, in the sense that
if
$d(0)=d(v^0(x),v_0(x))$ is small enough,
then the weak solution $v(t,x)$ of \eqref{Euler}-\eqref{Eulerboundary} satisfies
\begin{align*}
d(t)=d(v(t,x),v_0(x))
\leq d(0)
\end{align*}
for any $t\geq 0$.
\end{theorem}
\begin{theorem}\label{nonunEuler}
Assume that 
the function $\omega(r)\in C^{\infty}(R_1,R_2)$.
If $\exists  r_0\in(R_1,R_2)$ such that $\frac{\partial_r(\omega^2r^4)}{r^3}|_{r=r_0}<0$,
then the steady state $v_0(x)$ in \eqref{Eulersteady} of \eqref{Euler}-\eqref{Eulerboundary} is nonlinearly unstable in the sense that for any $s\geq0$ large enough, there exists $\epsilon_{0}>0$, such that for any small $\delta>0$, there exists a family of classical axisymmetric solutions $v^\delta(t,x)$ to \eqref{Euler}-\eqref{Eulerboundary} satisfying
\begin{align*}
\|v^\delta(0,x)-v_0(x)\|_{H^s}\leq\delta
\end{align*}
and
\begin{align*}
\sup_{0\leq t\leq T^\delta}\|v^\delta(t,x)-v_0(x)\|_{L^2}\geq\epsilon_{0},
\end{align*}
where $T^\delta=O(|\ln\delta|).$
\end{theorem}

\subsection{Nonlinear Rayleigh stability}\label{RSI}

We assume Rayleigh stability condition $\frac{\partial_r(\omega^2r^4)}{r^3}>0$, which implies $\omega\neq 0$. Without loss of generality, we take $\omega(r)>0.$
Let $\tau(r)=r^2\omega(r)$, then $r$ can be written as $r=\beta(\tau)$, since $\frac{\partial_r(\omega^2r^4)}{r^3}>0$.
Define the total energy $E(v)$:
\begin{align}\label{Eulerenergy}
E(v)=\frac{1}{2}\int_{\Omega}(|v_r|^2+|v_z|^2+|v_\theta|^2) dx,
\end{align}
and the Casimir-energy $E_{c}(v)$:
\begin{align*}
E_{c}(v)=E(v)+\int_{\Omega}g(rv_\theta) dx,
\end{align*}
where $g(\tau)\in C^2(\mathbb{R})$ is a function which satisfies
\begin{align}\label{eulerextension}
g'(\tau)=\begin{cases}
 -\omega(\beta(\tau)),\quad  \tau\in (R_1\omega(R_1),R_2\omega(R_2)) ,\\
\text{some $C_0^\infty$ extension function such that $1+r^2g''(\tau)\geq0$, }  \tau\notin (R_1\omega(R_1),R_2\omega(R_2)).
\end{cases}
\end{align}
Hence, we deduce from \eqref{eulerextension} that $g'(\tau)=-\omega(\beta(\tau))$ and
$$1+r^2(\tau)g''(\tau)=1-r^2\frac{\partial_r(\omega(r))}{\partial_r(r^2\omega(r))}\bigg|_{r=\beta(\tau)}=\frac{4r^3\omega^2(r)}{\partial_r(r^4\omega^2(r))}\bigg|_{r=\beta(\tau)}$$
for $\tau\in (R_1\omega(R_1),R_2\omega(R_2)).$
The steady state $v_0=\omega r\mathbf{e}_\theta$ has the following variational structure.
First, $v_0$ is a critical point of $E_c$, i.e.,
\begin{align*}
\langle \nabla E_c(v_0),\delta v\rangle
=\int_\Omega(r\omega+rg'(r^2\omega))\delta v_\theta
 dx=0,
\end{align*}
and the second-order variation of $E_c(v)$
at $v_0$,
\begin{align*}
\langle \nabla^2E_c(v_0)\delta v,\delta v\rangle
&=\int_{\Omega}[|\delta v_r|^2+|\delta v_z|^2+(1+r^2g''(r^2\omega))|\delta v_\theta|^2] dx
\notag\\
&=\int_{\Omega}\left(|\delta v_r|^2+|\delta v_z|^2+\frac{4r^3\omega^2(r)}{\partial_r(r^4\omega^2(r))}|\delta v_\theta|^2\right)
dx>0
\end{align*}
for any $\delta v=(\delta v_r,\delta v_\theta,\delta v_z)\neq0\in L^2$.
Define the nonlinear distance functional
\begin{align}\label{distanceEuler}
d((v_r,v_\theta,v_z),(0,r\omega,0))=d(t)&=d_1(t)+d_2(t),
\end{align}
where
\begin{align*}
d_1(t)&=d_1((v_r,v_z),(0,0))=\frac{1}{2}\int_{\Omega}(|v_r|^2+|v_z|^2)dx,\\
d_2(t)&=d_2(v_\theta,r\omega)=\int_{\Omega}\left[\frac{1}{2}|v_\theta-r\omega|^2+g(rv_\theta)-g(r^2\omega)-rg'(r^2\omega)(v_\theta-r\omega)\right]dx.
\end{align*}
It is easy to check that $d((v_r,v_\theta,v_z),(0,r\omega,0))$ is well defined for $v_r,v_\theta,v_z\in L^2$ and is a distance functional
since
\begin{align*}
d_2(v_\theta,r\omega)&=\frac{1}{2}\int_{\Omega}(1+r^2g''(r\tilde{v}_\theta))(v_\theta-r\omega)^2dx\geq0,
\end{align*}
where $\tilde{v}_\theta$ is between $r\omega$ and $v_\theta$.
It is obvious that
\begin{align*}
d(t) = E_c(v)-E_c(v_0).
\end{align*}
Thus $d(t)\leq d(0)$. This proves Theorem \ref{nonEuler} on the nonlinear Rayleigh stability.

\subsection{Nonlinear Rayleigh instability}\label{newRSI}

If $\exists  r_0$ such that $\Upsilon(r)|_{r=r_0}=\frac{\partial_r(\omega^2r^4)}{r^3}|_{r=r_0}<0$,
it follows from \eqref{linearizedEuler} that
\begin{equation}\label{eulersecondorder}
\partial_{tt}\left(
\begin{array}
[c]{c}%
u_r\\
u_z
\end{array}
\right)= \mathcal{P}\begin{pmatrix}-\frac{\partial_r(\omega^2r^4)}{r^3}u_r\\
0\end{pmatrix}:=\widetilde{\mathbb{L}}_E \left(
\begin{array}
[c]{c}%
u_r\\
u_z
\end{array}
\right)
.
\end{equation}
In this case, two fundamental differences between Euler equations and MHD equations are:

1) Unstable continuous spectra vs discrete spectrum

By Theorem \ref{linearstability}, the MHD equations have unstable discrete spectrum and the unstable subspace for MHD equations is finite dimensional. However, we shall show that Euler equations have unstable continuous spectra and the unstable subspace for Euler equations is infinite dimensional. By the Rayleigh instability condition $\Upsilon(r)|_{r=r_0}<0$, 
we know that $\text{range}(\Upsilon(r))=[-a_1,a_2]$ for some real numbers $a_1:=-\min\limits_{r\in [R_1,R_2]}\frac{\partial_r(\omega^2r^4)}{r^3}>0$ and $a_2.$
It is obvious that $\widetilde{\mathbb{L}}_E$ is a bounded and symmetric operator, and hence $\widetilde{\mathbb{L}}_E$ is a self-adjoint operator. Then, by the same arguments as in the proof of Lemma 4.6
in \cite{LinWang2022}, we can
prove that
$\sigma_{ess}(\widetilde{\mathbb{L}}_E)\supset\text{range}(\Upsilon(r))=[-a_1,a_2].$ Therefore, we obtain a linear Rayleigh instability result of the system \eqref{Euler} similar to Theorem 1.2 in \cite{LinWang2022}.

2) Define
\begin{align}\label{smalllambdak}
\lambda_k^2=\sup\frac{\int_\Omega-\frac{\partial_r(\omega^2r^4)}{r^3}|u_r|^2dx}{\int_\Omega\left(|u_r|^2+\frac{|\partial_r(ru_r)|^2}{k^2r^2}\right)dx}>0
\end{align}
for any $k\in \mathbb{Z}.$
Then
\begin{align}\label{limlambdak}
\lambda_k^2\rightarrow \sup_{u_r\in L^2}\frac{\int_\Omega-\frac{\partial_r(\omega^2r^4)}{r^3}|u_r|^2dx}{\int_\Omega|u_r|^2dx}=-\min_{r\in [R_1,R_2]}\frac{\partial_r(\omega^2r^4)}{r^3}=a_1 \quad\text{as} \quad k\rightarrow\infty.
\end{align}
Note that the largest growth rate of Euler system is $\sqrt{a_1}$ and
 the maximal unstable growing mode of Euler system is obtained in the high frequency limits.
But from
Theorem \ref{linearstability}, we know that the maximal unstable growing mode of MHD system is obtained in the low frequency, and it is stable for high frequencies.
Moreover, a nonlinear Rayleigh instability result similar to Theorem \ref{nonlinearunstable} can be obtained through replacing Lemma \ref{linearexponential} by Lemma \ref{linearexponentialEuler}.
\begin{lemma}\label{linearexponentialEuler}
Let $u(t)$ be the solution to the linearized equations \eqref{linearizedEuler} with the initial and boundary conditions
\eqref{newEulerboundary}. Then, for any $s\geq0$, there holds
\begin{align*}
\|u(t)\|_{H^s}\lesssim e^{\sqrt{a_1} t}\|u(0)\|_{H^{s}}.
\end{align*}
\end{lemma}
\begin{proof}
Multiplying \eqref{eulersecondorder} by $(\partial_t u_r,\partial_t u_z)$, and then integrating it on $\Omega$, we obtain
\begin{align*}
\frac{1}{2}\frac{d}{dt}(\|\partial_tu_r(t)\|_{L^2}^2+\|\partial_tu_z(t)\|_{L^2}^2)&=
\int_\Omega\begin{pmatrix}-\frac{\partial_r(\omega^2r^4)}{r^3}u_r\\
0\end{pmatrix}\cdot\begin{pmatrix}\partial_t u_r\\ \partial_t u_z\end{pmatrix}dx\\
&=\frac{1}{2}\frac{d}{dt}\int_\Omega-\frac{\partial_r(\omega^2r^4)}{r^3}|u_r|^2dx.
\end{align*}
Integrating the above inequality over time, using the definition of $a_1$ and the fact
$$\frac{d}{dt}(\|u_r(t)\|_{L^2}^2+\|u_z(t)\|_{L^2}^2)^{\frac{1}{2}}\leq (\|\partial_tu_r(t)\|_{L^2}^2+\|\partial_tu_z(t)\|_{L^2}^2)^{\frac{1}{2}},$$ one has
\begin{align}\label{0orderestimate}
\frac{d}{dt}(\|u_r(t)\|_{L^2}^2+\|u_z(t)\|_{L^2}^2)^{\frac{1}{2}}&\leq(\|\partial_tu_r(t)\|_{L^2}^2+\|\partial_tu_z(t)\|_{L^2}^2)^{\frac{1}{2}}\nonumber\\
&\leq 2\|\partial_tu_r(0)\|_{L^2}+2\|\partial_tu_z(0)\|_{L^2}\nonumber\\
&\quad
+
\left(\int_\Omega-\frac{\partial_r(\omega^2r^4)}{r^3}|u_r(t)|^2dx-\int_\Omega-\frac{\partial_r(\omega^2r^4)}{r^3}|u_r(0)|^2dx\right)^{\frac{1}{2}}\nonumber\\
&\leq C\|u(0)\|_{L^2}+\sqrt{a_1}\|u_r(t)\|_{L^2}.
\end{align}
By Gronwall inequality, we have
\begin{align}\label{G1orderestimate}
\|u_r(t)\|_{L^2}+\|u_z(t)\|_{L^2}\leq \sqrt{2}(\|u_r(t)\|_{L^2}^2+\|u_z(t)\|_{L^2}^2)^{\frac{1}{2}}\lesssim e^{\sqrt{a_1}t}\|u(0)\|_{L^2}.
\end{align}
Using \eqref{G1orderestimate}, we deduce from $\eqref{linearizedEuler}_2$ that
\begin{align}\label{G2orderestimate}
\|u_\theta(t)\|_{L^2}\lesssim \|u_\theta(0)\|_{L^2}+\int_0^t\|u_r (\tau)\|_{L^2}d\tau\lesssim e^{\sqrt{a_1}t}\|u (0)\|_{L^2}.
\end{align}
Combing \eqref{G1orderestimate} and \eqref{G2orderestimate}, one has
\begin{align*}
\|u(t)\|_{L^2}\lesssim e^{\sqrt{a_1}t}\|u (0)\|_{L^2}.
\end{align*}
Moreover, taking $\partial_z^{\alpha}$ with $(|\alpha|= s)$ on \eqref{eulersecondorder}, we deduce
\begin{align*}
\partial_{tt}\left(
\begin{array}
[c]{c}%
\partial_z^{\alpha}u_r\\
\partial_z^{\alpha}u_z
\end{array}
\right)
=\mathcal{P}\begin{pmatrix}-\frac{\partial_r(\omega^2r^4)}{r^3}\partial_z^{\alpha}u_r\\
0\end{pmatrix}.
\end{align*}
Similar to \eqref{0orderestimate} and \eqref{G1orderestimate}, we get
\begin{align}\label{z1orderestimate}
\|\partial_z^{\alpha}u_r(t)\|_{L^2}+\|\partial_z^{\alpha}u_z(t)\|_{L^2}\lesssim e^{\sqrt{a_1}t}\|u (0)\|_{H^s}
\end{align}
with $|\alpha|= s.$
Using \eqref{z1orderestimate}, we deduce from $\eqref{linearizedEuler}_2$ that
\begin{align}\label{z2orderestimate}
\|\partial_z^{\alpha}u_\theta(t)\|_{L^2}\lesssim \|\partial_z^{\alpha}u_\theta(0)\|_{L^2}+\int_0^t\|\partial_z^{\alpha}u_r (\tau)\|_{L^2}d\tau\lesssim e^{\sqrt{a_1}t}\|u (0)\|_{H^s}
\end{align}
with $|\alpha|= s.$
Thus, we deduce from \eqref{z1orderestimate} and \eqref{z2orderestimate} that
\begin{align}\label{z3orderestimate}
\|\partial_z^{\alpha}u(t)\|_{L^2}\lesssim e^{\sqrt{a_1}t}\|u (0)\|_{H^s}, \quad |\alpha|= s.
\end{align}
Next, we need to prove $\|\partial_{r}^{\alpha}u(t)\|_{L^2}\lesssim e^{\sqrt{a_1} t}\|u(0)\|_{H^{s}}$ $(|\alpha|=s)$ by an induction on $s$.
For $s=1$, taking $\partial_r$ on both side of $\eqref{linearizedEuler}_2$, we have
\begin{align*}
\partial_t\partial_ru_\theta&=-\partial_r(ru_r)\frac{\partial_r(r^2\omega)}{r^2}-ru_r\partial_r\left(\frac{\partial_r(r^2\omega)}{r^2}\right)\nonumber\\
&=\partial_zu_z\frac{\partial_r(r^2\omega)}{r}-ru_r\partial_r\left(\frac{\partial_r(r^2\omega)}{r^2}\right).
\end{align*}
By \eqref{z1orderestimate}, it follows that
\begin{align}\label{1orderrestimate}
\|\partial_ru_\theta(t)\|_{L^2}&\lesssim \|\partial_ru_\theta(0)\|_{L^2}+\int_0^t(\|\partial_zu_z(\tau)\|_{L^2}+\|u_r(\tau)\|_{L^2})d\tau
\nonumber\\&\lesssim \|\partial_ru_\theta(0)\|_{L^2}+e^{\sqrt{a_1}t}\|u(0)\|_{H^1}
\nonumber\\&\lesssim e^{\sqrt{a_1}t}\|u(0)\|_{H^1}.
\end{align}
Then, taking $\partial_r$ on both side of $\eqref{linearizedEuler}_1$ and $\eqref{linearizedEuler}_3$, we have
\begin{equation*}
\partial_t\left(
\begin{array}
[c]{c}%
\partial_ru_r\\
\partial_ru_z
\end{array}
\right)  = \partial_r\mathcal{P}\begin{pmatrix}2\omega(r)u_\theta\\
0\end{pmatrix}
.
\end{equation*}
By \eqref{1orderrestimate}, we obtain
\begin{equation}\label{11orderestimate}
\left\|\left(
\begin{array}
[c]{c}%
\partial_ru_r\\
\partial_ru_z
\end{array}
\right) (t)\right\|_{L^2}\lesssim \left\|\left(
\begin{array}
[c]{c}%
\partial_ru_r\\
\partial_ru_z
\end{array}
\right) (0)\right\|_{L^2}+\int_0^t \|\partial_ru_\theta(s)\|_{L^2}ds \lesssim e^{\sqrt{a_1}t}\|u(0)\|_{H^1}.
\end{equation}
Combining \eqref{1orderrestimate} and \eqref{11orderestimate}, we obtain
 \begin{align*}
\|\partial_{r}u(t)\|_{L^2}
\lesssim e^{\sqrt{a_1} t}\|u(0)\|_{H^{1}}.
\end{align*}
Assume it holds for $|\alpha|\leq s-1$, we will prove $|\alpha|=s$ by induction. Similarly, we get
\begin{align*}
\partial_t\partial_r^\alpha u_\theta&=\sum\limits_{0\leq|\beta|,|\gamma|\leq s-1,|\beta|+|\gamma|=s-1}
\partial_r^{\beta}\partial_zu_z\cdot\partial_r^{\gamma}\left(\frac{\partial_r(r^2\omega)}{r^2}\right)\\
&\quad
-\sum\limits_{0\leq|\beta|,|\gamma|\leq s-1,|\beta|+|\gamma|=s-1}\partial_r^\beta (ru_r)\cdot\partial_r^{\gamma}\partial_r\left(\frac{\partial_r(r^2\omega)}{r^2}\right)
\end{align*}
and
\begin{equation*}
\partial_t\left(
\begin{array}
[c]{c}%
\partial_r^\alpha u_r\\
\partial_r^\alpha u_z
\end{array}
\right)  = \partial_r^\alpha \mathcal{P}\begin{pmatrix}2\omega(r)u_\theta\\
0\end{pmatrix}
\end{equation*}
with $|\alpha|=s.$
By the induction hypothesis and Duhamel's principle, one gets
\begin{align*}
\|\partial_{r}^{\alpha}u(t)\|_{L^2}
\lesssim e^{\sqrt{a_1} t}\|u(0)\|_{H^{s}}
\end{align*}
with $|\alpha|=s.$
Together with \eqref{z3orderestimate}, we get
\begin{align*}
\|u(t)\|_{H^s}
\lesssim e^{\sqrt{a_1} t}\|u(0)\|_{H^{s}}.
\end{align*}
We complete the proof of the lemma.

\end{proof}

\subsection{Comparison of rotating flows
with and without magnetic fields}
Although the stability criterion of rotating Euler system is totally different from the MHD system, their most unstable growth rates are close if $\epsilon$ is small enough. Without loss of generality, we take $b(r)=1$, $R_1=0$ and $R_2=\mathfrak{R}<\infty$.

Let
the maximal unstable growth rate of MHD equations
$$\Lambda^2:=\sup_{\langle u_2,u_2\rangle:=1,u_2\in Y}\langle-\mathbb{B}^{'}\mathbb{L}\mathbb{B}Au_2,u_2\rangle,$$
and the unstable growth rate of MHD equations for each frequency $k$
$$\Lambda_k^2:=\sup_{\langle u^k_2,u^k_2\rangle=1,u_2\in Y}\langle-\mathbb{B}^{'}\mathbb{L}\mathbb{B}Au^k_2,u^k_2\rangle,$$
where $u_2^k=(u^k_r(r)e^{ikz},u^k_z(r)e^{ikz},B^k_\theta(r)e^{ikz})$.
\begin{proposition}\label{pp5.1}
Assume $\partial_r(\omega^2)|_{r=r_0}<0$.

1) \textbf{Rayleigh stable case:} If $\Upsilon(r)> 0$ for all $r\in [0,\mathfrak{R}]$, then for any frequency $k$,
the corresponding unstable growth rate $\Lambda_k^2\leq O(\epsilon^2)$ and
the maximal unstable growth rate of MHD equations $\Lambda^2\geq O(\epsilon^2)$.

2) \textbf{Rayleigh unstable case:} If $\Upsilon(r_0)< 0$ for some $r_0\in [0,\mathfrak{R}]$, then for any frequency $k$,
the corresponding unstable growth rate $\Lambda_k$ satisfies $\Lambda_k^2\in(\lambda_k^2  + O(\epsilon^2), \lambda_k^2+\varepsilon k^2+O(\epsilon^2))\subseteq(a_1 -\varepsilon_k + O(\epsilon^2), a_1+\varepsilon k^2+O(\epsilon^2))$,
where $a_1:=-\min\limits_{r\in [0,\mathfrak{R}]}\frac{\partial_r(\omega^2r^4)}{r^3}>0$, $\varepsilon>0$ is an arbitrary small constant, $\varepsilon_k>0$ satisfies $\lim\limits_{k\rightarrow\infty}\varepsilon_k=0.$
\end{proposition}

\begin{proof}
1)
By taking $u_r(r,z)=\frac{h(r)}{r\epsilon}e^{iz}$, $u_z(r,z)=i\frac{e^{iz}}{r\epsilon}\partial_rh$, $B_\theta=-i\frac{2\omega}{\epsilon^2r}he^{iz}$,
where $h(r)$ is the maximizer of $$\sup_{h_1(0)=0}\frac{-\int_{0}^{\mathfrak{R}}\partial_r(\omega^2)|h_1(r)|^2dr}{\int_{0}^{\mathfrak{R}}\left(\frac{1}{r}|\partial_rh_1(r)|^2+\frac{1}{r}|h_1(r)|^2\right)dr}\equiv B_0^2,$$ i.e.,
$-\int_{0}^{\mathfrak{R}}\partial_r(\omega^2)|h(r)|^2dr=B_0^2\int_{0}^{\mathfrak{R}}\left(\frac{1}{r}|\partial_rh(r)|^2+\frac{1}{r}|h(r)|^2\right)dr$.
We know the maximal unstable growth rate of MHD system has the lower bound
\begin{align*}
\Lambda^2&=\sup_{\langle u_2,u_2\rangle=1,u_2\in Y}\langle-\mathbb{B}^{'}\mathbb{L}\mathbb{B}Au_2,u_2\rangle\\
&=\sup\frac{\int_{\Omega}\left(-r\partial_r(\omega^2)|u_r|^2-\epsilon^2|\partial_{z}u_r|^2
-\epsilon^2|\partial_{z}u_z|^2
-\epsilon^2\left|\partial_{z}B_\theta+\frac{2\omega}{-\epsilon}u_r\right|^2\right)dx}{\int_\Omega (|u_r|^2+|u_z|^2+|B_\theta|^2) dx}\\
&\geq\frac{\int_{0}^{\mathfrak{R}}\left(-\frac{\partial_r(\omega^2)}{\epsilon^2}|h|^2-\frac{1}{r}|h|^2
-\frac{1}{r}|\partial_{r}h|^2\right)dr}{\int_{0}^{\mathfrak{R}} \left(\frac{1}{r\epsilon^2}|h|^2+\frac{1}{r\epsilon^2}|\partial_r h|^2+\frac{1}{\epsilon^4r}|2\omega h|^2\right) dr}\\
&=\frac{\left(B_0^2-\epsilon^2\right)\int_{0}^{\mathfrak{R}}\left(\frac{1}{r}|h|^2
+\frac{1}{r}|\partial_{r}h|^2\right)dr}{\int_{0}^{\mathfrak{R}} \left(\frac{1}{r}|h|^2+\frac{1}{r}|\partial_r h|^2+\frac{1}{\epsilon^2r}|2\omega h|^2 \right)dr}\geq C(B_0^2-\epsilon^2) \epsilon^2=O(\epsilon^2).
\end{align*}
For any frequency $k$ fixed,
one has for any unstable growth rate corresponding to the frequency $k$,
\begin{align*}
\Lambda_k^2&=\sup_{\langle u^k_2,u^k_2\rangle=1,u_2\in Y}\langle-\mathbb{B}^{'}\mathbb{L}\mathbb{B}Au^k_2,u^k_2\rangle\\
&=\sup\frac{\int_{\Omega}\left(-r\partial_r(\omega^2)|u^k_r|^2-\epsilon^2|\partial_{z}u^k_r|^2
-\epsilon^2|\partial_{z}u^k_z|^2
-\epsilon^2\left|\partial_{z}B^k_\theta+\frac{2\omega}{-\epsilon}u^k_r\right|^2\right)dx}{\int_\Omega (|u^k_r|^2+|u^k_z|^2+|B^k_\theta|^2) dx}\\
&=\sup\frac{\int_{0}^{\mathfrak{R}}\left(-\partial_r(\omega^2)|h_2|^2-\epsilon^2k^2r|\frac{h_2}{r}|^2
-\epsilon^2r|\frac{\partial_rh_2}{r}|^2
-r|h_1|^2\right)dr}{\int_{0}^{\mathfrak{R}} \left[r|\frac{h_2}{r}|^2+r|\frac{\partial_rh_2}{kr}|^2+r|\frac{1}{k\epsilon}\left(h_1+2\omega\frac{h_2}{r}\right)|^2\right] dr}\\
&\leq \epsilon^2\sup\frac{\int_{0}^{\mathfrak{R}}(-\partial_r(\omega^2)|h_2|^2
-|h_1|^2r)dr}{\int_{0}^{\mathfrak{R}} \left(\epsilon^2r|\frac{h_2}{r}|^2+\epsilon^2r|\frac{\partial_rh_2}{kr}|^2+r\frac{1}{k^2}|h_1+2\omega\frac{h_2}{r}|^2\right) dr}\leq O(\epsilon^2),
\end{align*}
where $u_2^k=(u^k_r,u^k_z,B^k_\theta)$ satisfies $\mathbb{B}Au_2^k=(h_1e^{ikz},\epsilon h_2e^{ikz})$ for any $(h_1(r),h_2(r))\in R(\mathbb{B})$, i.e.,
\begin{align}
&u_r^k(r,z)=\frac{h_2}{r}e^{ikz},\nonumber\\
&u_z^k(r,z)=i\frac{\partial_rh_2}{kr}e^{ikz},\nonumber\\
&B_\theta^k(r,z)=-i\frac{1}{k\epsilon}\left(h_1+2\omega\frac{h_2}{r}\right)e^{ikz}.\nonumber
\end{align}
This yields that the maximal unstable growth rate of MHD equations for each frequency $k$ tends to $0$, as $\epsilon$ tends to $0$.

2)
By taking $u_r(r,z)=h(r)e^{ikz}$, $u_z(r,z)=ie^{ikz}\frac{\partial_r(rh)}{kr}$, $B_\theta=0$.
Hence for any frequency $k$ fixed big enough,
the corresponding unstable growth rate has the following lower bound by using \eqref{limlambdak}
\begin{align*}
\Lambda_k^2&=\sup_{\langle u_2,u_2\rangle=1,u_2\in Y}\langle-\mathbb{B}^{'}\mathbb{L}\mathbb{B}Au_2,u_2\rangle\\
&=\sup\frac{\int_{\Omega}\left(-r\partial_r(\omega^2)|u_r|^2-\epsilon^2|\partial_{z}u_r|^2
-\epsilon^2|\partial_{z}u_z|^2
-\epsilon^2\left|\partial_{z}B_\theta+\frac{2\omega}{-\epsilon}u_r\right|^2\right)dx}{\int_\Omega (|u_r|^2+|u_z|^2+|B_\theta|^2) dx}\\
&\geq\sup\frac{\int_{0}^{\mathfrak{R}}\left(-\frac{\partial_r(\omega^2r^4)}{r^3}r|h|^2-\epsilon^2r|ikh|^2
-\epsilon^2r|\frac{\partial_r(rh)}{r}|^2\right)dr}{\int_0^{\mathfrak{R}} \left(r|h|^2+r|\frac{\partial_r(rh)}{kr}|^2\right) dr}\\
&=\lambda_k^2+O(\epsilon^2)\geq a_1-\varepsilon_k+O(\epsilon^2),
\end{align*}
where $\lambda_k$ is defined by \eqref{smalllambdak}, $\varepsilon_k>0$ satisfies $\lim\limits_{k\rightarrow\infty}\varepsilon_k=0$ by \eqref{limlambdak}.

For the unstable growth rate corresponding to any fixed frequency $k$, by Young's inequality,
one has
\begin{align*}
\Lambda_k^2&=\sup_{\langle u^k_2,u^k_2\rangle=1,u_2\in Y}\langle-\mathbb{B}^{'}\mathbb{L}\mathbb{B}Au^k_2,u^k_2\rangle\\
&=\sup\frac{\int_{\Omega}\left(-r\partial_r(\omega^2)|u^k_r|^2-\epsilon^2|\partial_{z}u^k_r|^2
-\epsilon^2|\partial_{z}u^k_z|^2
-\epsilon^2\left|\partial_{z}B^k_\theta+\frac{2\omega}{-\epsilon}u^k_r\right|^2\right)dx}{\int_\Omega (|u^k_r|^2+|u^k_z|^2+|B^k_\theta|^2) dx}\\
&=\sup\frac{\int_{\Omega}\left(-r\partial_r(\omega^2)|\tilde{u}_r|^2-\epsilon^2k^2|\tilde{u}_r|^2
-\epsilon^2k^2|\tilde{u}_z|^2
-\epsilon^2\left|-k\tilde{B}_\theta+\frac{2\omega}{-\epsilon}\tilde{u}_r\right|^2\right)dx}{\int_\Omega (|\tilde{u}_r|^2+|\tilde{u}_z|^2+|\tilde{B}_\theta|^2) dx}\\
&\leq\sup\frac{\int_\Omega\left(-\frac{\partial_r(\omega^2r^4)}{r^3}|\tilde{u}_r|^2
-\epsilon^2k^2|\tilde{B}_\theta|^2+C\epsilon k|\tilde{B}_\theta|^2+C\epsilon k|\tilde{u}_r|^2\right)dx}{\int_\Omega (|\tilde{u}_r|^2+|\tilde{u}_z|^2+|\tilde{B}_\theta|^2) dx}\\
&\leq \sup\frac{\int_{\Omega}-\frac{\partial_r(\omega^2r^4)}{r^3}|\tilde{u}_r|^2  dx  }
{\int_\Omega (|\tilde{u}_r|^2+|\tilde{u}_z|^2+|\tilde{B}_\theta|^2) dx} +C(\varepsilon) \epsilon^2+ \varepsilon k^2 \\
&\leq \lambda_k^2+\varepsilon k^2 + O(\epsilon^2)\leq a_1 +\varepsilon k^2 + O(\epsilon^2),
\end{align*}
where $C>0$ is a constant independent of $\epsilon$,
$\lambda_k$ is defined by \eqref{smalllambdak}, $\varepsilon>0$ is an arbitrary small constant,
 $u_2^k=(u^k_r,u^k_z,B^k_\theta)$ with the following form
\begin{align}
&u_r^k(r,z)=\tilde{u}_r(r)e^{ikz},\nonumber\\
&u_z^k(r,z)=i\tilde{u}_z(r)e^{ikz},\nonumber\\
&B_\theta^k(r,z)=i\tilde{B}_\theta(r)e^{ikz}.\nonumber
\end{align}
This yields that the unstable growth rate of MHD equations for any fixed frequency $k$ tends to $a_1$, as $\epsilon$ tends to $0$.

\end{proof}

\begin{remark}
Comparing Theorem \ref{nonEuler}, Theorem \ref{nonunEuler} and Corollary \ref{BHcorolly}, we find
that the stability criterion of the Euler system is a singular limit of the MHD system in the sense that
if $\epsilon\rightarrow0$, then the MRI criterion is $\partial_r(\omega^2)>0$ instead of the Rayleigh stability criterion $\Upsilon>0$.
Proposition \ref{pp5.1} shows that the unstable growth rate of the MRI is not a singular limit in the sense that
1) if the MHD system is magneto-rotational unstable and Rayleigh stable, then the unstable growth rate of the MHD system is small;
2) if the MHD system is magneto-rotational unstable and Rayleigh unstable, then the unstable growth rate of the MHD system is close
to the unstable growth rate of the Euler system.
\end{remark}

\section{Appendix}\label{appendix}
In this section, we shall prove Lemma \ref{Fderivate} and complete the proof of Theorem \ref{nonlinearstable} when $R_2=\infty$. 

\subsection{Proof of Lemma \ref{Fderivate}}
In this subsection, we shall prove $F_1(\psi)\in C^2(H^1_{mag})$ and $F_2(v_\theta,\psi)\in C^2(L^2\times H^1_{mag})$.
Define
$$\Omega=\{(r,\theta,z)| 0\leq r\leq1, (\theta,z)\in\mathbb{T}_{2\pi}\times\mathbb{T}_{2\pi}\}.$$
In $\Omega$, by Hardy's inequality and $\frac{1}{r}\geq1,$ for any $\psi\in H^1_{mag}$, it follows
\begin{align*}
\left\|\frac{\psi}{r}\right\|_{L^2(\Omega)}&\leq\left\|\frac{\psi}{r^2}\right\|_{L^2(\Omega)}
\lesssim\left(\int_{\Omega} \left(\frac{1}{r^2}|\partial_z \psi|^2+\frac{1}{r^2}|\partial_r \psi|^2\right) dx\right)^{\frac{1}{2}}\lesssim\|\psi\|_{H^1_{mag}}
\end{align*}
and
\begin{align*}
\left\|\nabla\left(\frac{\psi}{r}\right)\right\|_{L^2(\Omega)}&=\left(\int_{\Omega} \left(\left|\partial_z \left(\frac{\psi}{r}\right)\right|^2+\left|\partial_r \left(\frac{\psi}{r}\right)\right|^2\right) dx\right)^{\frac{1}{2}}\\&\leq\left(\int_{\Omega} \left(\frac{1}{r^2}|\partial_z \psi|^2+\frac{1}{r^2}|\partial_r \psi|^2+\left|\frac{\psi}{r^2}\right|^2\right)  dx\right)^{\frac{1}{2}}\lesssim\|\psi\|_{H^1_{mag}}.
\end{align*}
Hence, we have
\begin{align}\label{bounded}
\left\|\frac{\psi}{r}\right\|_{H^1(\Omega)}\lesssim\|\psi\|_{H^1_{mag}}.
\end{align}
Notice, by $\omega(r)\in C^{3}(0,1)$, $\partial_r\omega(r)=O(r^{\beta-3}) (r\rightarrow0)$ with $\beta\geq4$,  
we have
\begin{align}
& |r\Phi(\psi_0)|=O(r^{\beta-1})+O(r)\quad (r\rightarrow0),
\label{order1bound}\\
& |r^2\Phi^{'}(\psi_0)|=O(r^{\beta-2})\quad (r\rightarrow0),
\label{order2bound}\\
& |r^3\Phi^{''}(\psi_0)|=O(r^{\beta-3})\quad (r\rightarrow0),
\label{order3bound}\\
& |r^4\Phi^{'''}(\psi_0)|=O(r^{\beta-4})\quad (r\rightarrow0),
\label{order4bound}\\
& |rf^{'}(\psi_0)|=O(r^{\beta-1})\quad (r\rightarrow0),
\label{order5bound}\\
& |r^2f^{''}(\psi_0)|=O(r^{\beta-2})\quad (r\rightarrow0),
\label{order6bound}\\
& |r^3f^{'''}(\psi_0)|=O(r^{\beta-3})\quad (r\rightarrow0),
\label{order7bound}
\end{align}
with $\beta\geq4.$

\begin{proof}[Proof of Lemma \ref{Fderivate}]
1) For any $\psi, \widetilde{\psi}\in H^1_{mag}$, we have
\begin{align*}
|\partial_\lambda F_1(\psi+\lambda\widetilde{\psi})|_{\lambda=0}|&=\left|\int_\Omega f'(\psi)\widetilde{\psi} dx\right|
\leq \|rf'(\psi)\|_{L^2}\left\|\frac{\widetilde{\psi}}{r}\right\|_{L^2}\lesssim \|\widetilde{\psi}\|_{H^1_{mag}},
\end{align*}
by \eqref{bounded} and \eqref{order5bound}. This implies that $F_1$ is G\^{a}teaux differentiable at $\psi\in H^1_{mag}$. To show that
$F_1(\psi) \in C^1(H^1_{mag})$, we choose $\{\psi_n\}\in H^1_{mag}$ such that $\psi_n\rightarrow \psi$ in $H^1_{mag}$. 
Then we need to claim that
\begin{align*}
\partial_\lambda F_1(\psi_n+\lambda\widetilde{\psi})|_{\lambda=0}\rightarrow \partial_\lambda F_1(\psi+\lambda\widetilde{\psi})|_{\lambda=0} \text{ as } n\rightarrow\infty.
\end{align*}
As a matter of fact, by \eqref{bounded} and \eqref{order6bound}, it follows that
\begin{align*}
&|\partial_\lambda F_1(\psi_n+\lambda\widetilde{\psi})|_{\lambda=0}-\partial_\lambda F_1(\psi+\lambda\widetilde{\psi})|_{\lambda=0}|\nonumber\\
&=\left|\int_\Omega (f'(\psi_n)\widetilde{\psi} - f'(\psi)\widetilde{\psi}) dx\right|=\|f''(\psi+\kappa(\psi_n-\psi))(\psi_n -\psi )\|_{L^2}\|\widetilde{\psi}\|_{L^2}\nonumber\\
&\lesssim \|r^2f''(\psi+\kappa(\psi_n-\psi))\|_{L^{3/2}}\left\|\frac{\psi_n -\psi}{r} \right\|_{L^6}\left\|\frac{\widetilde{\psi}}{r}\right\|_{L^6}\nonumber\\
&\lesssim \|\psi_n -\psi \|_{H^1_{mag}}\|\widetilde{\psi}\|_{H^1_{mag}}\rightarrow0 \text{ as } n\rightarrow\infty,
\end{align*}
where $\kappa\in(0,1)$.
Thus $F_1(\psi) \in C^1(H^1_{mag})$.

Next, we show that the  2-th order G$\hat{\text{a}}$teaux derivative of $F_1$ exists at $\psi \in H^1_{mag}$.
For any $\widetilde{\psi},\widehat{\psi} \in H^1_{mag}$, we have
\begin{align*}
|\partial_{\tau}\partial_{\lambda}F_1(\psi+\lambda\widetilde{\psi}+\tau \widehat{\psi} )|_{\lambda=\tau=0}|=&\left|\int_{\Omega} f''(\psi)\widetilde{\psi} \widehat{\psi} dx\right|
\lesssim \|\widetilde{\psi}\|_{H^1_{mag}}\|\widehat{\psi}\|_{H^1_{mag}}.
\end{align*}
Thus,  $F_1$ is 2-th G$\hat{\text{a}}$teaux differentiable at $\psi\in H^1_{mag}$.
We will prove that
 \begin{align*}
\partial_{\tau}\partial_{\lambda}F_1(\psi_n+\lambda\widetilde{\psi}+\tau \widehat{\psi} )|_{\lambda=\tau=0}\to\partial_{\tau}\partial_{\lambda}F_1(\psi+\lambda\widetilde{\psi}+\tau \widehat{\psi} )|_{\lambda=\tau=0}\text{ as } n\rightarrow\infty
\end{align*}
for  $\{\psi_n\}\subset H^1_{mag}$  such that $\psi_n\to\psi (n\rightarrow\infty)$ in $ H^1_{mag}$.
By H\"{o}lder's inequality, \eqref{bounded} and \eqref{order7bound}, we have
\begin{align*}
&|\partial_{\tau}\partial_{\lambda}F_1(\psi_n+\lambda\widetilde{\psi}+\tau \widehat{\psi} )|_{\lambda=\tau=0}-\partial_{\tau}\partial_{\lambda}F_1(\psi+\lambda\widetilde{\psi}+\tau \widehat{\psi} )|_{\lambda=\tau=0}|\\
&= \left|\int_{\Omega} [f''(\psi_n)-f''(\psi)]\widetilde{\psi} \widehat{\psi} dx\right|\\
&\lesssim\|r^3f'''(\psi+\kappa(\psi_n-\psi))\|_{L^2} \left\|\frac{\psi_n-\psi}{r}\right\|_{L^6}\left\|\frac{\widetilde{\psi}}{r}\right\|_{L^6} \left\|\frac{\widehat{\psi}}{r}\right\|_{L^6}\\
&\lesssim  \|\psi_n-\psi\|_{H^1_{mag}}\|\widetilde{\psi}\|_{H^1_{mag}} \|\widehat{\psi}\|_{H^1_{mag}}\rightarrow 0 \text{ as } n\rightarrow\infty,
\end{align*}
with $0<\kappa<1.$ Hence, $F_1(\psi) \in C^2(H^1_{mag})$.

2) For any $v_\theta,\widetilde{v}_\theta\in L^2(\Omega),\psi,\widetilde{\psi}\in H^1_{mag}$, one has
\begin{align*}
|\partial_{\lambda}F_2(v_\theta+\lambda\widetilde{v}_\theta,\psi)|_{\lambda=0}|&=\left|\partial_{\lambda}\int_\Omega r(v_\theta+\lambda\widetilde{v}_\theta) \Phi(\psi)dx|_{\lambda=0}\right|\\
&=\left|\int_\Omega r\widetilde{v}_\theta \Phi(\psi)dx\right|\\
&\lesssim\|\widetilde{v}_\theta\|_{L^2}\|r\Phi(\psi)\|_{L^2}
\end{align*}
and
\begin{align*}
|\partial_{\lambda}F_2(v_\theta,\psi+\lambda\widetilde{\psi})|_{\lambda=0}|&=\left|\partial_{\lambda}\int_\Omega rv_\theta \Phi(\psi+\lambda\widetilde{\psi})dx|_{\lambda=0}\right|\\
&=\left|\int_\Omega rv_\theta \Phi^{'}(\psi)\widetilde{\psi}dx\right|\\
&\lesssim\|v_\theta\|_{L^2}\|r^2\Phi^{'}(\psi)\|_{L^3}\left\|\frac{\widetilde{\psi}}{r}\right\|_{L^6}\\
&\lesssim\|v_\theta\|_{L^2}\|\widetilde{\psi}\|_{H^1_{mag}},
\end{align*}
by \eqref{bounded} and \eqref{order1bound}-\eqref{order2bound}.
This implies that $F_2$ is G\^{a}teaux differentiable at $(v_\theta,\psi)\in L^2\times H^1_{mag}$. To show that
$F_2(v_\theta,\psi) \in C^1(L^2\times H^1_{mag})$, we choose $\{(v_\theta)_n\}\in L^2, \{\psi_n\}\in H^1_{mag}$ such that $(v_\theta)_n\rightarrow v_\theta$
in $L^2$ and $\psi_n\rightarrow\psi$ in $H^1_{mag}$. 
It remains to prove that
\begin{align*}
\partial_{\lambda}F_2((v_\theta)_n+\lambda\widetilde{v}_\theta,\psi)|_{\lambda=0}\rightarrow \partial_{\lambda}F_2(v_\theta+\lambda\widetilde{v}_\theta,\psi)|_{\lambda=0} \text{ as } n\rightarrow\infty,
\end{align*}
\begin{align*}
\partial_{\lambda}F_2(v_\theta,\psi_n+\lambda\widetilde{\psi})|_{\lambda=0}\rightarrow \partial_{\lambda}F_2(v_\theta,\psi+\lambda\widetilde{\psi})|_{\lambda=0} \text{ as } n\rightarrow\infty.
\end{align*}
In fact, using \eqref{bounded} and \eqref{order3bound}, we have
\begin{align*}
&|\partial_{\lambda}F_2((v_\theta)_n+\lambda\widetilde{v}_\theta,\psi)|_{\lambda=0}- \partial_{\lambda}F_2(v_\theta+\lambda\widetilde{v}_\theta,\psi)|_{\lambda=0}|
=\left|\int_\Omega (r\widetilde{v}_\theta \Phi(\psi)-r\widetilde{v}_\theta \Phi(\psi))dx\right|=0
\end{align*}
and
\begin{align*}
&|\partial_{\lambda}F_2(v_\theta,\psi+\lambda\widetilde{\psi})|_{\lambda=0}- \partial_{\lambda}F_2(v_\theta,\psi_n+\lambda\widetilde{\psi})|_{\lambda=0}|\nonumber\\
&= \left|\int_\Omega (rv_\theta \Phi^{'}(\psi_n)-rv_\theta \Phi^{'}(\psi))\widetilde{\psi}dx\right|
\nonumber\\
&\lesssim\|v_\theta\|_{L^2}\|r^3\Phi^{''}(\psi+\kappa(\psi_n-\psi))\|_{L^6}\left\|\frac{\psi_n-\psi}{r}\right\|_{L^6}\left
\|\frac{\widetilde{\psi}}{r}\right\|_{L^6}
\nonumber\\
&\lesssim\|v_\theta\|_{L^2}\|\psi_n-\psi\|_{H^1_{mag}}\|\widetilde{\psi}\|_{H^1_{mag}}\rightarrow 0 \text{ as } n\rightarrow\infty,
\end{align*}
with $0<\kappa<1.$
Therefore, $F_2(v_\theta,\psi) \in C^1(L^2\times H^1_{mag})$.

Next, we show that the  2-th order G$\hat{\text{a}}$teaux derivative of $F_2$ exists at $(v_\theta,\psi) \in L^2\times H^1_{mag}$.
For $v_\theta,\widetilde{v}_\theta,\widehat{v}_\theta\in L^2,\psi,\widetilde{\psi},\widehat{\psi}\in H^1_{mag}$, we have
\begin{align*}
|\partial_\tau\partial_{\lambda}F_2(v_\theta+\lambda\widetilde{v}_\theta+\tau\widehat{v}_\theta,\psi)|_{\lambda=\tau=0}|
&=\left|\partial_\tau\partial_{\lambda}\int_\Omega r(v_\theta+\lambda\widetilde{v}_\theta+\tau\widehat{v}_\theta) \Phi(\psi)dx|_{\lambda=\tau=0}\right|\\
&=\left|\partial_\tau\int_\Omega r\widetilde{v}_\theta \Phi(\psi)dx||_{\lambda=\tau=0}\right|=0,
\end{align*}
\begin{align*}
|\partial_\tau\partial_{\lambda}F_2(v_\theta,\psi+\lambda\widetilde{\psi}+\tau\widehat{\psi})|_{\lambda=\tau=0}|
&=\left|\partial_\tau\partial_{\lambda}\int_\Omega rv_\theta \Phi(\psi+\lambda\widetilde{\psi}+\tau\widehat{\psi})dx|_{\lambda=\tau=0}\right|\\
&=\left|\partial_\tau\int_\Omega rv_\theta \Phi^{'}(\psi+\lambda\widetilde{\psi}+\tau\widehat{\psi})\widetilde{\psi}dx|_{\lambda=\tau=0}\right|\\
&=\left|\int_\Omega rv_\theta \Phi^{''}(\psi)\widetilde{\psi}\widehat{\psi}dx\right|
\\
&\lesssim\|v_\theta\|_{L^2}\|r^3\Phi^{''}(\psi)\|_{L^6}\left\|\frac{\widetilde{\psi}}{r}\right\|_{L^6}\left\|\frac{\widehat{\psi}}{r}\right\|_{L^6}
\\
&\lesssim\|v_\theta\|_{L^2}\|\widetilde{\psi}\|_{H^1_{mag}}\|\widehat{\psi}\|_{H^1_{mag}}
\end{align*}
and
\begin{align*}
|\partial_\tau\partial_{\lambda}F_2(v_\theta+\lambda\widetilde{v}_\theta,\psi+\tau\widetilde{\psi})|_{\lambda=\tau=0}|
&=\left|\partial_\tau\partial_{\lambda}\int_\Omega r(v_\theta+\lambda\widetilde{v}_\theta) \Phi(\psi+\tau\widetilde{\psi})dx|_{\lambda=\tau=0}\right|\\
&=\left|\partial_\tau\int_\Omega r\widetilde{v}_\theta \Phi(\psi+\tau\widetilde{\psi})dx|_{\lambda=\tau=0}\right|\\
&=\left|\int_\Omega r\widetilde{v}_\theta \Phi^{'}(\psi)\widetilde{\psi}dx\right|
\\
&\lesssim\|\widetilde{v}_\theta\|_{L^2}\|r^2\Phi^{'}(\psi)\|_{L^3}\left\|\frac{\widetilde{\psi}}{r}\right\|_{L^6}
\\
&\lesssim\|\widetilde{v}_\theta\|_{L^2}\|\widetilde{\psi}\|_{H^1_{mag}},
\end{align*}
by \eqref{bounded} and \eqref{order2bound}-\eqref{order3bound}.
Thus, $F_2$ is 2-th G$\hat{\text{a}}$teaux differentiable at $(v_\theta,\psi) \in L^2\times H^1_{mag}$.
We will prove that
 \begin{align*}
\partial_{\tau}\partial_{\lambda}F_2((v_\theta)_n+\lambda\widetilde{v}_\theta+\tau\widehat{v}_\theta,\psi)|_{\lambda=\tau=0}
\to  \partial_{\tau}\partial_{\lambda}F_2(v_\theta+\lambda\widetilde{v}_\theta+\tau\widehat{v}_\theta,\psi)|_{\lambda=\tau=0},
\end{align*}
 \begin{align*}
\partial_{\tau}\partial_{\lambda}F_2(v_\theta,\psi_n+\lambda\widetilde{\psi}+\tau\widehat{\psi})|_{\lambda=\tau=0}
\to  \partial_{\tau}\partial_{\lambda}F_2(v_\theta,\psi+\lambda\widetilde{\psi}+\tau\widehat{\psi})|_{\lambda=\tau=0},
\end{align*}
 \begin{align*}
\partial_{\tau}\partial_{\lambda}F_2((v_\theta)_n+\lambda\widetilde{v}_\theta,\psi_n+\tau\widetilde{\psi})|_{\lambda=\tau=0}
\to  \partial_{\tau}\partial_{\lambda}F_2(v_\theta+\lambda\widetilde{v}_\theta,\psi+\tau\widetilde{\psi})|_{\lambda=\tau=0},
\end{align*}
for $(v_\theta)_n\rightarrow v_\theta$
in $L^2$ and $\psi_n\rightarrow\psi$ in $H^1_{mag}$. 
In specify, by \eqref{bounded} and \eqref{order3bound}-\eqref{order4bound},
 \begin{align*}
|\partial_{\tau}\partial_{\lambda}F_2((v_\theta)_n+\lambda\widetilde{v}_\theta+\tau\widehat{v}_\theta,\psi)|_{\lambda=\tau=0}
-\partial_{\tau}\partial_{\lambda}F_2(v_\theta+\lambda\widetilde{v}_\theta+\tau\widehat{v}_\theta,\psi)|_{\lambda=\tau=0}|=0,
\end{align*}
 \begin{align*}
&|\partial_{\tau}\partial_{\lambda}F_2(v_\theta,\psi_n+\lambda\widetilde{\psi}+\tau\widehat{\psi})|_{\lambda=\tau=0}
-\partial_{\tau}\partial_{\lambda}F_2(v_\theta,\psi+\lambda\widetilde{\psi}+\tau\widehat{\psi})|_{\lambda=\tau=0}|
\nonumber\\&=\left|\int_\Omega (rv_\theta \Phi^{''}(\psi_n)\widetilde{\psi}\widehat{\psi}-rv_\theta \Phi^{''}(\psi)\widetilde{\psi}\widehat{\psi})dx\right|
\nonumber\\
&\lesssim\|v_\theta\|_{L^2}\|r^4\Phi^{'''}(\psi+\kappa(\psi_n-\psi))\|_{L^\infty}
\left\|\frac{\psi_n-\psi}{r}\right\|_{L^6}\left\|\frac{\widetilde{\psi}}{r}\right\|_{L^6}\left\|\frac{\widehat{\psi}}{r}\right\|_{L^6}\nonumber\\
&\lesssim\|v_\theta\|_{L^2}\|\psi_n-\psi\|_{H^1_{mag}}\|\widetilde{\psi}\|_{H^1_{mag}}\|\widehat{\psi}\|_{H^1_{mag}}\rightarrow 0\text{ as } n\rightarrow\infty
\end{align*}
and
 \begin{align*}
&|\partial_{\tau}\partial_{\lambda}F_2((v_\theta)_n+\lambda\widetilde{v}_\theta,\psi_n+\tau\widetilde{\psi})|_{\lambda=\tau=0}
-\partial_{\tau}\partial_{\lambda}F_2(v_\theta+\lambda\widetilde{v}_\theta,\psi+\tau\widetilde{\psi})|_{\lambda=\tau=0}|\nonumber\\
&=\left|\int_\Omega (r\widetilde{v}_\theta \Phi^{'}(\psi_n)\widetilde{\psi}-r\widetilde{v}_\theta \Phi^{'}(\psi)\widetilde{\psi})dx\right|
\nonumber\\
&\lesssim\|\widetilde{v}_\theta\|_{L^2}\|r^3\Phi^{''}(\psi+\kappa(\psi_n-\psi))\|_{L^6}
\left\|\frac{\psi_n-\psi}{r}\right\|_{L^6}\left\|\frac{\widetilde{\psi}}{r}\right\|_{L^6}\nonumber\\
&\lesssim\|\widetilde{v}_\theta\|_{L^2}\|\psi_n-\psi\|_{H^1_{mag}}\|\widetilde{\psi}\|_{H^1_{mag}}\rightarrow 0\text{ as } n\rightarrow\infty,
\end{align*}
with $0<\kappa<1.$ Consequently, we have proved that
$F_2(v_\theta,\psi)\in C^2(L^2\times H^1_{mag})$.

\end{proof}

\subsection{Nonlinear stability for MHD equations in unbounded domain}
In this section, we will consider the nonlinear stability for MHD equations in unbounded domain $\Omega=\{x\in\mathbb{R}^3|r=\sqrt{x_1^2+x_2^2}< \infty,  x_3=z\in\mathbb{T}_{2\pi}\}$. For simplify, we take $\epsilon=1, b(r)=b$, where $b$ is a constant.

Define
$$\Omega_1=\{(r,\theta,z)| 0\leq r\leq1, (\theta,z)\in\mathbb{T}_{2\pi}\times\mathbb{T}_{2\pi}\},$$
$$\Omega_2=\{(r,\theta,z)| 1<r<\infty, (\theta,z)\in\mathbb{T}_{2\pi}\times\mathbb{T}_{2\pi}\}.$$
In $\Omega_1$, by \eqref{bounded}, it follows
\begin{align}\label{unbounded}
 \left\|\frac{\psi}{r}\right\|_{H^1(\Omega_1)}\lesssim\|\psi\|_{H^1_{mag}}.
 \end{align}
In $\Omega_2$, by Hardy's inequality and $\frac{1}{r}\leq1,$ it follows
\begin{align*}
\left\|\frac{\psi}{r^2}\right\|_{L^2(\Omega_2)}&\lesssim\left(\int_{\Omega} \frac{1}{r^2}|\partial_r \psi|^2 dx\right)^{\frac{1}{2}}\\
&\lesssim\left(\int_{\Omega} \left(\frac{1}{r^2}|\partial_z \psi|^2+\frac{1}{r^2}|\partial_r \psi|^2\right) dx\right)^{\frac{1}{2}}\lesssim\|\psi\|_{H^1_{mag}}
\end{align*}
 and
 \begin{align*}
\left\|\nabla\left(\frac{\psi}{r^2}\right)\right\|_{L^2(\Omega_2)}&=\left(\int_{\Omega_2}\left(\left|\partial_r\left(\frac{\psi}{r^2}\right)\right|^2
+\left|\partial_z\left(\frac{\psi}{r^2}\right)\right|^2\right)dx\right)^{\frac{1}{2}}
\\&\lesssim\left(\int_{\Omega_2}\left(\frac{1}{r^4}|\partial_z \psi|^2+\frac{1}{r^4}|\partial_r \psi|^2+\left|\frac{\psi}{r^3}\right|^2\right)dx\right)^{\frac{1}{2}}
\\&\lesssim\left(\int_{\Omega_2}\left(\frac{1}{r^2}|\partial_z \psi|^2+\frac{1}{r^2}|\partial_r \psi|^2+\left|\frac{\psi}{r^2}\right|^2\right)dx\right)^{\frac{1}{2}}
\lesssim\|\psi\|_{H^1_{mag}}.
\end{align*}
Thus, we get
\begin{align}\label{magunbounded}
\left\|\frac{\psi}{r^2}\right\|_{H^1(\Omega_2)}\lesssim\|\psi\|_{H^1_{mag}}.
 \end{align}
Define the relative energy $E(\delta v,\delta H_\theta,\delta \psi)$ and the relative Casimir-energy functional $E_{c}(\delta v,\delta H_\theta,\delta\psi)$:
\begin{align}\label{'7.1}
E(\delta v,\delta H_\theta,\delta\psi)\notag&=\frac{1}{2}\int_{\Omega}(|v_r|^2+|v_z|^2+|H_\theta|^2)dx\nonumber\\
&\quad+\frac{1}{2}\int_{\Omega}\left(|v_\theta|^2+\frac{1}{r^2}|\partial_z\psi|^2+\frac{1}{r^2}|\partial_r\psi|^2
 -|r\omega(r)|^2-|b|^2\right) dx
\end{align}
and
\begin{align*}
&E_{c}(\delta v,\delta H_\theta,\delta\psi)=E(\delta v,\delta H_\theta,\delta\psi)+\int_{\Omega}(rv_\theta\Phi(\psi)-r^2\omega(r)\Phi(\psi_0))dx+\int_{\Omega}(f(\psi)-f(\psi_0)) dx,
\end{align*}
with $(\delta v,\delta H_\theta,\delta\psi)=(v,H_{\theta},\psi)-(v_{0},0,\psi_{0})$,
where
\begin{align*}
\Phi(s)=\begin{cases}
-\omega\left(\sqrt{\frac{-2s}{b}}\right),  \text{ for any }s\in (-\infty,0],\\
\text{some $C_0^\infty$ extension function, } \text{ for any } s\notin (-\infty,0]
\end{cases}
\end{align*}
and
\begin{align*}
f^{'}(s)=\begin{cases}
\frac{-s}{b}\partial_s(\omega^2(\sqrt{\frac{-2s}{b}})),  \text{ for any }s\in (-\infty,0],\\
\text{some $C_0^\infty$ extension function, } \text{ for any } s\notin (-\infty,0].
\end{cases}
\end{align*}
Hence, by $\omega(r)\in C^{3}(0,\infty)$, $\omega(r)=O(r^{-(1+\alpha)}) (r\rightarrow\infty)$ with $\alpha>1$ and $\partial_r\omega(r)=O(r^{\beta-3}) (r\rightarrow0)$ with $\beta\geq4$,  
we have
\begin{align}
&|r\Phi(\psi_0)|=O(r^{-\alpha}) (r\rightarrow\infty),\quad |r\Phi(\psi_0)|=O(r^{\beta-1})+O(r) (r\rightarrow0);
\label{order1}\\
&|r^3\Phi^{'}(\psi_0)|=O(r^{-\alpha}) (r\rightarrow\infty),\quad |r^2\Phi^{'}(\psi_0)|=O(r^{\beta-2}) (r\rightarrow0);
\label{order2}\\
&|r^5\Phi^{''}(\psi_0)|=O(r^{-\alpha}) (r\rightarrow\infty),\quad |r^3\Phi^{''}(\psi_0)|=O(r^{\beta-3}) (r\rightarrow0);
\label{order3}\\
&|r^7\Phi^{'''}(\psi_0)|=O(r^{-\alpha}) (r\rightarrow\infty),\quad |r^4\Phi^{'''}(\psi_0)|=O(r^{\beta-4}) (r\rightarrow0);
\label{order4}\\
&|r^2f^{'}(\psi_0)|=O(r^{-2\alpha}) (r\rightarrow\infty),\quad |rf^{'}(\psi_0)|=O(r^{\beta-1}) (r\rightarrow0);
\label{order5}\\
&|r^4f^{''}(\psi_0)|=O(r^{-2\alpha}) (r\rightarrow\infty),\quad |r^2f^{''}(\psi_0)|=O(r^{\beta-2}) (r\rightarrow0);
\label{order6}\\
&|r^6f^{'''}(\psi_0)|=O(r^{-2\alpha}) (r\rightarrow\infty),\quad |r^3f^{'''}(\psi_0)|=O(r^{\beta-3}) (r\rightarrow0).
\label{order7}
\end{align}
\begin{lemma} 
\label{fpsiconunbounded}
For any $\Phi(\tau),f(\tau)\in C^1(\mathbb{R})$, 
the functionals $\int_{\Omega}(f(\psi)-f(\psi_0)) dx$
and $\int_{\Omega}(rv_{\theta} \Phi(\psi)-r^2\omega(r)\Phi(\psi_0)) dx$ are conserved for strong solutions of MHD equations \eqref{nonlinearMHDrz}.
\end{lemma}
\begin{proof}
The proof of the lemma is similar to Lemma \ref{fpsicon}.
\end{proof}
\begin{lemma}\label{ubdFderivate}
1) 
For any $f(\tau)\in C^3(\mathbb{R})$ and satisfies \eqref{order5}-\eqref{order7}, then
$F_1(\delta\psi):=\int_\Omega (f(\delta\psi+\psi_0)-f(\psi_0))dx $ is second order Fr\'{e}chet differentiable at $\delta\psi=0$, i.e.,
\begin{align*}
\notag&F_1(\delta\psi)-F_1(0)-F_1^{'}(0)\delta\psi
=\int_{\Omega}(f(\delta\psi+\psi_0)-f(\psi_0)-f^{'}(\psi_0)\delta\psi) dx=o(\|\delta\psi\|_{H^1_{mag}})
\end{align*}
\\
and
\begin{align*}
\notag&F_1(\delta\psi)-F_1(0)-F_1'(0)\delta\psi-\frac{1}{2}\langle F_1''(0)\delta\psi,\delta\psi\rangle\\
&=\int_{\Omega}\left(f(\delta\psi+\psi_0)-f(\psi_0)-f^{'}(\psi_0)\delta\psi-\frac{1}{2}f^{''}(\psi_0)|\delta\psi|^2\right) dx\\
\notag&=o(\|\delta\psi\|^2_{H^1_{mag}}).
\end{align*}
2) For any $\Phi(\tau)\in C^3(\mathbb{R})$ and satisfies \eqref{order1}-\eqref{order4}, then
$F_2(\delta v_\theta,\delta\psi):=\int_\Omega (r(\mathbf{v}_0+\delta v_\theta) \Phi(\psi_0+\delta\psi)-r\mathbf{v}_0\Phi(\psi_0))dx$ is second order Fr\'{e}chet differentiable at $(\delta v_\theta,\delta\psi)=(0,0)$, i.e.,
\begin{align*}
\notag&F_2(\delta v_\theta,\delta\psi)-F_2(0,0)-\partial_{\delta\psi}F_2(0,0)\delta\psi-\partial_{\delta v_\theta}F_2(0,0)\delta v_\theta\\
\notag&=\int_{\Omega}r\delta v_\theta(\Phi(\delta\psi+\psi_0)-\Phi(\psi_0))dx
+\int_{\Omega}r\mathbf{v}_0(\Phi(\delta\psi+\psi_0)-\Phi(\psi_0)-\Phi^{'}(\psi_0)\delta\psi)dx\\
\notag&=o(\|\delta\psi\|_{H^1_{mag}})+o(\|\delta v_\theta\|_{L^2})
\end{align*}
and
\begin{align*}
\notag&F_2(\delta v_\theta,\delta\psi)-F_2(0,0)-\partial_{\delta\psi}F_2(0,0)\delta\psi-\partial_{\delta v_\theta}F_2(0,0)\delta v_\theta\\
\notag&-\langle\partial_{\delta\psi \delta v_\theta}F_2(0,0)\delta\psi,\delta v_\theta\rangle-\frac{1}{2}\langle\partial_{\delta\psi\delta\psi}F_2(0,0)\delta\psi,\delta\psi\rangle\\
\notag&=\int_{\Omega}r\delta v_\theta(\Phi(\delta\psi+\psi_0)-\Phi(\psi_0)-\Phi^{'}(\psi_0)\delta\psi)dx
\\&\quad\quad+\int_{\Omega}r\mathbf{v}_0\left(\Phi(\delta\psi+\psi_0)-\Phi(\psi_0)-\Phi^{'}(\psi_0)\delta\psi-\frac{1}{2}\Phi^{''}(\psi_0)|\delta\psi|^2\right)dx\\
\notag&=o(\|\delta\psi\|^2_{H^1_{mag}})+o(\|\delta v_\theta\|^2_{L^2})
.
\end{align*}
\end{lemma}
\begin{proof}
1) For any $\tilde{\psi}\in Z$, we have
\begin{align*}
|\partial_\lambda F_1(0+\lambda\tilde{\psi})|_{\lambda=0}|&=\left|\int_\Omega f'(\psi_0)\tilde{\psi} dx\right|\leq\left|\int_{\Omega_1} f'(\psi_0)\tilde{\psi} dx\right|+\left|\int_{\Omega_2} f'(\psi_0)\tilde{\psi} dx\right|\nonumber\\
&\leq \|rf'(\psi_0)\|_{L^2(\Omega_1)}\left\|\frac{\tilde{\psi}}{r}\right\|_{L^2(\Omega_1)}+\|f'(\psi_0)r^2\|_{L^2(\Omega_2)}\left\|\frac{\tilde{\psi}}{r^2}\right\|_{L^2(\Omega_2)}\nonumber\\
&\lesssim \|\tilde{\psi}\|_{H^1_{mag}},
\end{align*}
with the facts \eqref{unbounded}-\eqref{magunbounded} and \eqref{order5}. This implies that $F_1$ is G\^{a}teaux differentiable at $0$. To show that
$F_1(\delta\psi)$ is Fr\'{e}chet differentiable at $0$, we choose $\{\delta\psi_n\}\in Z$ such that $\delta\psi_n\rightarrow 0$ in $Z$,
and prove that
\begin{align*}
\partial_\lambda F_1(\delta\psi_n+\lambda\tilde{\psi})|_{\lambda=0}\rightarrow \partial_\lambda F_1(0+\lambda\tilde{\psi})|_{\lambda=0} \text{ as } n\rightarrow\infty.
\end{align*}
In fact, using \eqref{unbounded}-\eqref{magunbounded} and \eqref{order6}, it follows that
\begin{align*}
&|\partial_\lambda F_1(\delta\psi_n+\lambda\tilde{\psi})|_{\lambda=0}-\partial_\lambda F_1(0+\lambda\tilde{\psi})|_{\lambda=0}|\nonumber\\
&=\left|\int_\Omega (f'(\delta\psi_n+\psi_0)\tilde{\psi} - f'(0+\psi_0)\tilde{\psi}) dx\right|\nonumber\\
&\lesssim\|r^2f''(\psi_0+\kappa\delta\psi_n)\|_{L^{3/2}(\Omega_1)}\left\|\frac{\delta\psi_n}{r}\right\|_{L^6(\Omega_1)}
\left\|\frac{\tilde{\psi}}{r}\right\|_{L^6(\Omega_1)}\nonumber\\
&\quad + \|f''(\psi_0+\kappa\delta\psi_n)r^4\|_{L^{3/2}(\Omega_2)}\left\|\frac{\delta\psi_n}{r^2}\right\|_{L^6(\Omega_2)}
\left\|\frac{\tilde{\psi}}{r^2}\right\|_{L^6(\Omega_2)}\nonumber\\
&\lesssim \|\delta\psi_n\|_{H^1_{mag}}\|\tilde{\psi}\|_{H^1_{mag}}\rightarrow0 \text{ as } n\rightarrow\infty,
\end{align*}
where $\kappa\in(0,1)$.
Thus, $F_1(\delta\psi)$ is Fr\'{e}chet differentiable at $0$.

Next, we show that the  2-th order G$\hat{\text{a}}$teaux derivative of $F_1(\delta\psi)$ exists at $0$.
For $\tilde{\psi} \in Z$ and $\widehat{\psi}\in Z$, by \eqref{unbounded}-\eqref{magunbounded} and \eqref{order6}, we have
\begin{align*}
|\partial_{\tau}\partial_{\lambda}F_1(0+\lambda\psi+\tau \widehat{\psi} )|_{\lambda=\tau=0}|=&\left|\int_{\Omega} f''(\psi_0)\tilde{\psi} \widehat{\psi} dx\right|\\
\lesssim&   \|r^2f''(\psi_0)\|_{L^{3/2}(\Omega_1)}\left\|\frac{\tilde{\psi}}{r}\right\|_{L^6(\Omega_1)}\left\|\frac{\widehat{\psi}}{r}\right\|_{L^6(\Omega_1)}\\
& + \|f''(\psi_0)r^4\|_{L^{3/2}(\Omega_2)}\left\|\frac{\tilde{\psi}}{r^2}\right\|_{L^6(\Omega_2)}\left\|\frac{\widehat{\psi}}{r^2}\right\|_{L^6(\Omega_2)}\\
\lesssim&   \|\tilde{\psi}\|_{H^1_{mag}}\|\widehat{\psi}\|_{H^1_{mag}}.
\end{align*}
Thus,  $F_1(\delta\psi)$ is 2-th G$\hat{\text{a}}$teaux differentiable at $0$.
We will prove that
 \begin{align*}
\partial_{\tau}\partial_{\lambda}F_1(\delta\psi_n+\lambda\tilde{\psi}+\tau \widehat{\psi} )|_{\lambda=\tau=0}\to\partial_{\tau}\partial_{\lambda}F_1(0+\lambda\tilde{\psi}+\tau \widehat{\psi} )|_{\lambda=\tau=0},
\end{align*}
for  $\{\delta\psi_n\}\subset Z$  such that $\delta\psi_n\to0$ in $ Z$, which implies that
$F_1(\delta\psi)$ is second order Fr\'{e}chet differentiable at $0$.
By H\"{o}lder's inequality, \eqref{unbounded}-\eqref{magunbounded} and \eqref{order7}, we have
\begin{align*}
&|\partial_{\tau}\partial_{\lambda}F_1(\delta\psi_n+\lambda\tilde{\psi}+\tau \widehat{\psi} )|_{\lambda=\tau=0}-\partial_{\tau}\partial_{\lambda}F_1(0+\lambda\tilde{\psi}+\tau \widehat{\psi} )|_{\lambda=\tau=0}|\\
=& \left|\int_{\Omega} [f''(\delta\psi_n+\psi_0)-f''(\psi_0)]\tilde{\psi} \widehat{\psi} dx\right|\\
\lesssim & \|r^3f'''(\psi_0+\kappa\delta\psi_n)\|_{L^2(\Omega_1)}\left\|\frac{\delta\psi_n}{r}\right\|_{L^6(\Omega_1)}
\left\|\frac{\tilde{\psi}}{r}\right\|_{L^6(\Omega_1)}  \left\|\frac{\widehat{\psi}}{r}\right\|_{L^6(\Omega_1)} \\
&+ \|f'''(\psi_0+\kappa\delta\psi_n)r^6\|_{L^2(\Omega_2)}\left\|\frac{\delta\psi_n}{r^2}\right\|_{L^6(\Omega_2)}
\left\|\frac{\tilde{\psi}}{r^2}\right\|_{L^6(\Omega_2)} \left\|\frac{\widehat{\psi}}{r^2}\right\|_{L^6(\Omega_2)}\\
\lesssim & \|\delta\psi_n\|_{H^1_{mag}}\|\tilde{\psi}\|_{H^1_{mag}} \|\widehat{\psi}\|_{H^1_{mag}}\rightarrow 0 \text{ as } n\rightarrow\infty,
\end{align*}
with $0<\kappa<1.$ Hence, $F_1(\delta\psi)$ is second-order Fr\'{e}chet differentiable at $0$.

2) For any $\tilde{v}_\theta \in L^2(\Omega),\tilde{\psi}\in Z$, combining \eqref{unbounded}-\eqref{magunbounded} and \eqref{order1}-\eqref{order2}, one has
\begin{align*}
|\partial_{\lambda}F_2(0+\lambda \tilde{v}_\theta,0)|_{\lambda=0}|&=\left|\partial_{\lambda}\int_\Omega r(\mathbf{v}_0+\lambda \tilde{v}_\theta) \Phi(\psi_0)dx|_{\lambda=0}\right|\\
&=\left|\int_\Omega r \tilde{v}_\theta \Phi(\psi_0)dx\right|\\
&\lesssim\|\tilde{v}_\theta\|_{L^2}\|r\Phi(\psi_0)\|_{L^2}\lesssim\|\tilde{v}_\theta\|_{L^2}
\end{align*}
and
\begin{align*}
|\partial_{\lambda}F_2(0,0+\lambda\tilde{\psi})|_{\lambda=0}|&=\left|\partial_{\lambda}\int_\Omega r\mathbf{v}_0 \Phi(\psi_0+\lambda\tilde{\psi})dx|_{\lambda=0}\right|\\
&=\left|\int_\Omega r\mathbf{v}_0 \Phi^{'}(\psi_0)\tilde{\psi} dx\right|\\
&\lesssim\|\mathbf{v}_0\|_{L^2(\Omega_1)}\|r^2\Phi^{'}(\psi_0)\|_{L^3(\Omega_1)}\left\|\frac{\tilde{\psi}}{r}\right\|_{L^6(\Omega_1)}\\
&\quad+ \|\mathbf{v}_0\|_{L^2(\Omega_2)}\|r^3\Phi^{'}(\psi_0)\|_{L^3(\Omega_2)}\left\|\frac{\tilde{\psi}}{r^2}\right\|_{L^6(\Omega_2)}
\\
&\lesssim\|\tilde{\psi}\|_{H^1_{mag}}.
\end{align*}
This implies that $F_2(\delta v_\theta,\delta\psi)$ is G\^{a}teaux differentiable at $(0,0)$. To show that
$F_2(\delta v_\theta,\delta\psi)$ is Fr\'{e}chet differentiable at $(0,0)$, we choose $\{(\delta v_\theta)_n\}\in L^2(\Omega), \{\delta\psi_n\}\in Z$ such that $(\delta v_\theta)_n\rightarrow 0$ in $L^2(\Omega), \delta\psi_n\rightarrow0$ in $Z$,
and prove that
\begin{align*}
\partial_{\lambda}F_2((\delta v_\theta)_n+\lambda \tilde{v}_\theta,0)|_{\lambda=0}\rightarrow \partial_{\lambda}F_2(0+\lambda \tilde{v}_\theta,0)|_{\lambda=0} \text{ as } n\rightarrow\infty,
\end{align*}
\begin{align*}
\partial_{\lambda}F_2(0,\delta\psi_n+\lambda\tilde{\psi})|_{\lambda=0}\rightarrow \partial_{\lambda}F_2(0,0+\lambda\tilde{\psi})|_{\lambda=0} \text{ as } n\rightarrow\infty.
\end{align*}
In fact, through \eqref{unbounded}-\eqref{magunbounded} and \eqref{order3}, we have
\begin{align*}
&|\partial_{\lambda}F_2((\delta v_\theta)_n+\lambda \tilde{v}_\theta,0)|_{\lambda=0}- \partial_{\lambda}F_2(0+\lambda \tilde{v}_\theta,0)|_{\lambda=0}| \nonumber\\
&=\left|\int_\Omega (r\tilde{v}_\theta \Phi(\psi_0)-r\tilde{v}_\theta \Phi(\psi_0))dx\right|=0
\end{align*}
and
\begin{align*}
&|\partial_{\lambda}F_2(0,\delta\psi_n+\lambda\tilde{\psi})|_{\lambda=0}- \partial_{\lambda}F_2(0,0+\lambda\tilde{\psi})|_{\lambda=0}|\nonumber\\
&= \left|\int_\Omega (r\mathbf{v}_0 \Phi^{'}(\delta\psi_n+\psi_0)\tilde{\psi}-r\mathbf{v}_0 \Phi^{'}(\psi_0)\tilde{\psi}) dx\right|
\nonumber\\
&\lesssim\|\mathbf{v}_0\|_{L^2(\Omega_1)}\|r^3\Phi^{''}(\psi_0+\kappa\delta\psi_n)\|_{L^6(\Omega_1)}
\left\|\frac{\delta\psi_n}{r}\right\|_{L^6(\Omega_1)}\left\|\frac{\tilde{\psi}}{r}\right\|_{L^6(\Omega_1)}\nonumber\\
&\quad+ \|\mathbf{v}_0\|_{L^2(\Omega_2)}\|r^5\Phi^{''}(\psi_0+\kappa\delta\psi_n)\|_{L^6(\Omega_2)}
\left\|\frac{\delta\psi_n}{r^2}\right\|_{L^6(\Omega_2)}\left\|\frac{\tilde{\psi}}{r^2}\right\|_{L^6(\Omega_2)}
\nonumber\\
&\lesssim\|\mathbf{v}_0\|_{L^2}\|\delta\psi_n\|_{H^1_{mag}}\|\tilde{\psi}\|_{H^1_{mag}}\rightarrow 0 \text{ as } n\rightarrow\infty
\end{align*}
with $0<\kappa<1.$
Therefore, $F_2(\delta v_\theta,\delta\psi)$ is Fr\'{e}chet differentiable at $(0,0)$.

Next, we show that the  2-th order G$\hat{\text{a}}$teaux derivative of $F_2(\delta v_\theta,\delta\psi)$ exists at $(0,0)$.
For any $\tilde{v}_\theta,\widehat{u}_\theta\in L^2,\tilde{\psi},\widehat{\psi}\in Z$, combining \eqref{unbounded}-\eqref{magunbounded} and \eqref{order2}-\eqref{order3}, we have
\begin{align*}
|\partial_\tau\partial_{\lambda}F_2(0+\lambda \tilde{v}_\theta+\tau\widehat{v}_\theta,0)|_{\lambda=\tau=0}|
&=\left|\partial_\tau\partial_{\lambda}\int_\Omega r(\lambda \tilde{v}_\theta+\tau\widehat{v}_\theta) \Phi(\psi_0)dx|_{\lambda=\tau=0}\right|\\
&=\left|\partial_\tau\int_\Omega r\tilde{v}_\theta \Phi(\psi_0)dx|_{\lambda=\tau=0}\right|=0,
\end{align*}
\begin{align*}
&|\partial_\tau\partial_{\lambda}F_2(0,0+\lambda\tilde{\psi}+\tau\widehat{\psi})|_{\lambda=\tau=0}|=\left|\partial_\tau\partial_{\lambda}\int_\Omega r\mathbf{v}_0 \Phi(\psi_0+\lambda\tilde{\psi}+\tau\widehat{\psi})dx|_{\lambda=\tau=0}\right|\\
&=\left|\partial_\tau\int_\Omega r\mathbf{v}_0 \Phi^{'}(\psi_0+\lambda\tilde{\psi}+\tau\widehat{\psi})\psi dx|_{\lambda=\tau=0}\right|\\
&=\left|\int_\Omega r\mathbf{v}_0 \Phi^{''}(\psi_0)\tilde{\psi}\widehat{\psi}dx\right|
\\
&\lesssim\|\mathbf{v}_0\|_{L^2(\Omega_1)}\|r^3\Phi^{''}(\psi_0)\|_{L^6(\Omega_1)}\left\|\frac{\tilde{\psi}}{r}\right\|_{L^6(\Omega_1)}
\left\|\frac{\widehat{\psi}}{r}\right\|_{L^6(\Omega_1)}\\
&\quad+ \|\mathbf{v}_0\|_{L^2(\Omega_1)}\|r^5\Phi^{''}(\psi_0)\|_{L^6(\Omega_2)}\left\|\frac{\tilde{\psi}}{r^2}\right\|_{L^6(\Omega_2)}\left\|\frac{\widehat{\psi}}{r^2}\right\|_{L^6(\Omega_2)}
\\
&\lesssim\|\mathbf{v}_0\|_{L^2(\Omega_1)}\|\tilde{\psi}\|_{H^1_{mag}}\|\widehat{\psi}\|_{H^1_{mag}}
\end{align*}
and
\begin{align*}
|\partial_\tau\partial_{\lambda}F_2(0+\lambda \tilde{v}_\theta,0+\tau\psi)|_{\lambda=\tau=0}|
&=\left|\partial_\tau\partial_{\lambda}\int_\Omega r(\mathbf{v}_0+\lambda \tilde{v}_\theta) \Phi(\psi_0+\tau\tilde{\psi})dx|_{\lambda=\tau=0}\right|\\
&=\left|\partial_\tau\int_\Omega r {\tilde{v}}_\theta \Phi(\psi_0+\tau\tilde{\psi})dx|_{\lambda=\tau=0}\right|\\
&=\left|\int_\Omega r {\tilde{v}}_\theta \Phi^{'}(\psi_0)\tilde{\psi} dx\right|
\\
&\lesssim\| {\tilde{v}}_\theta\|_{L^2(\Omega_1)}\|r^2\Phi^{'}(\psi_0)\|_{L^3(\Omega_1)}\left\|\frac{\tilde{\psi}}{r}\right\|_{L^6(\Omega_1)}\\
&\quad+ \| {\tilde{v}}_\theta\|_{L^2(\Omega_2)}\|r^3\Phi^{'}(\psi_0)\|_{L^3(\Omega_2)}\left\|\frac{\tilde{\psi}}{r^2}\right\|_{L^6(\Omega_2)}
\\
&\lesssim\| {\tilde{v}}_\theta\|_{L^2}\|\tilde{\psi}\|_{H^1_{mag}}.
\end{align*}
Thus,  $F_2(\delta v_\theta,\delta\psi)$ is 2-th G$\hat{\text{a}}$teaux differentiable at $(0,0)$.
We will prove that
 \begin{align*}
\partial_{\tau}\partial_{\lambda}F_2((\delta v_\theta)_n+\lambda \tilde{v}_\theta+\tau\widehat{v}_\theta,0)|_{\lambda=\tau=0}
\to  \partial_{\tau}\partial_{\lambda}F_2(0+\lambda \tilde{v}_\theta+\tau\widehat{v}_\theta,0)|_{\lambda=\tau=0},
\end{align*}
 \begin{align*}
\partial_{\tau}\partial_{\lambda}F_2(0,\delta\psi_n+\lambda\tilde{\psi}+\tau\widehat{\psi})|_{\lambda=\tau=0}
\to  \partial_{\tau}\partial_{\lambda}F_2(0,0+\lambda\tilde{\psi}+\tau\widehat{\psi})|_{\lambda=\tau=0},
\end{align*}
 \begin{align*}
\partial_{\tau}\partial_{\lambda}F_2((\delta v_\theta)_n+\lambda \tilde{v}_\theta,\delta\psi_n+\tau\tilde{\psi})|_{\lambda=\tau=0}
\to  \partial_{\tau}\partial_{\lambda}F_2(0+\lambda \tilde{v}_\theta,0+\tau\tilde{\psi})|_{\lambda=\tau=0},
\end{align*}
for $\{(\delta v_\theta)_n\}\in L^2, \{\delta\psi_n\}\in Z$ such that $(\delta v_\theta)_n\rightarrow 0
\in L^2, \delta\psi_n\rightarrow 0$ in $Z$, which implies that
$F_2(\delta v_\theta,\delta\psi)$ is second order Fr\'{e}chet differentiable at $(0,0)$.
In specify, by \eqref{unbounded}-\eqref{magunbounded} and \eqref{order3}-\eqref{order4},
 \begin{align*}
\left|\partial_{\tau}\partial_{\lambda}F_2((\delta v_\theta)_n+\lambda \tilde{v}_\theta+\tau\widehat{v}_\theta,0)|_{\lambda=\tau=0}
-\partial_{\tau}\partial_{\lambda}F_2(0+\lambda \tilde{v}_\theta+\tau\widehat{v}_\theta,0)|_{\lambda=\tau=0}\right|=0,
\end{align*}
 \begin{align*}
&\left|\partial_{\tau}\partial_{\lambda}F_2(0,\delta\psi_n+\lambda\tilde{\psi}+\tau\widehat{\psi})|_{\lambda=\tau=0}
-\partial_{\tau}\partial_{\lambda}F_2(0,0+\lambda\tilde{\psi}+\tau\widehat{\psi})|_{\lambda=\tau=0}\right|
\nonumber\\&=\left|\int_\Omega \left(r\mathbf{v}_0 \Phi^{''}(\delta\psi_n+\psi_0)\tilde{\psi}\widehat{\psi}-r\mathbf{v}_0 \Phi^{''}(\psi_0)\tilde{\psi}\widehat{\psi}\right)dx\right|
\nonumber\\
&\lesssim\|\mathbf{v}_0\|_{L^2(\Omega_1)}\|r^4\Phi^{'''}(\psi_0+\kappa\delta\psi_n)\|_{L^\infty(\Omega_1)}
\left\|\frac{\delta\psi_n}{r}\right\|_{L^6(\Omega_1)}\left\|\frac{\tilde{\psi}}{r}\right\|_{L^6(\Omega_1)}\left\|\frac{\widehat{\psi}}{r}\right\|_{L^6(\Omega_1)}
\nonumber\\
&\quad+ \|\mathbf{v}_0\|_{L^2(\Omega_2)}\|r^7\Phi^{'''}(\psi_0+\kappa\delta\psi_n)\|_{L^\infty(\Omega_2)}\left\|\frac{\delta\psi_n}{r^2}\right\|_{L^6(\Omega_2)}
\left\|\frac{\tilde{\psi}}{r^2}\right\|_{L^6(\Omega_2)}\left\|\frac{\widehat{\psi}}{r^2}\right\|_{L^6(\Omega_2)}\nonumber\\
&\lesssim\|{\mathbf{v}}_0\|_{L^2}\|\delta\psi_n\|_{H^1_{mag}}\|\tilde{\psi}\|_{H^1_{mag}}\|\widehat{\psi}\|_{H^1_{mag}}\rightarrow 0 \text{ as } n\rightarrow\infty
\end{align*}
and
 \begin{align*}
&|\partial_{\tau}\partial_{\lambda}F_2((\delta v_\theta)_n+\lambda \tilde{v}_\theta,\delta\psi_n+\tau\tilde{\psi})|_{\lambda=\tau=0}
-\partial_{\tau}\partial_{\lambda}F_2(0+\lambda \tilde{v}_\theta,0+\tau\tilde{\psi})|_{\lambda=\tau=0}|\nonumber\\
&=\left|\int_\Omega (r{\tilde{v}}_\theta \Phi^{'}(\delta\psi_n+\psi_0)\tilde{\psi}-r{\tilde{v}}_\theta \Phi^{'}(\psi_0)\tilde{\psi}) dx\right|
\nonumber\\
&\lesssim\|{\tilde{v}}_\theta\|_{L^2(\Omega_1)}\|r^3\Phi^{''}(\psi_0+\kappa\delta\psi_n)\|_{L^6(\Omega_1)}\left\|\frac{\delta\psi_n}{r}\right\|_{L^6(\Omega_1)}
\left\|\frac{\tilde{\psi}}{r}\right\|_{L^6(\Omega_1)}
\nonumber\\
&\quad+ \|{\tilde{v}}_\theta\|_{L^2(\Omega_2)}\|r^5\Phi^{''}(\psi_0+\kappa\delta\psi_n)\|_{L^6(\Omega_2)}\left\|\frac{\delta\psi_n}{r^2}\right\|_{L^6(\Omega_2)}
\left\|\frac{\tilde{\psi}}{r^2}\right\|_{L^6(\Omega_2)}\nonumber\\
&\lesssim\|{\tilde{v}}_\theta\|_{L^2}\|\delta\psi_n\|_{H^1_{mag}}\|\tilde{\psi}\|_{H^1_{mag}}\rightarrow 0 \text{ as } n\rightarrow\infty,
\end{align*}
with $0<\kappa<1.$

\end{proof}

Define the distance function as in \eqref{distance}, i.e.,
\begin{align*}
d(t)&=d_1(t)+d_2(t)+d_3(t),
\end{align*}
where
\begin{align*}
d_1(t)&=d_1(\delta v_r,\delta v_z,\delta H_\theta)=\frac{1}{2}\int_{\Omega}(|\delta v_r|^2+|\delta v_z|^2+|\delta H_\theta|^2)dx,\\
d_2(t)&=d_2(\delta v_\theta)=\frac{1}{2}\int_{\Omega}|\delta v_\theta|^2dx,\\
d_3(t)&=d_3(\delta\psi)=\frac{1}{2}\int_{\Omega}\left(\frac{1}{r^2}|\partial_z\delta\psi|^2+\frac{1}{r^2}|\partial_r\delta\psi|^2\right)
 dx.
\end{align*}
\begin{lemma}\label{unboundedEcEleq}
 There exists a constant $C>0$, such that
\begin{align*}
E_c(\delta v,\delta H_\theta,\delta\psi)\leq Cd +o(d).
\end{align*}
\end{lemma}
\begin{proof}
From \eqref{EcE01} and ${H^1_{mag}}\hookrightarrow\hookrightarrow L^2_{|\frac{\partial_r(\omega^2)}{b^2r}|}$, we can obtain
\begin{align*}
E_c(\delta v,\delta H_\theta,\delta \psi)
&=\frac{1}{2}\int_{\Omega}(|\delta v_r|^2+|\delta v_z|^2+|\delta H_\theta|^2)dx+\frac{1}{2}\int_{\Omega}\left(\frac{1}{r^2}|\partial_z\delta\psi|^2+\frac{1}{r^2}|\partial_r\delta\psi|^2\right)dx
\notag\\&\quad+\frac{1}{2}\int_{\Omega}\left(\frac{\partial_r(\omega^2)}{b^2r} |\delta\psi|^2+\left|\delta v_\theta+\frac{\partial_r\omega}{ b}\delta\psi\right|^2\right)
dx\notag \\&\quad+ o(\|\delta\psi\|^2_{H^1_{mag}})+o(\|\delta v_\theta\|^2_{L^2})\leq Cd +o(d),
\end{align*}
with the fact that
\begin{align*}
&\frac{1}{2}\int_{\Omega}\frac{(\partial_r\omega)^2}{b^2}|\delta\psi|^2dx
+\int_{\Omega}\delta v_\theta\frac{\partial_r(\omega)}{b}\delta\psi dx
+\frac{1}{2}\int_{\Omega}\frac{\partial_r(\omega^2)}{b^2r}|\delta\psi|^2dx
\\&\leq C\left[\|\delta v_\theta\|_{L^2}^2+\|\delta\psi\|_{H^1_{mag}}^2
\right]
\leq C(d_2+d_3).
\end{align*}
\end{proof}

Then, for the proof of Lemma \ref{unboundedEcEgeq}, the key point is to verify $J_i(\delta\psi)\in C^2(H^1_{mag})$. The other steps can be done by the same arguments as Lemma \ref{EcEgeq}. In order to prove $J_i(\delta\psi)\in C^2(H^1_{mag})$, similar to the bounded case as before in Section \ref{nonlinearstability}, by solving the eigenvalue problem
\begin{align*}
\mathbb{L}_0 h+(\lambda-1)\mathfrak{F}(r)h=-\frac{1}{r}\partial_r\left(\frac{1}{r}\partial_r h \right)+\lambda\mathfrak{F}(r)h=0
\end{align*}
for $h\in Z$. Note that the case $\ker(\mathbb{L}_0)= \{0\}$ can be proved by the same arguments as the case $\ker(\mathbb{L}_0)\neq \{0\}$. Here, we consider the case $\ker(\mathbb{L}_0)\neq \{0\}$. Thus, we obtain $K$ negative directions of $\langle \mathbb{L}_0\cdot,\cdot\rangle$ denoted by $\{h_1,...,h_{K}\}\in C^2(0,\infty)$ and
the kernel direction of $\langle \mathbb{L}_0\cdot,\cdot\rangle$ by $h_{0}\in C^2(0,\infty)$.
Moreover,
\begin{align}\label{derivativeestimate}
\lim_{r\rightarrow\infty}h(r)=0.
\end{align}
Then, we can choose $K+1$ invariants $$J_{i}(\delta\psi)=\int_\Omega (f_i(\delta\psi+\psi_0)-f_i(\psi_0))dx,$$ where $f^{'}(\psi_0)$ is defined by \eqref{derivateofQ}, i.e., \begin{align*}
f'_i(\psi_0)=
 h_i\left(r\right)\mathfrak{F}(r),\quad \text{for } 0\leq i\leq K.
\end{align*}
Together with $\mathfrak{F}(r)=\frac{\partial_r(\omega^2)}{ b^2r} $, $\omega(r)=O(r^{-(1+\alpha)}) (r\rightarrow\infty)$ for $\alpha>1$, $\partial_r\omega(r)=O(r^{\beta-3}) (r\rightarrow0)$ for $\beta\geq4$,  and \eqref{derivativeestimate}, it follows that
\begin{align}
&|r^2f_i^{'}(\psi_0)|=o(r^{-2-2\alpha}) (r\rightarrow\infty),\quad |rf^{'}(\psi_0)|=O(r^{\beta-3}) (r\rightarrow0);
\label{order8}\\
&|r^4f_i^{''}(\psi_0)|=o(r^{-2-2\alpha}) (r\rightarrow\infty),\quad |r^2f^{''}(\psi_0)|=O(r^{\beta-3}) (r\rightarrow0);
\label{order9}\\
&|r^6f_i^{'''}(\psi_0)|=o(r^{-2-2\alpha}) (r\rightarrow\infty),\quad |r^3f^{'''}(\psi_0)|=O(r^{\beta-3}) (r\rightarrow0)
\label{order10}
\end{align}
for $0\leq i\leq K.$

Applying Lemma \ref{ubdFderivate} 1), it is clear that $J_i(\delta\psi)\in C^2(H^1_{mag})$ through replacing \eqref{order5}-\eqref{order7} for $\alpha>0 $ by \eqref{order8}-\eqref{order10} for $\alpha>0.$ Therefore, we can prove the following lemma by the same arguments in Lemma \ref{EcEgeq}.
\begin{lemma}\label{unboundedEcEgeq}
There exists a constant $\tau>0$, such that
\begin{align*}
E_c(v,H_\theta,\psi)-E_c(v_0,0,\psi_0)\geq  \tau d-o(d)- O( d(0)).
\end{align*}
\end{lemma}

In conclusion,
we finish the proof of Theorem \ref{nonlinearstable} when $R_2=\infty$.

\textbf{Acknowledgments:} This work is supported partly by the NSF grants
DMS-1715201 and DMS-2007457 (Lin).

\textbf{Data Availability Statement:} This manuscript has no associated data.


\end{CJK*}

\end{document}